\documentclass[regno]{amsart}

\RequirePackage{amsthm,amsmath,amsfonts,amssymb}
\RequirePackage{mathrsfs}
\RequirePackage{graphicx}
\RequirePackage{enumitem}
\RequirePackage{tikz}
\RequirePackage[numbers]{natbib}
\RequirePackage[colorlinks,citecolor=blue,urlcolor=blue,linkcolor=blue]{hyperref}
\RequirePackage{comment}
\usepackage[normalem]{ulem}
\usepackage{cancel}
\usepackage[lmargin=2.8cm,rmargin=2.8cm,bmargin=2cm,tmargin=2cm]{geometry}

\DeclareFontFamily{U}{mathx}{\hyphenchar\font45}
\DeclareFontShape{U}{mathx}{m}{n}{%
      <5> <6> <7> <8> <9> <10>
      <10.95> <12> <14.4> <17.28> <20.74> <24.88>
      mathx10}{}
\DeclareSymbolFont{mathx}{U}{mathx}{m}{n}
\DeclareMathAccent{\widecheck}{0}{mathx}{"71}

\DeclareMathAlphabet{\mathbbold}{U}{bbold}{m}{n}
\newcommand{\ind}{\mathbbold{1}}

\newcommand{\nto}{\mathrel{\ooalign{%
  $\to$\cr
  \hidewidth\scalebox{0.8}{$\diagup$}\hidewidth\cr}}}

\newcommand{\Zu}{Z^{(1)}_t}
\newcommand{\Zus}{Z^{(1)}}
\newcommand{\zl}{z^{(1)}_{\lim}}
\newcommand{\zu}{z^{(1)}_t}
\newcommand{\zuz}{z^{(1)}_0}

\newcommand{\Zd}{ Z^{(2)}_t }
\newcommand{\Zds}{ Z^{(2)}}
\newcommand{\zd}{ z^{(2)}_t }
\newcommand{\zdz}{ z^{(2)}_0 }

\newcommand{\Prob}{\mathbb{P}}
\newcommand{\Esp}{\mathbb{E}}
\newcommand{\E}{\mathbb{E}}

\newcommand{\R}{\mathbb{R}}
\newcommand{\N}{\mathbb{N}}

\renewcommand{\P}{\mathbb{P}}
\newcommand{\eps}{\varepsilon}
\newcommand{\var}{\mathbb{V}}
\newcommand{\EXCLref}[1]{\hyperref[excli]{(\mathrm{EXCL}{#1})}}
\newcommand{\ADVref}[1]{\hyperref[ADVi]{(\mathrm{ADV}{#1})}}
\newcommand{\exclref}[1]{\hyperref[exclidet]{(\mathrm{excl}{#1})}}

\theoremstyle{plain}
\newtheorem{theo}{Theorem}[section]
\newtheorem{coro}[theo]{Corollary}
\newtheorem{lemma}[theo]{Lemma}
\newtheorem{prop}[theo]{Proposition}

\theoremstyle{definition}
\newtheorem{defi}[theo]{Definition}
\newtheorem{rem}[theo]{Remark}
\newtheorem{assu}[theo]{Assumption}

\begin{document}

\title[Lotka--Volterra systems with random catastrophes]{Competitive Lotka--Volterra systems with random catastrophes: invasion fitnesses and reversal of the competitive hierarchy}

\date{}
\author{
	Eugène \textsc{Ferragu}
		%\thanks{First footnote}
		\and
	Aline \textsc{Marguet}
		%\thanks{Second footnote}
		\and
	Charline \textsc{Smadi}
}
\def\runauthor{
	Eugène \textsc{Ferragu}, Aline \textsc{Marguet}, Charline \textsc{Smadi}
}

\date{  \today}
    \address[E. \textsc{Ferragu}, A. \textsc{Marguet}]{Univ.\ Grenoble Alpes, Inria, 38000 Grenoble, France, Univ. Grenoble Alpes, CNRS, LIPhy, Grenoble, France.}
    \email{eugene.ferragu@inria.fr, aline.marguet@inria.fr}
    \address[C. \textsc{Smadi}]{Univ.\ Grenoble Alpes, INRAE, LESSEM, 38000 Grenoble, France.}
    \email{charline.smadi@univ-grenoble-alpes.fr}

%\footnotetext[1]{Univ.\ Grenoble Alpes, Inria, 38000 Grenoble, France, Univ. Grenoble Alpes, CNRS, LIPhy, Grenoble, France\\
%E-mail: eugene.ferragu@inria.fr, aline.marguet@inria.fr}

%\footnotetext[2]{Univ.\ Grenoble Alpes, INRAE, LESSEM, 38000 Grenoble, France\\
%E-mail: charline.smadi@univ-grenoble-alpes.fr}

\begin{abstract}
We study a two-type individual-based population with logistic competition subject
to random catastrophes: at the epochs of a Poisson process, each individual
survives independently with a random, type-dependent probability. In the
large-population limit, the process converges to a competitive Lotka--Volterra
system perturbed by multiplicative jumps. We show that the long-time behaviour of this system -- exclusion of one
type or the other, or coexistence -- is governed by two averaged invasion
fitnesses, obtained from the classical ones by replacing each growth rate $a_i$ with
its average $\bar a_i=a_i+r\E[\ln\Theta^{(i)}]$, where $r$ is the catastrophe rate
and $\Theta^{(i)}$ is the survival probability of an individual of type $i$ when a
catastrophe occurs. A consequence is that catastrophes can reverse the competitive
hierarchy: a type doomed to exclusion in the deterministic system may persist or
even exclude its competitor, when catastrophes are frequent enough, even if each
catastrophe is individually more harmful to it. In the finite population setting, we also study the invasion of a mutant on the time scale $\ln K$, where $K$ is a scaling parameter,  through a coupling with branching processes in the varying random environment generated by the resident population.

\end{abstract}



\maketitle

\section{Introduction}

Understanding the fate of competing types, and, in particular, of newly arising mutants, is a central question in ecology and evolution. Individual demographic rates depend both on the biotic environment -- the other individuals through competition, predation, or other ecological interactions \cite{qian2020balance} -- and on abiotic conditions: temperature, resources, seasonality, storms, droughts, fires, or dilution protocols, whose fluctuations need not occur on the time scale of individual births and deaths \cite{Levin1992,VasseurYodzis2004}.

Our main result is that random catastrophes do not destroy the classical invasion-fitness structure of competitive Lotka--Volterra systems: they act on it through a single substitution. For $i\in \{1,2\}$, each intrinsic growth rate $a_i$ has to be replaced by an average growth rate taking into account the catastrophes acting multiplicativly on the population size $\bar a_i=a_i+r\E[\ln\Theta^{(i)}]$, where $r$ is the rate of catastrophes and $\Theta^{(i)}$ the probability that a type-$i$ individual survives one of them. The whole trichotomy -- exclusion of type $1$, exclusion of type $2$, coexistence -- is then read off the signs of the resulting fitnesses. The competition thresholds themselves do not depend on $r$, so that increasing the catastrophe rate simply moves the system along a single axis and may take it through the three regimes in turn. This yields an explicit reversal of the competitive hierarchy: one can find parameters such that every single catastrophe is more severe for type $1$ than for type $2$ and type $1$ is excluded for $r=0$, coexists with type $2$ for $r=1$, and excludes type $2$ for $r=1.8$.

The effect of abiotic variability is not unidirectional: environmental fluctuations may increase extinction risk \cite{LewontinCohen1969, lande1993risks}, or on the contrary maintain diversity by letting rare types recover during favourable periods \cite{hutchinson1961paradox, chesson1982stabilizing, GillespieTurelli1989} -- a question made pressing by climate change and by the increasing frequency of extreme events \cite{walther2002ecological,wernberg2024impacts}.

For our purposes, the relevant lesson from stochastic persistence theory is that the long-term fate of a type is often encoded in its growth rate when rare. One studies the dynamics generated by the resident community and by the fluctuating environment, and then asks whether a small population of the missing type has positive or negative average growth in this regime. This is the role played by realized per-capita growth rates, or invasion rates, in stochastic persistence results \cite{chesson1989invasibility, HeningNguyen2018, BenaimSchreiber2019, HeningNguyenChesson2021}.

These results are mostly formulated either for stochastic difference equations, where the population is updated by random maps from one generation or census time to the next \cite{chesson1989invasibility, BenaimSchreiber2019}, or for stochastic differential equations and randomly switched vector fields \cite{HeningNguyen2018, benaim2016lotka}. In both cases, the size of the population follows a continuous trajectory, which leads to techniques differing substantially from those needed for piecewise deterministic Markov processes with jumps \cite{Davis1984, BenaimLeBorgneMalrieuZitt2015}.

In this work, we study a two-type individual-based population model with logistic competition and random catastrophes. Between catastrophes, individuals reproduce and die according to type-dependent birth and death rates, and deaths include density-dependent competitive interactions. This part of the model belongs to the family of logistic birth-and-death processes introduced by Lambert \cite{lambert2005branching} and widely used in stochastic models of adaptive dynamics and population processes \cite{champagnat2006microscopic, bansaye2010branching, meleard2012quasi}. Catastrophes occur independently of the population state, according to an external Poisson process. At each catastrophe, a random fraction of individuals of each type is killed: conditionally on the catastrophe intensity $(\Theta^{(1)},\Theta^{(2)})$, each individual of type $i$ survives independently with probability $\Theta^{(i)}$.

The carrying capacity is represented by a scaling parameter $K$, and we study the large-$K$ limit of the rescaled process, as in \cite{fournier2004microscopic}. Demographic fluctuations then vanish between catastrophes, but the catastrophes remain visible at the macroscopic scale and produce multiplicative jumps in the densities: the limit is a competitive Lotka--Volterra system with discontinuous trajectories.

This discontinuity distinguishes our model from most Lotka--Volterra systems in time-dependent or randomly changing environments. In periodically varying or Markov-switching systems, and in adaptive dynamics models with piecewise constant environments motivated by seasons or drug treatments, the environment changes the
vector field but the densities themselves do not jump \cite{zhu2009competitive, benaim2016lotka, MalrieuZitt2017, MalrieuPhu2016, esser2025crossing}. Here the environmental event acts directly on the population size: a catastrophe instantaneously removes a macroscopic fraction of each type, which calls for arguments adapted to flows between jumps (multiplicative shocks) and large-population approximations of the underlying individual-based process. There is an important literature on discontinuous population models with catastrophes, going back to \cite{BrockwellGaniResnick1982}. Continuous-state branching processes or continuous-state nonlinear branching processes with multiplicative catastrophic jumps were studied in \cite{bansaye2013extinction,PalauPardoSmadi2016,marguet2021long}, where negative jumps in a L\'evy environment model dramatic events killing a fraction of the population. Branching processes in L\'evy random environments, including models with competition, were further developed in \cite{palau2018branching}.
At the discrete-time level, the fate of a single lineage in such a fluctuating environment is described by the theory of branching processes in varying and random environments, going back to Smith and Wilkinson \cite{smith1969branching} and Athreya and Karlin \cite{athreya1971branching}, and considerably refined since, in particular in the critical regime \cite{afanasyev2005criticality,KerstingVatutin2017}; we rely on the recent unifying approach of Kersting \cite{Kersting19}, which is well suited to the environment generated by our resident population. Catastrophe-like mechanisms also appear naturally in experimental-evolution models with repeated bottlenecks. For instance, \cite{GonzalezCasanovaKurtWakolbingerYuan2016} introduced an individual-based model for the Lenski experiment in which continuous intraday growth alternates with interday sampling, showing how a slower experimental protocol can reshape the effective relative fitness dynamics.

Our contribution lies at the intersection of these viewpoints: an individual-based continuous-time competitive two-type model, macroscopic catastrophic jumps, and a complete Lotka-Volterra classification in terms of averaged invasion fitnesses.

Let us be more precise. In the two-dimensional deterministic competitive Lotka--Volterra system without catastrophes, the relevant quantities are the so-called invasion fitnesses
\begin{equation}
\label{fitinv}    f_{ij} =  a_i-\frac{\beta_{ij}}{\beta_{jj}}a_j,
\end{equation}
where $a_i$ is the intrinsic growth rate of type $i$ and the $\beta$'s are the competition parameters. The sign of $f_{ij}$ gives the growth rate of a rare type $i$ introduced into a resident population of type $j$ at equilibrium. In the presence of catastrophes, the same classification remains valid after replacing $a_i$ by a growth rate that takes into account the average of the logarithm of the effect of the random catastrophes
\begin{equation*}
  \bar a_i    =     a_i+r\E[\ln \Theta^{(i)}],
\end{equation*}
where $r$ is the rate of arrival of catastrophic jumps. The relevant invasion fitnesses become
\begin{equation} \label{def_new_fitnesses}    \mathfrak f_{ij}   =    \bar a_i-\frac{\beta_{ij}}{\beta_{jj}}\bar a_j.
\end{equation}
Thus, all the effect of the catastrophic environment is encoded in the averaged fitnesses $\mathfrak f_{ij}$. The perturbation is random and discontinuous, but the qualitative classification is still read from the signs of two quantities.

Equivalently, the trichotomy only depends on the position of the ratio $\bar a_1/\bar a_2$ with respect to the two competition thresholds $\beta_{12}/\beta_{22}$ and $\beta_{11}/\beta_{21}$, which are not affected by the catastrophes. The environment does not move the thresholds; it moves the position of the ratio.

It is worth stressing that only the logarithmic average is the right one. The infinitesimal, arithmetic growth rate of a rare type is strictly larger than $\mathfrak f_{21}$, the gap being $r(\E[\Theta^{(2)}]-1-\E[\ln\Theta^{(2)}])>0$ by strict concavity of the logarithm; the two notions, which coincide in the switching environments of \cite{benaim2016lotka}, differ as soon as the environment acts multiplicatively on the population size.

The logarithmic averaging appearing in the definition of $\bar a_i$ has a biological meaning. A type that is less affected by a typical catastrophe is not necessarily favoured in the long run: what matters is the cumulative multiplicative effect of catastrophes, measured by $\E[\ln\Theta^{(i)}]$, relative to the intrinsic growth rates and to the competitive coefficients. This is analogous in spirit to the paradoxical effects observed in switching environments \cite{benaim2016lotka}, but here the mechanism comes from discontinuous multiplicative shocks rather than from the alternation between continuous vector fields.

The paper is organised as follows. Section \ref{Model} introduces the two-type birth-and-death process with competition and catastrophes. Section \ref{sec_conv} establishes the convergence of the rescaled process, as $K\to\infty$, towards the Lotka--Volterra system with multiplicative jumps, and identifies the averaged growth rates $\bar a_i$. Section \ref{sec_longtime_Z} classifies the long-time behaviour of the limiting system and computes the limiting time averages of the population sizes. Section \ref{subsec_reversion} studies the reversal of the hierarchy of traits and compares our results with randomly switched Lotka--Volterra systems. Section \ref{newcomer_section} returns to the individual-based process and describes the invasion of a mutant, by coupling with branching processes in the random environment generated by the resident.

\section{Model}
\label{Model}
We consider a population constituted of two types of individuals, $1$ and $2$, living in an environment that faces regular catastrophes, killing each time a (random) macroscopic number of individuals of both types. For $i,j \in \{1,2\}^2$, let $\lambda_i$ (resp. $\mu_i$) denote the natural birth (resp. death) rate of type $i$ individuals, and let $\beta_{ij}$ denote the competitive pressure exerted by an individual of type $j$ on an individual of type $i$. The competition will be rescaled by a parameter $K$ which quantifies the amount of resources available and can be seen from a biological point of view as the carrying capacity of the environment. Random catastrophes occur according to a compound Poisson process with rate $r$. Their intensity $(\theta^{(1)},\theta^{(2)})$ is drawn according to a distribution $\kappa$ on $(0,1)^2$ and has to be understood as follows: each individual of type $i \in \{1,2\}$, independently of other individuals, is killed by the catastrophe with a probability $1-\theta^{(i)}$.

More formally, we write $(\Gamma^{(K)}_t=(\Gamma^{(1,K)}_t,\Gamma^{(2,K)}_t),t\geq0)$ the continuous-time Markov process on $\N^2$ giving the sizes of populations $1$ and $2$ and denote by $(e_1,e_2)$ the canonical basis of $\R^2$. The jump rates of $\Gamma^{(K)}$ are:
\begin{itemize}
    \item Birth event: For $i\in \{1,2\}$ :
    $\Gamma^{(K)}_t\longrightarrow \Gamma^{(K)}_t+e_i$ at rate $\lambda_i\Gamma^{(i,K)}_t$.
    \item Death event: For $i\in \{1,2\}$ :
    $\Gamma^{(K)}_t\longrightarrow \Gamma^{(K)}_t-e_i$ at rate 
    \begin{equation*} \label{def_nui}\Big(\mu_i+\frac{\beta_{i1}}{K}\Gamma^{(1,K)}_t+\frac{\beta_{i2}}{K}\Gamma^{(2,K)}_t\Big)\Gamma^{(i,K)}_t=:\nu_t^i\Gamma^{(i,K)}_t.\end{equation*}
    \item Catastrophes: $\Gamma^{(K)}_{t-}\longrightarrow(\text{Bin}(\Gamma^{(1,K)}_{t-},\Theta^{(1)}),\text{Bin}(\Gamma^{(2,K)}_{t-},\Theta^{(2)}))$ at rate $r$, where the two binomials are independent and $(\Theta^{(1)},\Theta^{(2)})$ has law $\kappa$. All catastrophes are independent.
\end{itemize}
Equivalently, we can 
define the Markov process $\Gamma^{(K)}$ as the strong solution of a stochastic differential equation (SDE). For $i\in\{1,2\}$ 
\begin{equation}
\label{strongSDE}
\begin{aligned}
\Gamma_t^{(i,K)}
=&\Gamma_0^{(i,K)}+\int_0^t
\left(\int_0^{\lambda_i\Gamma^{(i,K)}_{s-}}N_i^+(ds,dx)-\int_0^{\nu_{s-}^i\Gamma^{(i,K)}_{s-}}N_i^-(ds,dx)\right)\\
&+\int_0^t\int_0^1\int_0^1\Big(F^{-1}_{\text{Bin}(\Gamma^{(i,K)}_{s-},\theta^{(i)})}(p) - \Gamma^{(i,K)}_{s-}\Big)\bar{N}_i(ds,d\theta^{(i)},dp),
\end{aligned}
\end{equation}
where $N_1^+,N_1^-,N_2^+,N_2^-$ are independent Poisson point processes on $\R^+\times \R$ with intensity $dsdx$, $\bar{N}_1$, $\bar{N}_2$ are two projections of the same Poisson point process $\bar{N}$ on $\R^+\times[0,1]^2\times [0,1]^2$ of intensity $rds\kappa(d\theta^{(1)},d\theta^{(2)})dp^{(1)}dp^{(2)}$, and $F^{-1}_{\text{Bin}(n,p)}$ is the generalized inverse of the cumulative distribution function of a binomial law of parameters $(n,p)$.

Our work heavily relies on the analysis of a piecewise-deterministic Markov process (PDMP), which is the limit of the rescaled birth and death process with catastrophes 
\begin{equation} \label{defZK} Z^{(K)}:={\Gamma^{(K)}}/{K},\end{equation}
when $K$ is large (see Lemma \ref{lem_Kurtz}) and which satisfies a system of competitive Lotka-Volterra equations, perturbed by the catastrophes. The non-perturbed system, denoted by $\mathbf{z}_t=(z^{(1)}_t,z^{(2)}_t)$, is the solution to:
    \begin{equation} \label{def_LV_det} \left\{
    \begin{array}{ll}
        \frac{dz^{(1)}_t}{dt} = z^{(1)}_t\left(a_1 - \beta_{11}z^{(1)}_t - \beta_{12}z^{(2)}_t\right)=:M_1(z^{(1)}_t,z^{(2)}_t),\\
        \frac{dz^{(2)}_t}{dt} = z^{(2)}_t\left(a_2 - \beta_{21}z^{(1)}_t - \beta_{22}z^{(2)}_t\right)=:M_2(z^{(1)}_t,z^{(2)}_t),
    \end{array}    
        \right.\end{equation}
where $a_i:=\lambda_i-\mu_i$ for $i\in \{1,2\}$.
The key quantities to characterize the long-time behavior of the solution to \eqref{def_LV_det} are the invasion fitnesses $f_{ij}$ introduced in \eqref{fitinv}. We refer the reader to Appendix \ref{secLVdet} for the classical properties of System \eqref{def_LV_det} which will be needed in the proofs.

The intensity of catastrophes results from summing of the order of $K$ Bernoulli random variables encoding the survival (or not) of each individual. Using that for $x \in \R_+^*$, $\theta \in (0,1)$ a variable of law $\text{Bin}(Kx,\theta)/K$ converges in probability to $\theta x$ (see Lemma \ref{lem_Kurtz} for a rigorous result), we obtain that if $Z^{(i,K)}_0\to z_0^{(i)}$ when $K\to\infty$, the limit of the process $Z^{(K)}$ (denoted $Z=(Z^{(1)}, Z^{(2)})$) is the solution to the following system of SDE: for $i\in\{1,2\}$ and $t\geq 0$,
\begin{equation}
\label{eqdep}
Z^{(i)}_t = z_0^{(i)}+ \int_0^tZ^{(i)}_s\left(a_i - \beta_{i1}Z^{(1)}_s - \beta_{i2}Z^{(2)}_s\right)ds + \int_0^t \int_0^1 (\theta^{(i)}-1) Z^{(i)}_{s-}S_i(ds,d\theta^{(i)}),
\end{equation}
where for $i \in \{1,2\}$, $S_i(ds,d\theta^{(i)})=\int_0^1 \bar{N}_i(ds,d\theta^{(i)},dp). $

\section{Results}

\subsection{Convergence of the discrete process and averaged invasion fitnesses} 

\label{sec_conv}
To analyse the long-time behaviour of the rescaled birth-death process with catastrophes $Z^{(K)}$ defined in \eqref{defZK}, we will rely on the study of the strong solution $Z$ to \eqref{eqdep}, as the two processes are close to each other for large $K$. This is the content of  Lemma \ref{lem_Kurtz} and Proposition \ref{prop_cvglongue} which will be proven in Section \ref{sec_proof_conv}. The first result is an adaptation of Theorem 2.1 p.456 in \cite{ethier2009markov}, which holds only on a time interval whose size stays bounded as $K$ goes to infinity.

\begin{lemma}\label{lem_Kurtz}
    For $(i,j) \in \{1,2\}^2$, let us fix $\lambda_i,\ \mu_i, \ \beta_{ij} \in \R_+, \ r,\ \beta_{ii} \in \R_+^*$ and a law $\kappa$ on $(0,1)^2$. Take $T<\infty$. If  $Z^{(K)}_0$ converges to $Z_0=(z_0^{(1)},z_0^{(2)})$ in probability, then
    \begin{equation} \label{lim_Kurtz}
        \lim_{K \to \infty} \sup_{s \leq T} \big\| Z^{(K)}_s - Z_s \big\| = 0 \quad \text{in probability,}
    \end{equation}
    where $\|.\|$ denotes the $L^1$-norm on $\R^2$.
\end{lemma}

In order to control the invasion of a mutant, as its dynamics will depend on the dynamics of the resident population, we will need to control the deviation of this discrete resident population $Z^{(K)}$ compared to its approximation $Z$ during a time of order $ \ln K$, which is the time it takes for a mutant population to reach a size of order $K$ (see \cite{champagnat2006microscopic}). This is the aim of the next proposition. The following assumption on the effect of the catastrophes on the resident population will be needed:
\begin{equation} \label{H3} \tag{H}
\E\big[(\Theta^{(1)})^{-p_0}\big]<\infty \ \text{ for some } p_0\in(0,1].
\end{equation}
In particular, we need to ensure that catastrophes will not drive the resident population too close to $0$, where the stochasticity is too high.
\begin{prop}
\label{prop_cvglongue}
Let $\alpha>1$ and $\varepsilon>0$. Assume that there exists $0<p_0 \leq 1$ such that
\[  Z_0^{(1)}>0, \quad 
\Gamma^{(1,K)}(0) = \lfloor Z_0^{(1)} K \rfloor,
\quad
\Gamma^{(2,K)}(0)=Z_0^{(2)}=0\quad\text{and} \quad \eqref{H3} \ \text{holds} .
\]
Then, 
$$ \lim_{K \to \infty} \P\Big( \sup_{t \leq (\ln K)^{\alpha-\eps}}|Z_t^{(1)}-Z^{(1,K)}_t|^2 > K^{3\eps-1} \Big)=0. $$
\end{prop}

For $i\in \{1,2\}$, let $K^{(i)}$ be the auxiliary L\'evy process defined by
\begin{equation}
\label{defK}    
K_t=(K^{(1)}_t,K^{(2)}_t):=(a_1t,a_2t) + \int_0^t \int_{(0,1)^2} (\ln\theta^{(1)},\ln\theta^{(2)})S(ds, (d\theta^{(1)},d\theta^{(2)})),
\end{equation}
where $S(ds,(d\theta^{(1)},d\theta^{(2)}))=\int_0^1 \int_0^1 \bar{N}(ds,d\theta^{(1)},d\theta^{(2)},dp^{(1)},dp^{(2)})$. We introduce the parameters:
\begin{align} \label{def_overline_ai}
    \overline{a}_i&:=a_i + r\Esp[\ln \Theta^{(i)}], \quad \text{and}\quad q_i:=-\Esp[\ln \Theta^{(i)}]>0,\quad  i \in \{1,2\}.
\end{align}

Note that  $\E[K_1^{(i)}]= \overline{a}_i.$ 
In particular, if $\overline{a}_i<0$, as $K^{(i)}$ is a Lévy process, $K_t^{(i)}$ drifts to $-\infty$ when $t$ goes to infinity (see Theorem 7.2 in \cite{kyprianou2006introductory}). Moreover, an application of Itô's formula to \eqref{eqdep} and \eqref{defK} shows that
\begin{equation} \label{Ito}  d \big(Z_t^{(i)}e^{-K_t^{(i)}}\big) = -Z_t^{(i)}e^{-K_t^{(i)}}\big(  \beta_{i1}Z^{(1)}_t + \beta_{i2}Z^{(2)}_t \big)dt , \end{equation}
which implies that $\big(Z_t^{(i)}e^{-K_t^{(i)}}, t \geq 0\big)$ is non-increasing. Hence, $\overline{a}_i < 0$ implies that  $e^{K_t^{(i)}}$ goes to $0$ and thus that $Z_t^{(i)}$ converges to $0$ as $t$ goes to infinity, which is not of interest in this work.
Therefore, in the sequel, we consider that the following assumption holds.
\begin{assu}\label{ass:a_bar_positive}
$\overline{a}_1>0, \quad \overline{a}_2>0.$
\end{assu}

\begin{rem} \label{rem_moments}
Assumption \ref{ass:a_bar_positive} implies $\E\big[|\ln\Theta^{(i)}|\big]<\infty$ for $i\in\{1,2\}$ so that $q_i<\infty$. 
\end{rem}

Recall the definition of the averaged invasion fitnesses in \eqref{def_new_fitnesses}.
In the case without jumps ($r=0$, $\overline{a}_i=a_i$), $\mathfrak{f}_{ij}=f_{ij}$ where $f_{ij}$, defined in \eqref{fitinv}, corresponds to the invasion fitness of a mutant $i$ in a population $j$ at equilibrium size. As detailed in Appendix \ref{secLVdet}, this is the Malthusian growth rate of a small population of type $i$ mixed with a population of type $j$ at equilibrium (of size $a_j/\beta_{jj}$). In the case with jumps ($r \neq 0$, $\overline{a}_i<a_i$), $\mathfrak{f}_{ij}$ is the Malthusian growth rate of a small population of type $i$ mixed with a population whose size is the time average of the size of a monomorphic population of type $j$ (see Proposition 4 in \cite{palau2018branching}). We will see that the values of $(\mathfrak{f}_{ij}, i \neq j \in \{1,2\})$  govern the asymptotics of the system with jumps in a similar manner to the case without jump.

\subsection{Long term behaviour of the limiting system} \label{sec_longtime_Z}
As the study of the long-time behaviour of the system with two types requires results on the long-time behaviour of the system with one type, we first recall what happens in the monomorphic case.
\label{SDElongtime}
\subsubsection{One population}

This case has been studied in \cite{palau2018branching}. We give here some of their results for the sake of completeness.
Recall the definition of $(K_t,t\geq 0)$ in \eqref{defK}.

\begin{prop}[Proposition 4 in \cite{palau2018branching}]
\label{prop_monomorph}
Let $(Z_t, t \geq 0)$ be the unique strong solution of \eqref{eqdep} with initial condition $Z_0=(z_0,0)$. Then, it satisfies $Z^{(2)} \equiv 0$ and for $t\geq 0$
\begin{equation} \label{Z_explicite} 
Z^{(1)}_t = \frac{z_0 e^{K^{(1)}_t}}
{1 + \beta_{11} z_0 \int_0^t e^{K^{(1)}_s} ds}.\end{equation}
Moreover, if $z_0 > 0$ then, $Z^{(1)}_t > 0$ for all
$t \geq 0$ a.s. and if $\overline{a}_1>0$, the process $K^{(1)}$ drifts to $+\infty$ and $(Z^{(1)}_t ,t\geq 0)$ has a unique stationary distribution concentrated on $(0,\infty)$.
\end{prop}

\begin{rem} \label{rem_lossofmemory}
In fact the monomorphic process forgets its initial condition at an explicit rate. Let $z_0,z_0'>0$ and let $Z^{(1)}$ and $Z'^{(1)}$ be the solutions to \eqref{eqdep} started at $(z_0,0)$ and $(z_0',0)$ and driven by the same catastrophes. From \eqref{Z_explicite}, we have
\begin{equation} \label{loss_of_memory}
\frac{1}{Z^{(1)}_t}-\frac{1}{Z'^{(1)}_t}
=e^{-K^{(1)}_t}\Big(\frac{1}{z_0}-\frac{1}{z_0'}\Big),\qquad t\geq0 .
\end{equation}
Both processes being bounded by $c:=(a_1/\beta_{11})\vee z_0\vee z_0'$, \eqref{loss_of_memory} gives
$$\sup_{t\geq0}\big|Z^{(1)}_t-Z'^{(1)}_t\big|
\leq c^{2}\,e^{-\underline{\mathscr{K}}^{(1)}_\infty}
\Big|\frac{1}{z_0}-\frac{1}{z_0'}\Big|
\qquad\text{and}\qquad
\big|Z^{(1)}_t-Z'^{(1)}_t\big|\underset{t\to\infty}{\longrightarrow}0
\quad\text{a.s.},$$
where $\underline{\mathscr{K}}^{(1)}_\infty:=\inf_{t\geq0}K^{(1)}_t$ is almost surely finite by Assumption \ref{ass:a_bar_positive} and Lemma \ref{momentinf}. Coupling two stationary versions in this way shows uniqueness and convergence to the stationary distribution of Proposition \ref{prop_monomorph}. We denote it by $\pi$, and the law of $Z^{(1)}_t$ converges to $\pi$ from any initial condition $z_0>0$.
\end{rem}

\subsubsection{Two populations}

In the case of two types of individuals, as explained in the previous section, the invasion fitnesses defined in \eqref{def_new_fitnesses} will govern the long time behaviour of the interacting populations. Three possible regimes may be distinguished:

\begin{align}
    &\mathfrak{f}_{12}<0
    \text{ and } \mathfrak{f}_{21}<0,
    \tag{EXCL}\label{excl}\\
    &\mathfrak{f}_{ij}>0
    \text{ and }\mathfrak{f}_{ji}<0,
    \text{ for } i\neq j
    \tag{EXCLi}\label{excli}\\
    &\mathfrak{f}_{12}>0
    \text{ and } \mathfrak{f}_{21}>0.
    \tag{COEX}\label{coex}
\end{align}

\begin{prop} \label{3cas}
Let $(Z_t^{(1)},Z_t^{(2)}, t \geq 0)$ be the unique strong solution of \eqref{eqdep} with $(z_0^{(1)},z_0^{(2)}) \in (\R_+^*)^2$. Then:
\begin{enumerate}
        \item If (\ref{excl}) and $a_1/a_2\neq \overline{a}_1/\overline{a}_2$, then $\left\{\begin{array}{ll}
          \Prob\big(\big\{Z^{(1)}_t \to 0\big\} \bigcup \big\{  Z^{(2)}_t \to 0\big\}, \ \text{as $t \to \infty$}\big) &=1,\\
          \Prob\big(\big\{Z^{(1)}_t \to 0\big\} \bigcap \big\{  Z^{(2)}_t \to 0\big\}, \ \text{as $t \to \infty$}\big) &=0.
        \end{array}\right.$
        \item If (\ref{excli}), then $ \Prob\big(\big\{Z^{(i)}_t \nto 0\big\} \bigcap \big\{  Z^{(j)}_t \to 0\big\}, \ \text{as $t \to \infty$}\big) =1.$
        \item If (\ref{coex}), then
        $ \Prob\big(\big\{Z^{(1)}_t \to 0\big\} \bigcup \big\{  Z^{(2)}_t \to 0\big\}, \ \text{as $t \to \infty$}\big) =0.$
    \end{enumerate}
\end{prop}
\begin{comment}
\begin{rem}
    With the notation \eqref{def_qi}, the additional assumption of the first case reads $a_1q_2\neq a_2q_1$: it excludes the degenerate situation where catastrophes are exactly neutral, that is, where neither $\ADVref{1}$ nor $\ADVref{2}$ holds and where the ratio $\overline a_1/\overline a_2$ does not depend on $r$, see \eqref{monotoniea/a}. It is only needed in the bistable case \eqref{excl}, in which the two regions of the state space leading to the extinction of one or the other type are reached thanks to a well-chosen sequence of catastrophes, favourable to one type relatively to the other (see Lemma \ref{irred}). We expect the conclusion to remain valid in the neutral case, which our accessibility argument does not cover.
\end{rem}
\end{comment}
In the next proposition, we refine the results and quantify the asymptotic mean size of both populations in the different scenarios.

\begin{prop}
\label{ergofaible}
    Suppose that $\E\left[(\ln \Theta^{(i)})^2\right]<\infty$ for $i\in\{1,2\}$.
    Let $Z$ be the unique strong solution of \eqref{eqdep} with $(z_0^{(1)},z_0^{(2)}) \in (\R_+^*)^2$. Then:
    \begin{enumerate}
        \item If (\ref{excl}) or (\ref{excli}) for some $i\in \{1,2\}$, then for each $j=1,2$, we have $$\displaystyle\lim_{t\to\infty}\frac{1}{t}\int_0^tZ^{(j)}_sds=\frac{\overline{a}_j}{\beta_{jj}}\ind_{\{\lim_{t\to\infty} Z_t^{(3-j)}=0\}}~~\text{almost surely}.$$
        \item If (\ref{coex}), then the following limits almost surely exist
        $$z^{(1)}_\infty:=\lim_{t\to\infty}\frac{1}{t}\int_0^tZ^{(1)}_sds~~\text{ and }~~z^{(2)}_\infty:=\lim_{t\to\infty}\frac{1}{t}\int_0^tZ^{(2)}_sds,$$
        and satisfy
        \begin{equation*} \label{syst-coex}
    \left\{\begin{array}{ll}
            \overline{a}_1-\beta_{11}z^{(1)}_\infty-\beta_{12}z^{(2)}_\infty &=0,  \\
            \overline{a}_2-\beta_{21}z^{(1)}_\infty-\beta_{22}z^{(2)}_\infty &=0 
        \end{array}\right.        
        \end{equation*}
        whose unique solution is that of the deterministic system with $a_i$ replaced by
$\overline{a}_i$ (see \eqref{xyeq}). 
    \end{enumerate}
\end{prop}

In what follow, let
\begin{align}\label{eq:z2bar}
\overline{z}_2:=\begin{cases}
    z_{\infty}^{(2)}&\text{ if \ \eqref{coex},}\\
\overline{a}_2/\beta_{22}&\text{ if \ $\EXCLref{2}$}.
\end{cases}
\end{align}

\subsection{When catastrophes reverse the hierarchy of traits}
\label{subsec_reversion}

Propositions~\ref{3cas} and \ref{ergofaible} reveal counterintuitive behaviours arising from the introduction of random catastrophes into the deterministic system \eqref{def_LV_det}. In this section, we compare the long-time behaviour of the system with ($r>0$) and without ($r=0$) catastrophes, all the other parameters being fixed.

Assume for instance that $f_{12}>0>f_{21}$, so that without catastrophes type $2$ gets extinct while type $1$ reaches a positive stationary state. Assume in addition that
\begin{equation}\label{cond:jump}
\E\big[|\ln \Theta^{(1)}|\big] < \E\big[|\ln \Theta^{(2)}|\big] ,
\end{equation}
so that catastrophes are, on a logarithmic scale, more harmful to type $2$ than to type~$1$. One would then expect catastrophes to reinforce the dominance of type~$1$. This is not the case: \eqref{cond:jump} is neither necessary nor sufficient for the system with catastrophes to satisfy $\EXCLref{1}$. The correct condition, given in Definition~\ref{def_ADV} below, involves the growth rates as well.
\begin{comment} Throughout this section we set

\begin{equation} \label{def_qi}
q_i:=-\E[\ln\Theta^{(i)}]=\E[|\ln\Theta^{(i)}|]\in(0,\infty), \quad i=1,2,
\qquad\text{so that }\ \bar a_i=a_i-rq_i.
\end{equation}
\end{comment}
\subsubsection*{Visualisation of the three regimes on a single axis}

For $\alpha=(\alpha_1,\alpha_2)\in(\R_+^*)^2$ write $\mathfrak f^{\alpha}_{ij}:=\alpha_i-\frac{\beta_{ij}}{\beta_{jj}}\alpha_j$, so that $f_{ij}=\mathfrak f^{(a_1,a_2)}_{ij}$ and $\mathfrak f_{ij}=\mathfrak f^{(\bar a_1,\bar a_2)}_{ij}$. For $i\neq j$,
\begin{equation*} \label{sign_mathfrakf}
\mathfrak f^{\alpha}_{ij}>0
\iff \frac{\alpha_i}{\alpha_j}>\frac{\beta_{ij}}{\beta_{jj}},
\end{equation*}
so that the three regimes only depend on the position of the ratio $\alpha_1/\alpha_2$ with respect to the two thresholds $\beta_{12}/\beta_{22}$ and $\beta_{11}/\beta_{21}$:
\begin{equation}\label{sign_f}
\begin{array}{c}
\EXCLref{i} \iff
\frac{\alpha_i}{\alpha_j}>\frac{\beta_{ij}}{\beta_{jj}}\vee\frac{\beta_{ii}}{\beta_{ji}},
\quad
\eqref{coex}\iff \frac{\beta_{12}}{\beta_{22}}<\frac{\alpha_1}{\alpha_2}<\frac{\beta_{11}}{\beta_{21}},\\
\eqref{excl}\iff \frac{\beta_{11}}{\beta_{21}}<\frac{\alpha_1}{\alpha_2}<\frac{\beta_{12}}{\beta_{22}} .
\end{array}
\end{equation}
Applying \eqref{sign_f} with $\alpha=(\bar a_1,\bar a_2)$ gives the classification of Proposition~\ref{3cas}, and with $\alpha=(a_1,a_2)$ the one of Proposition~\ref{compLVdet}. In particular, coexistence is possible, for a suitable value of the ratio $\alpha_1/\alpha_2$, if and only if
\begin{equation}\tag{$\beta$-COEX}\label{Bcoex}
\frac{\beta_{12}}{\beta_{22}}<\frac{\beta_{11}}{\beta_{21}},
\end{equation}
a condition unaffected by the catastrophes. The two possible scenarios are drawn in Figure~\ref{fig:axes}.

\begin{figure}[ht]
\centering
\begin{tikzpicture}[>=stealth,scale=1.05]
  \draw[->] (-1,0) -- (6,0) node[right] {$\frac{\alpha_1}{\alpha_2}$};
  \foreach \x/\name in {1.5/{\frac{\beta_{12}}{\beta_{22}}}, 4/{\frac{\beta_{11}}{\beta_{21}}}}
    {\draw (\x,0.2) -- (\x,-0.2) node[below=3pt] {$\name$};}
  \node[below=8pt] at (0.4,0) {$\EXCLref{2}$};
  \node[below=8pt] at (2.75,0) {\eqref{coex}};
  \node[below=8pt] at (5.2,0) {$\EXCLref{1}$};
  \node at (2.5,0.75) {\small \eqref{Bcoex} holds};
\end{tikzpicture}

\vspace{0.8cm}
\begin{tikzpicture}[>=stealth,scale=1.05]
  \draw[->] (-1,0) -- (6,0) node[right] {$\frac{\alpha_1}{\alpha_2}$};
  \foreach \x/\name in {1.5/{\frac{\beta_{11}}{\beta_{21}}}, 4/{\frac{\beta_{12}}{\beta_{22}}}}
    {\draw (\x,0.2) -- (\x,-0.2) node[below=3pt] {$\name$};}
  \node[below=8pt] at (0.4,0) {$\EXCLref{2}$};
  \node[below=8pt] at (2.75,0) {\eqref{excl}};
  \node[below=8pt] at (5.2,0) {$\EXCLref{1}$};
  \node at (2.5,0.75) {\small \eqref{Bcoex} does not hold};
\end{tikzpicture}
\caption{The different accessible regimes of the system, depending on the position of $\alpha_1/\alpha_2$, with $\alpha=(\bar a_1,\bar a_2)$ (resp. $\alpha=(a_1,a_2)$) with (resp. without) catastrophes. Only the ratio depends on $r$: adding catastrophes makes the system move along the axis, without changing the two thresholds.}
\label{fig:axes}
\end{figure}
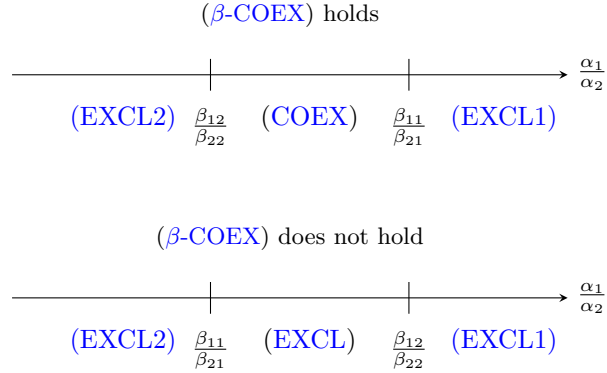

\begin{rem}\label{rem_marginals}
The averaged growth rates $\bar a_i$, hence the fitnesses $\mathfrak f_{ij}$ and the whole classification of Propositions~\ref{3cas} and \ref{ergofaible}, depend on the law $\kappa$ only through its marginals. The dependence structure between $\Theta^{(1)}$ and $\Theta^{(2)}$ is thus irrelevant for the long-time behaviour, although it does enter the accessibility argument of Lemma~\ref{irred}.
\end{rem}

\subsubsection*{Beneficial catastrophes}
Recall the notation $q_i$ introduced in \eqref{def_overline_ai} and for $i\in \{1,2\}$, let $r_i^{*}:=a_i/q_i$, so that $\bar a_i>0$ if and only if $r<r^{*}_i$; for $r \in [0,r_1^{*}\wedge r_2^{*})$ we have
\begin{equation} \label{monotoniea/a}
\frac{\bar a_1}{\bar a_2}=\frac{a_1-rq_1}{a_2-rq_2}
=\frac{a_1}{a_2}+\frac{r}{a_2\bar a_2}\left(a_1q_2-a_2q_1\right),
\qquad
\frac{d}{dr}\left(\frac{\bar a_1}{\bar a_2}\right)=\frac{a_1q_2-a_2q_1}{\bar a_2^{\,2}},
\end{equation}
the terms in $r$ cancelling in the derivative. Hence, the ratio $\bar a_1/\bar a_2$ is strictly monotone in $r$, its sign of variation being that of $a_1q_2-a_2q_1$, which depends only on the growth rates and the marginals of $\kappa$. This motivates the following definition.

\begin{defi}\label{def_ADV}
We say that catastrophes are beneficial to type $i$, relative to type $j$, iff
\begin{equation}\label{ADVi}\tag{ADVi}
a_i\E[|\ln\Theta^{(j)}|]>a_j\E[|\ln\Theta^{(i)}|],
\qquad\textit{i.e.}\qquad a_iq_j>a_jq_i .
\end{equation}
\end{defi}

Equivalently, $\ADVref{i}$ holds if and only if $r\mapsto\bar a_i/\bar a_j$ is increasing, if and only if $r^{*}_i>r^{*}_j$. In that case, increasing $r$ pushes the system to the right on the axis of Figure~\ref{fig:axes} when $i=1$, and to the left when $i=2$: the regimes are therefore visited in the order given by that
axis.  In particular, under \eqref{Bcoex} a deterministic exclusion may successively become coexistence and then the opposite exclusion. The next proposition quantifies this last transition.

\begin{prop}\label{prop_ADV_EXCL}
Assume $\ADVref{i}$ and let $(\beta_{kl})_{k,l=1,2}$ be positive. Set
\begin{equation*}\label{def_r0}
C:=\frac{\beta_{ij}}{\beta_{jj}}\vee\frac{\beta_{ii}}{\beta_{ji}},
\qquad
r_0:=\left(\frac{Ca_j-a_i}{Cq_j-q_i}\right)\vee 0.
\end{equation*}
Then $0\le r_0<r^{*}_j$ and, for every $r\in(r_0,r^{*}_j)$, Assumption \ref{ass:a_bar_positive} holds and the system satisfies $\EXCLref{i}$. Moreover $(r_0,r^{*}_j)$ is exactly the set of $r\in[0,r_j^{*})$ for which $\EXCLref{i}$ holds.
\end{prop}
\begin{proof}
The condition $\ADVref{i}$ reads $r_i^{*}>r_j^{*}$, so that $\bar a_i(r)>\bar a_i(r_j^{*})=q_i(r^*_i-r^*_j)>0$ and $\bar a_j(r)>0$ for every $r\in[0,r^{*}_j)$. Then,  Assumption \ref{ass:a_bar_positive} holds on $[0,r^{*}_j)$ and, by \eqref{sign_f}, $\EXCLref{i}$ holds iff $\bar a_i/\bar a_j>C$. By
\eqref{monotoniea/a} the function $r\mapsto\bar a_i(r)/\bar a_j(r)$ is continuous and  increasing on $[0,r^{*}_j)$, and tends to $+\infty$ as $r\uparrow r^{*}_j$ since $\bar a_i(r^{*}_j)>0$ while $\bar a_j(r)\to0^{+}$. Therefore, it is an increasing bijection from $[0,r^{*}_j)$ to $[a_i/a_j,+\infty)$, so that $\{r<r^{*}_j:\ \bar a_i/\bar a_j>C\}=(r_0,r^{*}_j)$.
\end{proof}

\subsubsection*{A numerical example}
In Figure \ref{fig:simu}, we give an example of a set of parameters satisfying both $\ADVref{1}$ and $\Theta^{(1)}<\Theta^{(2)}$ almost surely: catastrophes are beneficial (in the long-run) to type-$1$ individuals, even if (at the time of the catastrophe), they affect more type-$1$ individuals than type-$2$. Note that this can only occur with $a_1>a_{2}$. Heuristically, the addition of catastrophes will on average "pull" the system closer to the value $(0,0)$, that is, bring it to a regime where the intrinsic growth rate becomes more important than the effect of competition, which is from the previous remark a regime that benefits $1$. As an example, if without catastrophes a typical individual of type $2$ kills $3$ individuals of type $1$ because of competition during its life, we expect the type $1$ to benefit from catastrophes that kill on average two type-$1$ individuals for one type-$2$ individual.
\begin{figure}\label{fig:simu}
\begin{minipage}{1\textwidth}

    \centering
    \begin{minipage}{0.335\textwidth}
        \centering
        \includegraphics[width=\linewidth]{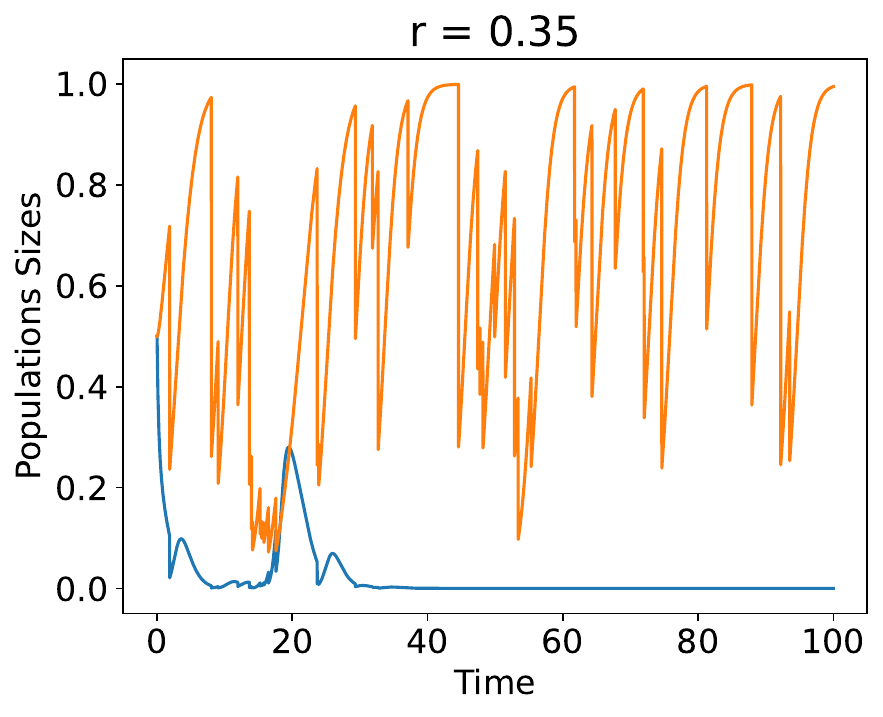}
    \end{minipage}
    \begin{minipage}{0.32\textwidth}
        \centering
        \includegraphics[width=\linewidth]{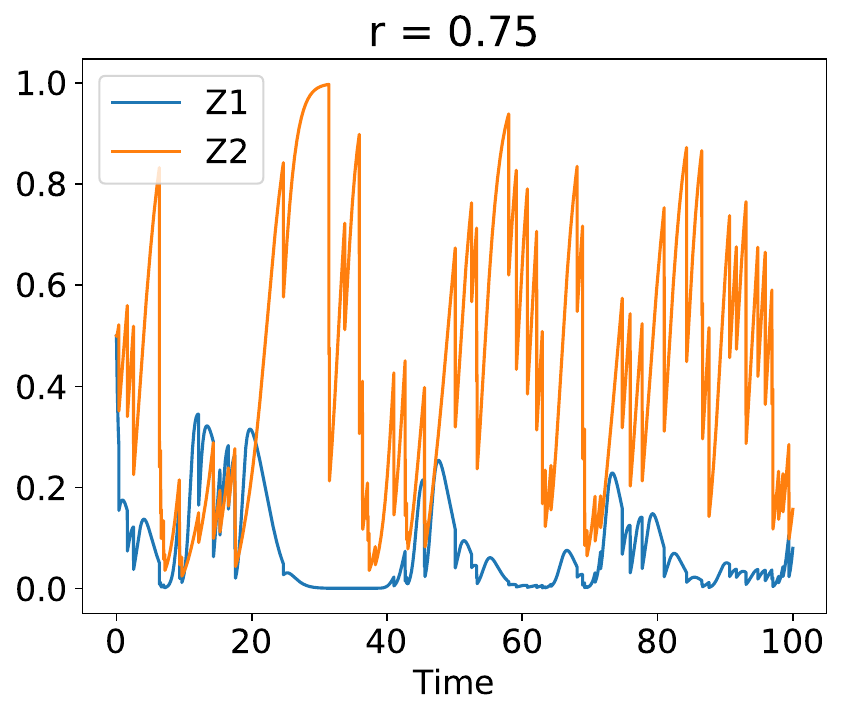}
    \end{minipage}
    \begin{minipage}{0.32\textwidth}
        \centering
        \includegraphics[width=\linewidth]{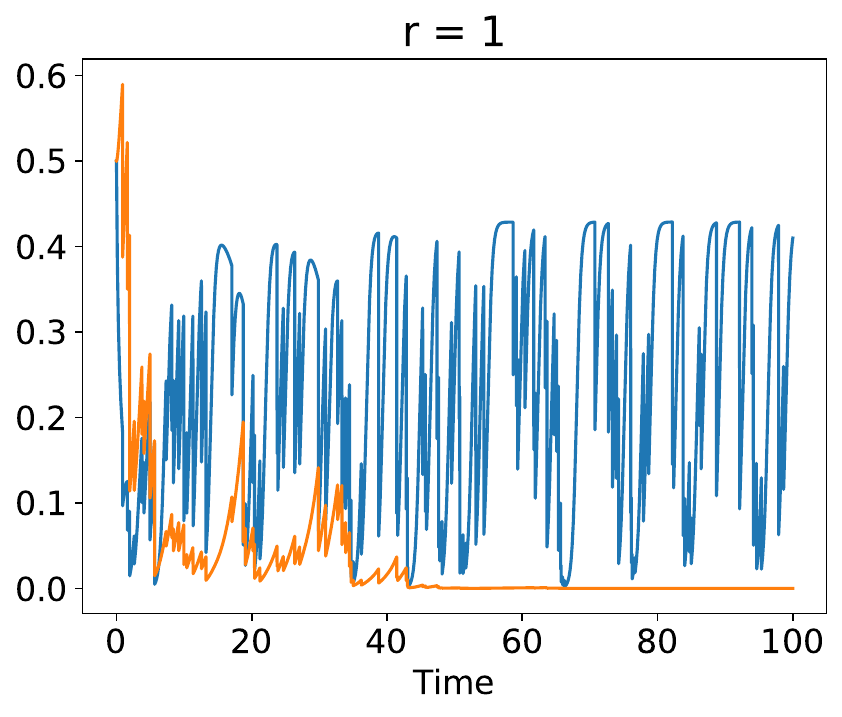}
    \end{minipage}
\vspace{0.2cm}
\begin{tikzpicture}[>=stealth,scale=1.2]
  \draw[->] (-1.5,0) -- (10,0) node[below=3pt] {$\frac{\overline{a}_1}{\overline{a}_2}$} node[above=3pt]{$r$};

  \foreach \x/\name/\nameb in {2.7/{\frac{\beta_{12}}{\beta_{22}}=4}/{}, 6.6/{\frac{\beta_{11}}{\beta_{21}}=7}/{}} {
    \draw (\x,0.2) -- (\x,-0.2) node[below=14pt] {$\name$} node[above=18pt] {$\nameb$};
  }

  \node[below=8pt] at (-1,0) {$\EXCLref{2}$};
  \node[below=8pt] at (4.7,0) {\eqref{coex}};
  \node[below=8pt] at (8.4,0) {$\EXCLref{1}$};

  \node at (0.5,0) {\color{olive} $\bullet$};
  \node[above=3pt] at (0.5,0) {\footnotesize $r=0.35$};
  \node[below=3pt] at (0.5,0) {\footnotesize $\frac{\overline{a}_1}{\overline{a}_2}\simeq3.54$};

  \node at (3.5,0) {\color{olive} $\bullet$};
  \node[above=3pt] at (3.5,0) {\footnotesize $r=0.75$} ;\node[below=3pt] at (3.5,0) {\footnotesize $\frac{\overline{a}_1}{\overline{a}_2}\simeq4.93$};

  \node at (7.3,0) {\color{olive} $\bullet$};
  \node[above=3pt] at (7.3,0) {\footnotesize $r=1$};
  \node[below=3pt] at (7.3,0) {\footnotesize $\frac{\overline{a}_1}{\overline{a}_2}\simeq7.38$};
\end{tikzpicture}
\end{minipage}
\caption{Three realizations of \eqref{eqdep} with different catastrophe rates $r$ and
$a_1=3$, $a_2=1$, $\beta_{11}=7$, $\beta_{12}=4$, $\beta_{21}=1$, $\beta_{22}=1$, $\Theta^{(1)}\sim\operatorname{Unif}(1/8,5/8)$ and
$\Theta^{(2)}=\Theta^{(1)}+1/8$ almost surely. Although each catastrophe is more harmful to type~$1$, increasing $r$ turns the exclusion of type~$1$ into
coexistence. Increasing $r$ further, beyond the second threshold $r_0$ even leads to the exclusion of type~$2$: the reversal
of the hierarchy is complete.}
\label{fig:simu}
\end{figure}
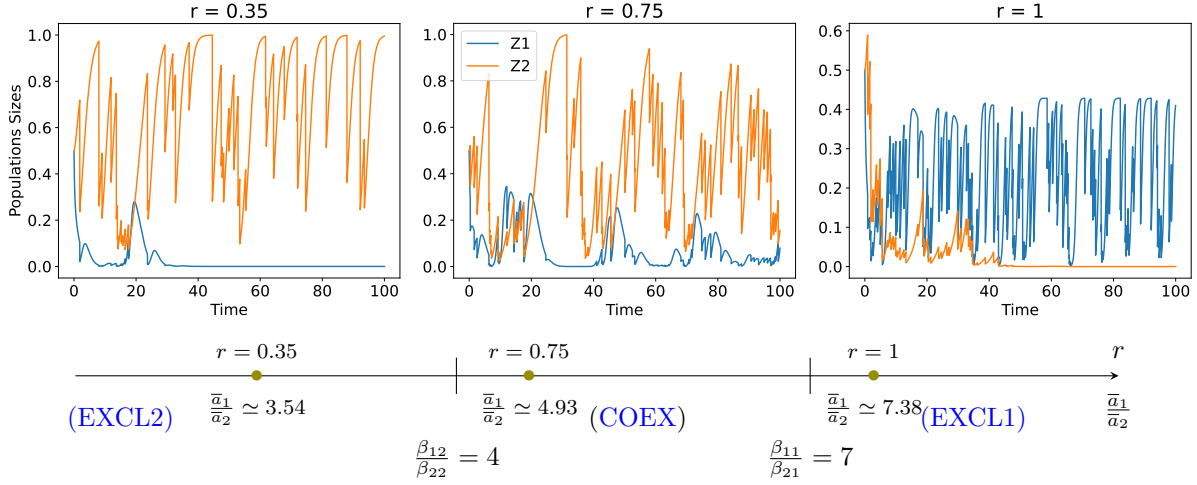

\subsubsection*{Link with another type of disturbed Lotka-Volterra}

Benaïm and Lobry \cite{benaim2016lotka} obtained a similar reversal for a different perturbation of \eqref{def_LV_det}: let $(N_t, t\geq 0)$ be a two-state pure jump Markov process, independent of the population, and let $(X_t,Y_t)$ be the solution to the random but continuous differential equation system $(E_{N_t})$ obtained by alternating between two sets of parameters. The classification of the longtime behaviour of the switching system is governed by the signs of the invasion rates $\Lambda_X$ and $\Lambda_Y$, which play the role of our $\mathfrak f_{ij}$: both are the mean growth rate of a rare type when the other one is at its own equilibrium, and their Theorems 3.1, 3.3, 3.4 and 4.1 are the analog of Proposition~\ref{3cas}. 
Note that in \cite{benaim2016lotka}, the invasion fitnesses are defined as an infinitesimal arithmetic mean growth rate. In our case, if $\pi$ denotes the stationary law of resident type $1$ (Remark~\ref{rem_lossofmemory}), from \eqref{eqdep} and Proposition~\ref{ergofaible}, the comparison between the arithmetic mean growth rate and the invasion fitness gives
$$
\lim_{\eps\to0}\lim_{t\to0^{+}}\frac1t\,
\E_{(\pi,\delta_\eps)}\left[\frac{Z^{(2)}_t}{Z^{(2)}_0}-1\right]
=a_2-\frac{\beta_{21}}{\beta_{11}}\bar a_1+r\E[\Theta^{(2)}-1]
>\mathfrak f_{21}
=\lim_{\eps\to 0}
\E_{(\pi,\delta_\eps)}\left[\ln\left(\frac{ Z^{(2)}_1}{ Z^{(2)}_0}\right)%\Big| Z^{(1)}_0\sim \pi, Z_0^{(2)}=\eps
\right],
$$
where we use the notation $\E_{(\pi,\delta_\eps)}$ to denote the expectancy under the initial conditions $Z^{(1)}_0\sim \pi, Z_0^{(2)}=\eps$.
The gap $r\big(\E[\Theta^{(2)}]-1-\E[\ln\Theta^{(2)}]\big)$ is positive by the strict concavity of the logarithm. Because catastrophes act multiplicatively on the population size, only the logarithmic version governs the long-time behaviour.

Finally, both works exhibit the same counterintuitive phenomenon: the random environment - switching in \cite{benaim2016lotka}, catastrophes here - can reverse the final state of the system, even if it is detrimental to the species that eventually survives. This close analogy suggests a more general
question: is there a broader class of randomly perturbed ecological models whose long-term dynamics is entirely determined by the mean growth rates of rare types? This goes beyond the scope of this paper and will be addressed in a future study.

\subsection{Arrival of a mutant in a monomorphic population}
\label{newcomer_section}
We now analyze the probability of invasion of a newcomer in a one-type community by investigating the long-term behaviour of the process with initial conditions $\Gamma_{0}^{(1,K)}$ of order $K$ and $\Gamma_{0}^{(2,K)}=1$. Before stating the main result of this section, let us introduce a family of stopping times. For $\eps\geq0$ and $K\in\mathbb{N}$,
\begin{equation}
    \label{TKtilde}
    \widetilde{T}^K_{\eps}:=\inf\{t\geq 0,\Gamma_t^{(2,K)}=\lceil K\eps\rceil\}.
\end{equation}

\begin{prop}
\label{possible_invasion}
  Assume that \eqref{H3} holds, and for $C>0$ take the initial conditions $(\Gamma_0^{(1,K)},\Gamma_0^{(2,K)})=(\lfloor CK\rfloor, 1)$. Recall the definition of $\overline{z}_2$ in \eqref{eq:z2bar}.

    If $\mathfrak{f}_{21}<0$, then, for any $\mathfrak{a}>0$,
      \begin{equation}
        \label{prop_inv_debut}\left(\frac{\widetilde{T}^K_0\wedge \widetilde{T}^K_{\mathfrak{a}}}{\ln K},\ind_{\{\widetilde{T}^K_0>\widetilde{T}^K_{\mathfrak{a}}\}}\right)\underset{K\to \infty}{\overset{\P}{\longrightarrow}}(0,0).\end{equation}

    If $\mathfrak{f}_{21}>0$, then, for any $z<\overline{z}_2$,
   \begin{equation}
        \label{prop_inv_intermediaire}\lim_{\eps \to 0}\limsup_{K \to \infty}d_{\mathrm{TV}}\left(\left(\frac{\widetilde{T}^K_0\wedge \widetilde{T}^K_{\eps}}{\ln K},\ind_{\{\widetilde{T}^K_0>\widetilde{T}^K_{\eps}\}}\right),B \cdot\left(\mathfrak{f}_{21}^{-1},1\right)\right)=0,\end{equation}
  \begin{equation} \label{prop_inv_final}\limsup_{K \to \infty}d_{\mathrm{TV}}\left(\left(\frac{\widetilde{T}^K_0\wedge \widetilde{T}^K_{z}}{\ln K},\ind_{\{\widetilde{T}^K_0>\widetilde{T}^K_{z}\}}\right),B \cdot \left(\mathfrak{f}_{21}^{-1},1\right)\right)=0,\end{equation}
    where $d_{\mathrm{TV}}$ is the total variation distance and $B$ is a Bernoulli random variable with parameter
$p=p(C)\in(0,1)$, the survival probability of the limiting branching process in random environment. Explicitly, let $0=T_0<T_1<T_2<\cdots$ be the catastrophe times, let $\Theta^{(2)}_k$ be the survival probability of type-$2$ individuals at the $k$th catastrophe, let $(Z^{(1)}_t,t\geq0)$ be the solution of \eqref{eqdep} with $Z^{(1)}_0=C$ and $Z^{(2)}\equiv0$. Then,
\begin{equation}\label{formule_p}
p(C)=\E\bigg[\Big(\sum_{j\geq0}\frac{b_j}{\mathcal{M}_j}\Big)^{-1}\bigg],
\qquad\text{with the convention }1/\infty=0,
\end{equation}
where
\begin{equation}\label{formule_p_bis}
\mathcal{M}_j:=\Big(\prod_{k=1}^{j}\Theta^{(2)}_k\Big)e^{\int_0^{T_{j+1}}\left(\lambda_2-\mu_2-\beta_{21}Z^{(1)}_s\right)ds},
\qquad
b_j:=\lambda_2\int_{T_j}^{T_{j+1}}e^{\int_{s}^{T_{j+1}}\left(\lambda_2-\mu_2-\beta_{21}Z^{(1)}_u\right)du}ds .
\end{equation}
\end{prop}

Our analysis will be based on the coupling of $\Gamma^{(2,K)}$ with a family of discrete-time branching processes $((X^{\eps}_n, n\in \N),\eps\in(-\eps_0,\eps_0))$ for some $\eps_0>0$ in a randomly varying environment, in such a way that, with high probability, when $K$ is large,
$$ X^{-\eps}_n \leq \Gamma^{(2,K)}_{T_n-} \leq X^{\eps}_n, \quad \forall n \leq N^K_{\eps},\quad \forall \eps \in (0,\eps_0),$$
where $N^K_{\eps}$ is a large stopping time that will be defined later.
The strategy consists then in proving that $X^{-\eps}$ and $X^{\eps}$ are arbitrarily close to each other when $\eps$ is very small, and in quantifying their survival probability.
The main idea is the following. When the number of type-$2$ individuals is small compared to $K$, the competition exerted by type-$2$ individuals on both type-$1$ and type-$2$ individuals can be regarded as negligible. Hence, only the competition exerted by type-$1$ individuals plays a significant role. As a consequence, as long as the system remains in a phase where $\Gamma^{(2,K)} \ll K$, the process $\Gamma^{(1,K)}$ can be effectively treated as decoupled from the rest of the population.
From the point of view of the type-$2$ invading individuals, the relevant quantities are only the number of type-$1$ individuals and the catastrophes that occur. 
These elements will constitute the random environment (see Section \ref{subsec_introofbpre} for details).

We now state the consequences of Proposition \ref{possible_invasion} for the finite population process $\Gamma^{(K)}$. The first result identifies the invasion event: its probability converges to $p(C)$, and it can equivalently be read on the hit time of a macroscopic level by the mutant population, or on the maximal size this population reaches before a well-chosen time.

\begin{coro}
\label{coro_invasion}
Take $C>0$ and $\Gamma^{(K)}$ the solution to \eqref{strongSDE} with $\Gamma^{(K)}_0=(\lfloor CK\rfloor,1)$ and $Z^{(K)}=\Gamma^{(K)}/K$. Let $z,z'\in(0,\overline{z}_2)$, where $\overline{z}_2$ is defined in \eqref{eq:z2bar}, and let $(t_K,K\geq1)$ be a sequence of times satisfying $\liminf_K t_K/\ln K>\mathfrak{f}_{21}^{-1}$. Then the two events $$E^K_1:=\big\{\widetilde{T}^K_z<\infty\big\} \qquad\text{and}\qquad
E^K_2:=\Big\{\sup_{s\leq t_K}Z^{(2,K)}_s\geq z'\Big\}$$
satisfy $\lim_{K\to\infty}\P\big(E^K_1\,\triangle\,E^K_2\big)=0$, and their probability converges to $p(C)$, defined in \eqref{formule_p}. In particular
$p(C)>0$ if $\mathfrak{f}_{21}>0$ and $p(C)=0$ if $\mathfrak{f}_{21}<0$.
\end{coro}
We then state results on the mean quantity of individuals in the spirit of Proposition \ref{ergofaible}.
\begin{theo}
\label{theorem2}
Take $C$, $\Gamma^{(K)}$, $Z^{(K)}$, and $\overline{z}_2$ as in Corollary~\ref{coro_invasion}, let $(Z_t,t\geq0)$ be the solution to \eqref{eqdep} with $Z_0=(C,0)$, and let $z\in(0,\overline{z}_2)$. Assume \eqref{H3}.
\begin{enumerate}
\item Assume $\mathfrak{f}_{21}<0$. Then for every $t>0$,
\begin{equation}\label{th2_sub}
\frac{1}{t}\int_0^tZ^{(K)}_s\,ds
\ \xrightarrow[K\to\infty]{\ \P\ }\ \frac{1}{t}\int_0^tZ_s\,ds ,
\qquad\text{and}\qquad
\frac{1}{t}\int_0^tZ_s\,ds\ \xrightarrow[t\to\infty]{\mathrm{a.s.}}\
\Big(\frac{\overline{a}_1}{\beta_{11}},0\Big).
\end{equation}
\item Assume $\mathfrak{f}_{21}>0$ and let
$\P^{K}_z(\,\cdot\,):=\P\big(\,\cdot\mid\widetilde{T}^{K}_z<\infty\big)$, which is
well defined for $K$ sufficiently large by Corollary~\ref{coro_invasion}. Set
$$\mathfrak{z}:=\begin{cases}
(z^{(1)}_\infty,z^{(2)}_\infty) & \text{if }\mathfrak{f}_{12}>0,\\[2pt]
(0,\overline{a}_2/\beta_{22}) & \text{if }\mathfrak{f}_{12}<0.
\end{cases}$$
Then, for every $\eta>0$,
\begin{equation}\label{th2_super}
\lim_{t\to\infty}\ \limsup_{K\to\infty}\
\P^{K}_z\Big(\Big\Vert\frac{1}{t}\int_0^tZ^{(K)}_{\widetilde{T}^{K}_z+s}\,ds
-\mathfrak{z}\Big\Vert_2>\eta\Big)=0 .
\end{equation}
\end{enumerate}
\end{theo}

In words, once the type $2$ has reached the macroscopic level $z$ - which happens with asymptotic probability $p>0$ - the occupation measure of the rescaled process stabilises around $\mathfrak{z}$, the long-time limit of the deterministic system with jumps.

Such a theorem raises the question of the behaviour of the process on the time scale $\ln K$, the typical extinction time of a subcritical process started at order $K$ \cite{champagnat2006microscopic}. We leave this question, which requires further work on stationary distributions, for future research.

\section{Proofs}

This section is dedicated to the proofs of the different previous results.

\subsection{Convergence of the discrete population model} \label{sec_proof_conv}

We  begin with the proof of Lemma~\ref{lem_Kurtz}. It follows from successive extensions of Theorem 2.1 p.456 in \cite{ethier2009markov} between catastrophes, and the gluing of the different time intervals. Let $T_0=0$ and for $i \geq 1$ let $(T_i, \Theta_i^{(1)},\Theta_i^{(2)})$ \label{def_triplet} in $\R_+ \times (0,1)^2$ be the triplet constituted of the occurrence time of the $i$th catastrophe, and the associated survival probabilities for type $1$ and type $2$ individuals. Let
$$\mathcal{G}:=\sigma\big((T_k,\Theta^{(1)}_k,\Theta^{(2)}_k),\,k\geq1\big)$$
be the $\sigma$-field generated by the catastrophes.
\begin{proof}[Proof of Lemma \ref{lem_Kurtz}]
    All the conditional expectations below are taken with respect to $\mathcal{G}$: notice that, conditionally on $\mathcal{G}$, $Z^{(K)}$ is still a time-inhomogeneous birth-and-death process between two catastrophes, and that $Z$ is $\mathcal{G}$-measurable.
Let $T,\eps>0$. We aim at bounding the probability
$$\P\Big(\Big\{\sup_{s\leq T}\big\|Z^{(K)}_s-Z_s\big\|>\eps\Big\}\cap\mathfrak{A}_k^{(K)}\Big) \qquad\text{where, for }k\geq0,$$
$$\mathfrak{A}_{k}^{(K)}:=\Big\{T_{k+1}>T,\text{ and }\sup_{t\leq T}\big\{Z^{(i,K)}_t\big\} \leq\frac{2ka_i}{\beta_{ii}},\ i\in\{1,2\}\Big\}.$$
In particular $\mathfrak{A}_{k}^{(K)}\subset\mathfrak{A}_{k+1}^{(K)}$. Moreover,
\begin{equation*}
1-\P\left(\mathfrak{A}_k^{(K)}\right) \leq \P(T_{k+1}\leq T) + \sum_{i=1}^2\P\left(\sup_{t\leq T} Z_t^{(i,K)}>\frac{2ka_i}{\beta_{ii}}\right).
\end{equation*}
The first term goes to $0$ uniformly in $K$ as the number of catastrophe before $T$ is Poisson distributed with parameter $rT$. For the second term, by classical large deviations results (see \textit{e.g.} the equation following Eq. (24) in \cite{champagnat2006microscopic} and Lemma 2.2 in \cite[Chapter 4]{FW84}) using that for $i\in\left\{1,2\right\}$, $Z^{(i,K)}$ is dominated by the birth-and-death process without catastrophes, we have
\begin{equation}\label{eq:gde_dev}
    \lim_{k\to\infty}\sup_{K}\P\left(\sup_{t\leq T}Z_t^{(i,K)}>\frac{2ka_i}{\beta_{ii}}\right)=0,
\end{equation}
so that $\lim_{k\to\infty}\inf_K\P\left(\mathfrak{A}_{k}^{(K)}\right)=1$. It is therefore enough to prove that for each fixed $k$, the above probability goes to $0$ as $K$ goes to infinity, and to let
$k$ go to infinity afterwards.
    
By symmetry, we treat only the first coordinate. If $z_0^{(1)}=0$, then $Z^{(1)}=0$ and $\lim_{K\to\infty}\sup_{0\leq t\leq T}Z^{(1,K)}_t=0$ by a direct application of Theorem 2.1 p.456 in \cite{ethier2009markov}, as $Z^{(1,K)}$ is dominated by the process without catastrophes. Therefore, we now assume $z_0^{(1)}>0$. To simplify the proof, we set $Z^{(2,K)}=0$, as the computations made in that case can be adapted straightforwardly to the case $Z^{(2,K)}\neq0$. We also write $Z^{(K)}$ instead of $Z^{(1,K)}$. On the time interval $[0,T_1)$ it has the same law as the strong solution to
\begin{align}\label{eq:ZtK-poisson}
Z^{(K)}_t =& Z^{(K)}_0 + \frac{1}{K} \Big[ \tilde{Y}_1 \Big( K \int_0^t \lambda_1 Z^{(K)}_s ds \Big)-\tilde{Y}_{-1} \Big( K \int_0^t (\mu_1+\beta_{11} Z^{(K)}_s) Z^{(K)}_s ds \Big) \Big]\\
\nonumber& + \int_0^t (a_1-\beta_{11} Z^{(K)}_s) Z^{(K)}_s ds
\end{align}
where $\tilde{Y}_{\pm 1}$ are two independent rate $1$ compensated Poisson random measures and similarly on each interval $[T_i,T_{i+1})$, started at $Z^{(K)}_{T_i}$, by the strong Markov property. Adding \eqref{eqdep} yields
\begin{align*}
Z^{(K)}_t-Z_t &= Z^{(K)}_0-Z_0 + \int_0^t (a_1-\beta_{11} Z^{(K)}_s) Z^{(K)}_s ds-\int_0^t (a_1-\beta_{11} Z_s) Z_s ds\\
&+ \frac{1}{K} \Big[ \tilde{Y}_1 \Big( K \int_0^t \lambda_1 Z^{(K)}_s ds \Big)-\tilde{Y}_{-1} \Big( K \int_0^t (\mu_1+\beta_{11} Z^{(K)}_s) Z^{(K)}_s ds \Big) \Big].
\end{align*}
As $Z^{(K)}$ and $Z$ stay in the compact $[0,2ka_1/\beta_{11}\vee Z_0]$ during the time interval $[0,T]$ on the event $\mathfrak{A}_{k}^{(K)}$, there exists a finite constant $\mathcal{M}_k$ such that on the event $\mathfrak{A}_{k}^{(K)}$:
\begin{align*}
  \big| Z^{(K)}_t- Z_t \big| \leq 
  %&  \big| Z^{(K)}_0 - Z_0 \big|  + \frac{1}{K} \Big[ \left|\tilde{Y}_1 \Big(K \int_0^t \lambda_1 Z^{(K)}_s ds\Big)-\tilde{Y}_{-1} \Big( K \int_0^t (\mu_1+\beta_{11} Z^{(K)}_s) Z^{(K)}_s ds\Big) \right|\Big]\\&+  \mathcal{M}_k\int_0^t  \big| Z^{(K)}_s- Z_s \big|ds=:
  \big| Z^{(K)}_0 - Z_0 \big|  + \frac{|W(t)|}{K} +  \mathcal{M}_k\int_0^t  \big| Z^{(K)}_s- Z_s \big|ds,
\end{align*}
with $W(t):=\tilde{Y}_1 \Big(K \int_0^t \lambda_1 Z^{(K)}_s ds\Big)-\tilde{Y}_{-1} \Big( K \int_0^t (\mu_1+\beta_{11} Z^{(K)}_s) Z^{(K)}_s ds\Big)$.
Applying Grönwall's inequality, we deduce that, under the same condition, for $t<T_1$,
\begin{equation}
\begin{aligned}
\label{lemma_conv_Gronw}
    \big| Z^{(K)}_t- Z_t \big| \leq  e^{\mathcal{M}_kt}\Big( \big| Z^{(K)}_0 - Z_0 |  &+ \frac{1}{K} \sup_{t\leq T_1}\big|W(t)\big|\Big).
    \end{aligned}
    \end{equation}
    Note that this inequality can be obtained with the same reasoning in the case $Z^{(2,K)}\neq 0$ by adding $Z^{(2,K)}$ in the argument of $\tilde{Y}_{-1}$ in \eqref{eq:ZtK-poisson} and afterwards.
On the event $\mathfrak{A}_k^{(K)}$,
$$ \sup_{s\leq T_1} \lambda_1 Z^{(K)}_s \leq \lambda_1 \frac{2k a_1}{\beta_{11}} \quad \text{and} \quad  \sup_{s\leq T_1}  (\mu_1+\beta_{11} Z^{(K)}_s) Z^{(K)}_s \leq  (\mu_1+\frac{2k a_1}{\beta_{11}}) \frac{2k a_1}{\beta_{11}}.  $$
By Doob's maximal inequality, we obtain that for $C_k=\frac{2ka_1}{\beta_{11}}(\lambda_1+\mu_1+\frac{2ka_1}{\beta_{11}})$, $B,\eps>0$
\begin{align} \label{Doob_Wt}
    \P \Big(  \sup_{t\leq T_1}\Big| \frac{W(t)}{K}\Big|  >B \eps \Big) & \leq  \frac{\E\big[ \big| W\big( T_1 \wedge T\big) \big|^2 \big]}{B^2 \eps^2 K^2} \leq \frac{C_k T}{B^2 \eps^2 K}.
\end{align}
This inequality is also satisfied when $Z^{(2,K)}\neq 0$, with an additional term equal to $2ka_2/\beta_{12}$ in the definition of $C_k$. From now on, the proof proceeds exactly the same way in cases $Z^{(2,K)}=0$ or $Z^{(2,K)}\neq0$
Combining this with \eqref{lemma_conv_Gronw} and introducing the event 
$$\mathscr B_i^k (\mathfrak{a},\mathfrak{b}):= \{ |Z^{(K)}_{T_i}- Z_{T_i}|>\mathfrak{a} e^{-\mathcal{M}_k \mathfrak{b} }\} , \quad (\mathfrak{a},\mathfrak{b})\geq 0,$$ 
imply that for any $i\leq k$ and $\eps>0$, taking $B=e^{-\mathcal{M}_k(T_{i+1}-T_i)}/6$ and $\Delta_i=T_{i+1}-T_i$,
\begin{equation}
\label{lemma_conv_derive}
\begin{aligned} &\ind_{\mathfrak{A}_{k}^{(K)}} \P \Big( \sup_{T_i\leq t <T_{i+1}} \big| Z^{(K)}_{t}- Z_{t} \big|  >\frac{\eps}{2}\Big|Z^{(K)}_{T_{i}}, Z_{T_i},\Delta_i\Big) \\ 
\leq & \ind_{\mathfrak{A}_{k}^{(K)}} \Big(\ind_{\mathscr B_i^k(\eps/3,\Delta_i)}+
\frac{36 C_k T e^{2\mathcal{M}_k\Delta_i}}{\eps^2 K}\Big) =: \ind_{\mathfrak{A}_{k}^{(K)}} \Big( \ind_{\mathscr B_i^k(\eps/3,\Delta_i)}+
\frac{\mathcal{D}(k,T)}{\eps^2 K}\Big).
\end{aligned}
\end{equation}
Now we aim at quantifying the extra error due to the approximation of the binomial random variable when a catastrophe occurs. First,
\begin{equation} \label{decomp} \big| Z^{(K)}_{T_{i+1}}- Z_{T_{i+1}} \big| \leq \Theta_{i+1}^{(1)} \big| Z^{(K)}_{T_{i+1}-}- Z_{T_{i+1}-} \big| + \big| Z^{(K)}_{T_{i+1}}- \Theta_{i+1}^{(1)}Z^{(K)}_{T_{i+1}-} \big|. \end{equation}
As a binomial random variable with parameters $n$ and $p$ has mean $np$ and variance $np(1-p)$, we obtain from Chebyshev's inequality, for any $\eps>0$:
\begin{equation}
\label{lemma_conv_sauts}
\begin{aligned}
 &\ind_{\mathfrak{A}_{k}^{(K)}}\P \Big( \big| Z^{(K)}_{T_{i+1}}- Z_{T_{i+1}} \big|  >\eps \Big| Z^{(K)}_{T_{i+1}-},Z_{T_{i+1}-},\Theta_{i+1}^{(1)}\Big)\\ 
&\leq 
 \ind_{\mathfrak{A}_{k}^{(K)}} \left(  \ind_{\{|Z^{(K)}_{T_{i+1}-}- Z_{T_{i+1}-}|>\frac{\eps}{2}\}} + \P \Big(  \big| Z^{(K)}_{T_{i+1}}- \Theta_{i+1}^{(1)}Z^{(K)}_{T_{i+1}-} \big| >\frac{\eps}{2}\Big|Z^{(K)}_{T_{i+1}-},\Theta_{i+1}^{(1)}\Big)\right)\\
    &\leq \ind_{\mathfrak{A}_{k}^{(K)}} \Big(  
    \ind_{\{|Z^{(K)}_{T_{i+1}-}- Z_{T_{i+1}-}|>\frac{\eps}{2}\}}+\frac{4Z^{(K)}_{T_{i+1}-} }{K\eps^2} \Theta_{i+1}^{(1)}(1-\Theta_{i+1}^{(1)})\Big)
\\& \leq\ind_{\mathfrak{A}_{k}^{(K)}} \Big(  
    \ind_{\{|Z^{(K)}_{T_{i+1}-}- Z_{T_{i+1}-}|>\frac{\eps}{2}\}}+\frac{2ka_1}{K\beta_{11}\eps^2}\Big).
\end{aligned}
\end{equation}
Combining \eqref{lemma_conv_derive} and \eqref{lemma_conv_sauts} we get the existence of a finite $\tilde{\mathcal{D}}(k,T)$ such that for $i \leq k$:
\begin{align*}
   \ind_{\mathfrak{A}_{k}^{(K)}} \P \big( \mathscr B_{i+1}^k(\eps,0) \big) & \leq \ind_{\mathfrak{A}_{k}^{(K)}}\Big( \frac{\tilde{\mathcal{D}}(k,T)}{\eps^2 K}  
  + \P \Big( \mathscr B_i^k\Big(\frac{\eps}{3},T_{i+1}-T_i\Big) \Big)  \Big)\\
 & =\ind_{\mathfrak{A}_{k}^{(K)}}\Big( \frac{\tilde{\mathcal{D}}(k,T)}{\eps^2 K}  
  + \P \Big( \mathscr B_i^k\Big(\frac{\eps}{3}e^{-\mathcal{M}_k(T_{i+1}-T_i)},0\Big) \Big)  \Big)
\end{align*}
By iteration, we find that there exists $\mathcal{C}(k,T)>0$ such that
$$  \ind_{\mathfrak{A}_{k}^{(K)}}  \sup_{i \leq k}\P \big( \mathscr B_{i}^k( \mathfrak{a} \eps,0) \big) \leq   \ind_{\mathfrak{A}_{k}^{(K)}} \frac{\mathcal{C}(k,T)}{\eps^2 K}\mathscr{B}_0^k\left(\frac{\eps}{3^{i+1}}e^{-\mathcal{M}_kT_{i+1}},0\right) \xrightarrow[K \to \infty]{} 0.$$
Using again \eqref{lemma_conv_derive} allows us to extend the property to the time interval $[0,T]$ under the event $\mathfrak{A}_{k}^{(K)}$.
This ends the proof of Lemma \ref{lem_Kurtz} in the case $z_0^{(2)}=0$.
\end{proof}

\begin{rem} \label{rem_Kurtz_random}
The proof of Lemma \ref{lem_Kurtz} uses the initial conditions only through the term $\big|Z^{(K)}_0-Z_0\big|$ of \eqref{lemma_conv_Gronw}, and all the other constants involved are uniform over compact subsets of $(\R_+^*)^2$. As a consequence, \eqref{lim_Kurtz} still holds when $Z^{(K)}_0=Z_0$ is random, provided its law is tight in a compact subset of $(\R_+^*)^2$; we shall use it in this form, after an application of the strong Markov property, in the proof of Theorem \ref{theorem2}.
\end{rem}

\begin{proof}[Proof of Proposition \ref{prop_cvglongue}]
Recall the definition of $((T_i, \Theta_i^{(1)},\Theta_i^{(2)}), i \geq 1)$ and $\mathcal{G}$ on page \pageref{def_triplet}. As by assumption $Z_t^{(2,K)}=0$ for all $t\geq 0$, we will write $Z^{(K)}, Z, \Theta_i, K$ instead of $Z^{(1,K)}, Z^{(1)}, \Theta^{(1)}_i, K^{(1)}$, for all $i\geq 1$. For $t \geq 0$, we define
$$ \Delta^{(K)}_t:= \big(Z^{(K)}_t-Z_t\big)^2.$$ 
The structure of the proof is as follows.
\begin{enumerate}
    \item We show that for $i \leq \ln^\alpha K$ the probability that the squared difference $\Delta^{(K)}_{T_i}$ just after catastrophes is larger than some $\mathscr{A}_K>0$ is small.
    \item We control the difference between two successive jumps.
\end{enumerate}
\begin{enumerate}
    \item  \textit{After catastrophes.} Let $\mathscr{A}_K,\mathscr{B}_K>0$ and
$$E_K:=\Big\{\sup_{i\leq\ln^\alpha K}\E\Big[\Delta^{(K)}_{T_i}\big|\mathcal{G}\Big]>\mathscr{B}_K\Big\}\in\mathcal{G}.$$
We have the following inequality:
\begin{equation} \label{eq:separationEK}\P\Big(\sup_{i\leq\ln^\alpha K}\Delta^{(K)}_{T_i}>\mathscr{A}_{K}\Big)
\leq\P\big(E_K\big)
+\E\Big[\ind_{E_K^C}\,\P\Big(\sup_{i\leq\ln^\alpha K}\Delta^{(K)}_{T_i}>\mathscr{A}_{K}\Big|\mathcal{G}\Big)\Big].\end{equation}
We now prove that the probability of $E_K$ is small, for a well-chosen $\mathscr{B}_K$. First, we control what happens when catastrophes occur.
Using \eqref{decomp}, we get that
\begin{align*}  \Delta^{(K)}_{T_1}= \Theta_{1}^2  \Delta^{(K)}_{T_{1}-} +\Big( Z^{(K)}_{T_{1}}- \Theta_{1}Z^{(K)}_{T_{1}-} \Big)^2 +2\Theta_{1}\big( Z^{(K)}_{T_{1}-}- Z_{T_{1}-} \big)\Big( Z^{(K)}_{T_{1}}- \Theta_{1}Z^{(K)}_{T_{1}-} \Big) . 
\end{align*}
Using the properties of binomial random variables as in \eqref{lemma_conv_sauts} leads to
\begin{align} \label{eq:boundwithjump}\E \Big[ \Delta^{(K)}_{T_1} \big| Z^{(K)}_{T_1-},Z_{T_1-},\Theta_{1}\Big] & = \big(\Theta_{1}\big)^2\Big(Z^{(K)}_{T_1-}-Z_{T_1-}\Big)^2+\Theta_{1} (1-\Theta_{1})\frac{Z^{(K)}_{T_1-}}{K}.
\end{align}
Next, we control the difference $\Delta_t^{(K)}$ before the first catastrophe.
For all $t< T_1$, using that $dZ_t/dt = Z_tR_t$, where $R_t:=a_1-\beta_{11}Z_t$,
\begin{align*}
  \frac{d}{dt}\E \big[ \Delta^{(K)}_t\big] =   \frac{d}{dt}\E \big[ \big(Z^{(K)}_t\big)^2 \big] + 2\big(Z_t\big)^2R_t  - 2Z_t\frac{d}{dt}\E \big[ Z^{(K)}_t\big]- 2\E \big[ Z^{(K)}_t\big]Z_tR_t.
\end{align*}
From Itô's formula, if we define $R^{(K)}_t:=a_1-\beta_{11}Z^{(K)}_t$, we obtain that for any $t < T_1$,
$$ \begin{aligned} 
\frac{d}{dt}\E \big[ Z^{(K)}_t \big]
=\E\big[Z^{(K)}_tR^{(K)}_t\big]\leq a_1 \E[Z^{(K)}_t]-\beta_{11}\E[Z^{(K)}_t]^2,
\end{aligned}$$
where we used Jensen's inequality. Then 
\begin{align}\label{eq:bound_l1}
    \sup_{t< T_1}\E \big[ Z^{(K)}_t \big]\leq C_1 := \frac{a_1}{\beta_{11}}\vee\left( \sup_{K}\E[{Z_0^{(K)}}]\right).
\end{align}
In addition, using again Jensen's inequality, we have
\begin{align*}
    \frac{d}{dt}\E \big[ \big(Z^{(K)}_t\big)^2 \big]& = 2\E\big[ \big(Z^{(K)}_t\big)^2R^{(K)}_t\big] + \E\big[ Z_t^{(K)}(2\lambda_1-R_t^{(K)})/K\big]\\
     & \leq  \big(2a_1 + \frac{\beta_{11}}{K}\big)\E\big[ \big(Z^{(K)}_t\big)^2\big] + \frac{\lambda_1 + \mu_1}{K}C_1 -2\beta_{11}\E\big[ \big(Z^{(K)}_t\big)^2\big]^{3/2} .
\end{align*}
Then, $\sup_{t\geq0}\E[ \big(Z^{(K)}_t\big)^2]<\infty$ if $\E[\big(Z^{(K)}_0\big)^2]<\infty$. We deduce that 
\begin{align*}
    \frac{d}{dt}\E \big[ \big(Z^{(K)}_t\big)^2 \big]& \leq 2\E\big[ \big(Z^{(K)}_t\big)^2R^{(K)}_t\big] + \frac{C_2}{K},
\end{align*}
for some $C_2>0$. Combining the previous inequalities, we obtain that for $t<T_1$,
\begin{align*}
 \frac{1}{2} \frac{d}{dt}\E \big[\Delta^{(K)}_t \big]
  \leq  \E\big[ \big(Z^{(K)}_t\big)^2R^{(K)}_t \big]+ Z_t^2 R_t- Z_t\E\big[ Z^{(K)}_tR^{(K)}_t \big] - \E\big[ Z^{(K)}_t \big]Z_t R_t+\frac{C_2}{K}.
\end{align*}
Using $R_t^{(K)} = R_t - \beta_{11}\big( Z_t^{(K)}-Z_t\big)$, we get
\begin{align*}
  \frac{d}{dt}\E \big[ \Delta^{(K)}_t \big] \leq & 2R_t\E\big[ \Delta^{(K)}_t\big ]  - 2\beta_{11}\E\big[Z^{(K)}_t \Delta_t^{(K)}\big]+\frac{C_2}{K}
  \leq  2 R_t\E\big[ \Delta^{(K)}_t\big]+\frac{C_2}{K}.
\end{align*}
Applying Grönwall's lemma, we obtain that
\begin{align*}
 \E \big[\Delta_t^{(K)}\big|\mathcal{G} \big] \leq \Big( \E \big[ \Delta_0^{(K)}\big|\mathcal{G} \big]+ \frac{C_2t}{K} \Big)e^{2\int_0^tR_sds}.
 \end{align*}
Combining this last equality with \eqref{eq:boundwithjump}, we obtain
\begin{align}
\notag
\E \Big[ \Delta_{T_1}^{(K)}\big| T_1,\Theta_{1}\Big] & \leq \Big( \E \big[\Delta^{(K)}_0 \big]+ \frac{C_2T_1}{K} \Big)e^{2(\int_0^{T_1}R_sds+\ln \Theta_{1})}+ \frac{\Theta_{1}(1-\Theta_{1})}{K}\E \Big[Z^{(K)}_{T_1-} \Big]\\
& \leq  \Big( \E \big[ \Delta^{(K)}_0\big]+ \frac{C_2T_1}{K} \Big)\left( \frac{Z_{T_1}}{Z_{0}} \right)^2+ \frac{C_1}{4K},\label{eq:bound_deltawithjump}
\end{align}
where the last inequality comes from \eqref{eq:bound_l1}.
Then, if we define for $i\geq 0$
$$ \mathfrak{u}_i :=  \E \Big[ \Delta^{(K)}_{T_i}\big| ((T_k,\Theta_{k}, 1 \leq k \leq i)\Big]+ \frac{C_2(T_{i+1}-T_i)}{K}, \quad 
\eta_{i} :=  \frac{C_1}{4K} + \frac{C_2(T_{i+1}-T_{i})}{K},
$$
according to \eqref{eq:bound_deltawithjump}, we obtain for $i\geq 1$
$$ \mathfrak{u}_i \leq \left( \frac{Z_{T_i}}{Z_{T_{i-1}}} \right)^2\mathfrak{u}_{i-1}+\eta_{i}, $$
which entails
\begin{align*} \mathfrak{u}_i \ &\leq\ \sum_{l=0}^{i}\eta_{l}\prod_{k=l+1}^{i}\left( \frac{Z_{T_k}}{Z_{T_{k-1}}} \right)^2
= Z_{T_{i}}^2\sum_{l=0}^{i}\frac{\eta_{l}}{Z_{T_{l}}^{2}}
\\ &\leq\ C_3^2\Big((i+1)\frac{C_1}{4K}+ \frac{C_2T_{i+1}}{K}\Big)\Big( \sum_{k=0}^{i} Z_{T_{k}}^{-1} \Big)^2,
\end{align*}
where we used in the last inequality that $Z$ is smaller than a logistic process without negative jumps, upper bounded by $C_3: = a_1/\beta_{11}\vee Z_0$. 
Let $\alpha>1$. 
From the previous computations, we deduce that, for any $i \leq \ln^\alpha K -1$, 
\begin{equation} \label{maj_one_step} \E \Big[ \Delta^{(K)}_{T_i}\big| ((T_k,\Theta_k), 1 \leq k \leq i)\Big]\leq C_3^2\Big(\ln^\alpha K\frac{C_1}{4K} + \frac{C_2T_{\ln^\alpha K}}{K}\Big)\Big( \sum_{k=0}^{\ln^\alpha K-1}Z_{T_{k}}^{-1} \Big)^2. \end{equation}
We will now use Cram\'er's theorem to estimate the possible range of values taken by $T_{\ln^\alpha K}$. The rate function of an exponential variable of parameter $r$ is
$\Lambda^*(x)=rx-1-\ln(rx)$, convex and minimal at $x=1/r$, so that for $\eps>0$ and large $K$,
 the infima on $[0,\ln^{-\eps}K]$ and $[\ln^{\eps}K,\infty)$ are attained at the boundary:
$\Lambda^*(\ln^{\eps}K)\sim r\ln^{\eps}K$ and $\Lambda^*(\ln^{-\eps}K)\sim\eps\ln\ln K$.
Using Cram\'er's theorem we thus obtain, for $K$ large enough,
\begin{align}\label{maj_LlnalphaK}
       \P\big( \big\{ \ln^{\alpha-\eps} K \geq  T_{\ln^\alpha K}\big\} \cup \big\{  T_{\ln^\alpha K} \geq \ln^{\alpha+\eps}K \big\} \big) \leq e^{-\ln^\alpha K}.
   \end{align}
We now establish an estimate on the sum of $Y_k:=Z_{T_k}^{-1}$. 
From \eqref{Z_explicite}, we have for $k \geq 1$,
\begin{align*}
Y_k &= e^{-K_{T_k}}Y_0 + \beta_{11} e^{-K_{T_k}} \int_0^{T_k}e^{K_s}ds\\
&= e^{-(K_{T_k}-K_{T_{k-1}})}\Big(e^{-K_{T_{k-1}}}Y_0 + \beta_{11} e^{-K_{T_{k-1}}} \int_0^{T_{k-1}}e^{K_s}ds\Big) + \beta_{11}e^{-K_{T_k}}\int_{T_{k-1}}^{T_k}e^{K_s}ds\\
&= \frac{e^{-a_1(T_k-T_{k-1})}}{\Theta_k}Y_{k-1} + \frac{\beta_{11}}{a_1}\cdot\frac{1-e^{-a_1(T_k-T_{k-1})}}{\Theta_k}=:A_kY_{k-1}+ B_k,
\end{align*}
where $((A_k,B_k),k \geq 1)$ are identically distributed independent positive random variables.

Let $\Psi(\lambda)=a_1\lambda+r\E\big[1-(\Theta^{(1)})^{-\lambda}\big]$ so that $\E[e^{-\lambda K_t}]= e^{-t\Psi(\lambda)}$. Under Assumption \ref{ass:a_bar_positive}
and if $\E[(\Theta^{(1)})^{-p_0}]<\infty$, the function $\Psi$ is $ \mathcal{C}^{2}$ and strictly concave on
$[0,p_0)$ with $\Psi(0)=0$ and $\Psi'(0^{+})=\bar{a}_1>0$; hence there exists
$p\in(0,\min(p_0,1)]$ such that $\Psi(p)>0$, and we choose $\gamma:=1/p\geq 1$. For this $p$,
$$
\rho:=\E\big[A_1^{p}\big]=\frac{r}{r+a_1p}\E\big[(\Theta^{(1)})^{-p}\big]<1,
\qquad
b:=\E\big[B_1^{\,p}\big]<\infty,
$$
where the inequality $\rho<1$ holds as $\Psi(p)>0$.
Iterating the recursion and using the subadditivity of $u\mapsto u^{p}$ for $p\le1$ and
the independence of the $(A_n,B_n)$, we get
$$
\E\big[Y_k^{p}\big]\le\rho^{k}Y_0^{\,p}+b\sum_{j=1}^{k}\rho^{k-j}
\le M:=Y_0^{p}+\frac{b}{1-\rho}\qquad\text{for all }k\geq 0 .
$$
Finally, since $\sum_ku_k\le\big(\sum_ku_k^{p}\big)^{1/p}$ for non-negative $u_k$, we get for all $n\geq 1$ and $C>0$
$$
\P\left(\frac{1}{n^{\gamma}}\sum_{k=0}^{n-1}\big(Z^{(1)}_{T_k}\big)^{-1}>C^{\gamma}\right)
\le\P\left(\frac1n\sum_{k=0}^{n-1}Y_k^{\,p}>C\right)\le\frac MC
$$
by Markov's inequality. Hence, for any $\eta>0$, taking $\mathscr{C}_\eta=(M/\eta)^{2\gamma}$ and
$n=\lceil\ln^{\alpha}K\rceil$,

\begin{equation} \label{maj_fn_levy}  \limsup_{K \to \infty} \P \Big( \Big((\ln^\alpha K)^{-\gamma}\sum_{k=0}^{\ln^\alpha K-1} \big(Z_{T_k}^{(1)}\big)^{-1}\Big)^2> \mathscr{C}_\eta \Big)\leq \eta. \end{equation}
Combining \eqref{maj_one_step} with \eqref{maj_LlnalphaK} and \eqref{maj_fn_levy}, we obtain that for any $\eta>0$, there exists $\mathscr{D}_\eta<\infty$ such that for $\mathscr{B}_K:= \mathscr{D}_\eta (\ln K)^{3\alpha\gamma+\eps}/K$ and $K$ large enough
\begin{equation*} \label{maj_diff_ZTi} 
\P\big(E_K\big)=\P \Big( \sup_{i \leq \ln^\alpha K} \E \big[ \Delta^{(K)}_{T_i}\big|\mathcal{G}\big]>\mathscr{D}_\eta \frac{\ln^{3\alpha\gamma+\eps}K}{K} \Big) \leq \eta + e^{-\ln^\alpha K}\leq 2\eta.  \end{equation*}

For the second term in \eqref{eq:separationEK} with $\mathscr{B}_K=\mathscr{D}_\eta {\ln^{3\alpha\gamma+\eps}K}/{K} $, Markov's inequality entails
\begin{align*} 
\ind_{E_K^C} \sum_{i=1}^{\ln^\alpha K} \P \big( \Delta_{T_i}^{(K)} >\mathscr{A}_K \big| \mathcal{G}\big)& \leq \ind_{E_K^C } \ln^\alpha K\mathscr{A}_K^{-1}\ \sup_{i\leq \ln^\alpha K}\E \big[  \Delta_{T_i}^{(K)} \big| \mathcal{G}\big]\\
& \leq\ln^\alpha K  \mathscr{D}_\eta \frac{\ln^{3\alpha\gamma+\eps}K}{K \mathscr{A}_K}
\leq\mathscr{D}_\eta \frac{\ln^{4\alpha\gamma+\eps}K}{K\mathscr{A}_K}.
\end{align*}
For $K$ large enough, the choice $\mathscr{A}_K= \ln^{5\alpha\gamma +\eps}K/K$ finally leads to 
\begin{equation} \label{maj_Delta}
\P \Big( \sup_{i \leq \ln^\alpha K} \Delta_{T_i}^{(K)}>\frac{\ln^{5\alpha \gamma +\eps}K}{K}\Big)\leq 2\eta + \frac{\mathscr{D}_\eta}{\ln^{\alpha\gamma} K}.
\end{equation}

\item \textit{Between two catastrophes.} To complete the proof, we bound the difference $\Delta^{(K)}$ between two successive jumps. Let $\mathscr{E}_K>0$. As inter-arrival times are independent exponential random variables with parameter $r$, we have
\begin{align*}
\P \Big( \exists i \leq \ln^\alpha K, T_i- T_{i-1} \geq \mathscr{E}_K \Big) = 1 - \Big( 1 - e^{-r \mathscr{E}_K} \Big)^{\ln^\alpha K}\leq  1 - e^{-2e^{-r \mathscr{E}_K}\ln^\alpha K },
\end{align*}
where the last inequality holds for $\mathscr{E}_K>0$ such that $\exp(-r\mathscr{E}_K) \leq \ln 2 / 2$ (as for $x\in[0, \ln 2/2]$, $1-x\geq e^{-2x}$).
Hence for $\mathscr{E}_K = 2\alpha (\ln\ln K)/r$, we obtain for $K$ large enough,
\begin{align} \label{maj_diff_Ti}
\P \Big( \exists i \leq \ln^\alpha K, T_i- T_{i-1} \geq \frac{2\alpha \ln\ln K}{r} \Big) \leq  1 - e^{-2\ln^{-\alpha} K }\leq \frac{2}{\ln^\alpha K}.
\end{align}
Hence the last step consists in proving that if 
$ \Delta_0^{(K)}\leq \mathscr{A}_K$
and no jump occurs before time $\mathscr{E}_K = 2\alpha (\ln\ln K) /r$, then
$ \sup\{\Delta_t^{(K)},\ t < \mathscr{E}_K \wedge T_1 \}$
stays small with a probability close to $1$. To do so, we will reuse some computations of the proof of Lemma \ref{lem_Kurtz}.
To be more precise, let us introduce the event
$$ \mathcal{G}_{2,K}:= \big\{ \sup_{t \leq \mathscr E_K} Z^{(1,K)}_t \leq  2a_1/\beta_{11} \big\}. $$
As $Z^{(1,K)}$ is smaller than the birth-death process without catastrophes (with $r=0$), classical large deviation results see references before \eqref{eq:gde_dev}) entail the existence of a finite $C$ such that 
\begin{equation}\label{maj_gdes_dev}
    1- \P(\mathcal{G}_{2,K})\leq e^{-CK}.
\end{equation}
We may then directly apply \eqref{lemma_conv_Gronw} and the first inequality of \eqref{Doob_Wt} with $k=2$ to obtain
\begin{align*}  
    \P \Big( \sup_{t < \mathscr{E}_K \wedge T_1} \Delta_t^{(K)}> e^{2 \mathcal{M}_2\mathscr{E}_K} \Big(\sqrt{\mathscr{A}_K}+  K^{\eps-\frac{1}{2}} \Big)^2\Big|  \mathcal{G}_{2,K}, \Big\{\Delta_0^{(K)}
    \leq \mathscr{A}_K \Big\}\Big) \leq \frac{C_2 \mathscr{E}_K}{K^{2\eps}}.
\end{align*}
Combining this with \eqref{maj_Delta}, \eqref{maj_diff_Ti}, \eqref{maj_gdes_dev} entails that for any $\eta>0$ and $K$ large enough,
\begin{align*}
      &\P \Big( \sup_{t \leq T_{\ln^\alpha K}} \Delta_t^{(K)}> e^{2 \mathcal{M}_2\mathscr{E}_K} \Big(\sqrt{\mathscr{A}_K}+  K^{\eps-\frac{1}{2}} \Big)^2\Big)\\ &\leq  2\eta + \frac{\mathscr{D}_\eta}{\ln^{\alpha\gamma} K}+\frac{2}{\ln^\alpha K} + \ln^\alpha K \Big( e^{-CK} +  \frac{C_2 \mathscr{E}_K}{K^{2\eps}} \Big).
\end{align*}
Adding \eqref{maj_LlnalphaK} ends the proof of Proposition \ref{prop_cvglongue}.
\end{enumerate}
\end{proof}

\subsection{Proofs of the results of Section \ref{SDElongtime}}
\subsubsection{Proof of Proposition \ref{3cas}}

We first recall the three properties at the core of the next lemma. The first has been stated in Proposition \ref{prop_monomorph}, the second is the law of large numbers applied to Lévy processes, and the last has been stated in Proposition 4 in \cite{palau2018branching}.

For $\alpha\in\R$, $\beta>0$ and $i \in \{1,2\}$, the SDE
\begin{equation*}\label{eq:onetype}
dW_t^{\alpha,\beta}=W_t^{\alpha,\beta}(\alpha-\beta W_t^{\alpha,\beta})dt+\int_0^1(\theta^{(i)}-1)W_{t-}^{\alpha,\beta}S_i(dt,d\theta^{(i)}),\qquad W_0>0,
\end{equation*}
has the explicit solution $W_t^{\alpha,\beta}=W_0e^{L_t}\big(1+\beta W_0\int_0^te^{L_s}ds\big)^{-1}$, where $L_t=(\alpha-a_i)t+K^{(i)}_t$ and $K^{(i)}$ is the Lévy process defined in \eqref{defK}. Besides, when $t$ tends to infinity,
\begin{equation}\label{eq:cesaro}
L_t/t\to\alpha+r\E[\ln\Theta^{(i)}] \quad \text{a.s.},\quad
\text{and}\quad \frac1t\int_0^tW_s^{\alpha,\beta}ds\to\frac{(\alpha+r\E[\ln\Theta^{(i)}])\vee 0}{\beta}\quad \text{a.s.}
\end{equation}

Let $i \in \{1,2\}$. The objective is to understand the behaviour of the process $Z^{(i)}$ during its excursions below a small positive number $\eta$. In particular, we want to know if it is more likely to go to $0$ or to reach $\eta$ with a positive probability. It will allow us to discriminate between the different possible long time behaviours of the population process $Z$ solution to \eqref{eqdep}.

\begin{lemma}\label{lem:comp}
Let $Z_0\in(\R_+^{*})^2$, $i\ne j\in\{1,2\}$, and $\eta\ge0$. For $t\geq 0$, define
$$
\mathcal{E}_t^{(i)}:= \{Z^{(i)}_s\le\eta, \forall s\le t\},\quad
\ell_i^{+}(\eta):=\bar a_i-\frac{\beta_{ij}}{\beta_{jj}}\big(\bar a_j-\eta\,\beta_{ji}\big)^{+},
\quad
\ell_i^{-}(\eta):=\bar a_i-\beta_{ii}\eta-\frac{\beta_{ij}}{\beta_{jj}}\,\bar a_j.
$$
Then, on a common probability space:
\begin{enumerate}[leftmargin=*,label=\textup{(\roman*)}]
\item \emph{(Upper bound).} There exist two random processes $\widehat L^{(i)}$ and $\widehat Z^{(i)}$ such that, 
%on the event
%$\big\{Z^{(i)}_s\leq\eta,  \forall s\le t\big\}$,
$$
\ind_{\mathcal{E}_t^{(i)}}Z^{(i)}_t \leq\widehat Z^{(i)}_t\leq Z^{(i)}_0e^{\widehat L^{(i)}_t},
\qquad
\lim_{t\to\infty}\frac{\widehat L^{(i)}_t}{t}\ind_{\mathcal{E}_t^{(i)}}=\ell_i^{+}(\eta)\quad\text{a.s.}
$$
\item \emph{(Lower bound).} There exist two random processes $\widecheck L^{(i)}$ and $\widecheck Z^{(i)}$ such that
%on the event
%$\big\{Z^{(i)}_s\le\eta, \forall s\le t\big\}$,
\[
Z^{(i)}_t \geq \ind_{\mathcal{E}_t^{(i)}}\widecheck Z^{(i)}_t = Z^{(i)}_0e^{\widecheck L^{(i)}_t},
\qquad
\lim_{t\to\infty}\frac{\widecheck L^{(i)}_t}{t}\ind_{\mathcal{E}_t^{(i)}}=\ell_i^{-}(\eta)\quad\text{a.s.}
\]
\end{enumerate}
Moreover, these processes satisfy the following properties:
\begin{enumerate}[leftmargin=*,label=\textup{(\alph*)}]
\item If $\ell_i^{+}(\eta)<0$, then $\widehat Z^{(i)}_t\to0$ a.s.,
and for every $p<1$ there exists $z^{*}>0$ such that
$$\P_{(z^{(i)},z^{(j)})}\big(Z^{(i)}_t\to0\big)>p,\quad \text{for all }z^{(i)}\le z^{*}\text{ and }z^{(j)}\geq\eta.$$
\item If $\ell_i^{-}(\eta)>0$, then $\widecheck Z^{(i)}_t\to+\infty$, hence, for all $z^{(i)},z^{(j)}>0$,
$$\P_{(z^{(i)},z^{(j)})}\big(\exists\,t\ge0:\ Z^{(i)}_t\geq\eta\big)=1.$$
\end{enumerate}
Consequently, for $\eta$ small enough:
\begin{enumerate}[leftmargin=*,label=\textup{(\alph*)},start=3]
\item If $\mathfrak f_{ij}<0$, then for every $z^{(j)}_0\in(0,\bar a_j/\beta_{jj}]$ and every $p<1$ there is $z^{(i)}_0>0$ such that 
$$\P_{(z^{(i)},z^{(j)})}\big(Z^{(i)}_t\to0\big)>p,\quad \text{for all }z^{(i)}\leq z^{(i)}_0\text{ and }z^{(j)}\geq z^{(j)}_0.$$
\item If $\mathfrak f_{ij}>0$, then, for all $z^{(i)}_0>0$ et $z^{(j)}_0\geq0$,
$$\P_{(z^{(i)}_0,z^{(j)}_0)}\big(Z^{(i)}_t\to0\big)=0.$$ 
\end{enumerate}
\end{lemma}

\begin{proof}
\noindent (i) (Upper bound). 
On the event $ \mathcal{E}_t^{(i)}$, the death rate of type-$j$ individuals contains the term $\beta_{ji}Z^{(i)}_s\le\beta_{ji}\eta,$ which implies that $Z^{(j)}_t \geq \widecheck Z^{(j)}_t$ on $\mathcal{E}_t^{(i)}$, where
\begin{equation}\label{eq:rival_lower}
\widecheck Z^{(j)}_t = \widecheck Z^{(j)}_0 + \int_0^t\widecheck Z^{(j)}_s\big(a_j-\eta\beta_{ji}-\beta_{jj}\widecheck Z^{(j)}_s\big)ds
+\int_0^t\int_0^1(\theta^{(j)}-1)\widecheck Z^{(j)}_{s-}S_j(ds,d\theta^{(j)}).
\end{equation}
By \eqref{eq:cesaro} we know that
\begin{equation}\label{eq:cesaro_lower}
\frac{1}{t}\int_0^t\widecheck Z^{(j)}_s\,ds\ \xrightarrow[t\to\infty]{}\
\frac{(\bar a_j-\eta\beta_{ji})^{+}}{\beta_{jj}}\qquad\text{a.s.}
\end{equation}
On $\mathcal{E}_t^{(i)}$, according to \eqref{eq:rival_lower}, the death rate of type $i$ individuals satisfies for all $s\leq t$
$$\mu_i+\beta_{ii}Z^{(i)}_s+\beta_{ij}Z^{(j)}_s\ge\mu_i+\beta_{ii}Z^{(i)}_s+\beta_{ij}\widecheck Z^{(j)}_s.$$
Hence, on $\mathcal{E}_t^{(i)}$, the process $Z^{(i)}$ is dominated by the solution $\widehat Z^{(i)}$ of
\begin{equation*}\label{eq:focal_upper}
\widehat Z^{(i)}_t= Z^{(i)}_0 + \int_0^t\widehat Z^{(i)}_s\Big(a_i-\beta_{ii}\widehat Z^{(i)}_s
-\beta_{ij}\widecheck Z^{(j)}_s\Big)ds+\int_0^t\int_0^1(\theta^{(i)}-1)\widehat Z^{(i)}_{s-}S_i(ds,d\theta^{(i)}).
\end{equation*}
Applying Itô's formula to $\widehat Z^{(i)}e^{-\widehat L^{(i)}}$ with
\begin{equation*}\label{eq:Lhat}
\widehat L_t^{(i)}:=K^{(i)}_t-\beta_{ij}\int_0^t\widecheck Z^{(j)}_sds
\end{equation*}
gives,
$$
\widehat Z^{(i)}_t=\frac{Z^{(i)}_0e^{\widehat L_t^{(i)}}}
{1+\beta_{ii}Z^{(i)}_0\int_0^te^{\widehat L^{(i)}_s}ds} \leq Z^{(i)}_0\,e^{\widehat L^{(i)}_t}.
$$
Finally, as $K^{(i)}_t/t\to\bar a_i$ as $t\to\infty$, using \eqref{eq:cesaro_lower}, we get
$$
\frac{\widehat L_t^{(i)}}{t}\ \xrightarrow[t\to\infty]{}\
\bar a_i-\beta_{ij}\frac{(\bar a_j-\eta\beta_{ji})^{+}}{\beta_{jj}}\ =\ \ell_i^{+}(\eta).
$$
\noindent {(ii) (Lower bound).} 
%exchange the roles of the two couplings.
Dropping the competition exerted by $i$ on $j$, $Z^{(j)}$ is dominated by
\begin{equation} \label{upper_bound_no_compet}
\widehat{Z}^{(j)}_t=  Z^{(j)}_0 + \int_0^t\widehat{Z}^{(j)}_s\big(a_j-\beta_{jj}\widehat Z^{(j)}_s\big)ds+\int_0^t\int_0^1(\theta^{(j)}-1)\widehat Z^{(j)}_{s-}S_j(ds,d\theta^{(j)}),
\end{equation}
satisfying $t^{-1}\int_0^t\widehat Z^{(j)}_sds\to\bar a_j/\beta_{jj}$ as $t\to\infty$ by \eqref{eq:cesaro}. On $\mathcal E^{(i)}_t$, the death rate of a type-$i$ individual is at most $\mu_i+\beta_{ii}\eta+\beta_{ij}\widehat Z^{(j)}_s$ for $s\leq t$, so that $Z^{(i)}\geq\widecheck Z^{(i)}$ with
\begin{equation}\label{eq:focal_lower}
\widecheck Z^{(i)}_t:=Z^{(i)}_0e^{\widecheck L^{(i)}_t},\qquad
\widecheck L^{(i)}_t:=K^{(i)}_t-\beta_{ii}\eta t-\beta_{ij}\int_0^t\widehat Z^{(j)}_sds,
\end{equation}
and $\widecheck L^{(i)}_t/t\to\bar a_i-\beta_{ii}\eta-\beta_{ij}\bar a_j/\beta_{jj}=\ell_i^{-}(\eta)$ as $t\to\infty$.
\begin{enumerate}[label=(\alph*)]
    \item  Let $p<1$. If $\ell_i^{+}(\eta)<0$, as $Z_0^{(j)}\geq \eta$, we obtain $\widehat L_t^{(i)}\to-\infty$, and $\sup_t e^{\widehat L_t^{(i)}}<\infty$ almost surely. As a consequence, $\widehat Z^{(i)}_t\le Z^{(i)}_0e^{\widehat L_t^{(i)}}$ goes to $0$ as $t$ goes to infinity. Moreover, as the law of $\sup_te^{\widehat L_t^{(i)}}$ does not depend on $Z^{(i)}_0$, there exists $z^*>0$ such that if $Z^{(i)}_0=z^{*}$, we have
$$
\P\left(\sup_{t\geq 0} Z^{(i)}_0e^{\widehat L_t^{(i)}}<\eta\right)> p.
$$
On that event, $\mathcal{E}_t^{(i)}$ holds for all $t\geq 0$, so that $\P_{(z^*, \eta)}(Z^{(i)}\to0)>p$. Monotonicity of the couplings with respect to the initial conditions (smaller $z^{(i)}$ and larger $z^{(j)}$ both help) extends the conclusion to initial conditions satisfying $z^{(i)}\leq z^{*}$ and $z^{(j)}\geq\eta$.
\item If $\ell_i^{-}(\eta)>0$ then $\widecheck L_t^{(i)}\to+\infty$, so that $\widecheck Z^{(i)}_t=Z^{(i)}_0e^{\widecheck L_t^{(i)}}\to+\infty$ when $t\to\infty$. Since $Z^{(i)}_t\ge\widecheck Z^{(i)}_t$ as long as $Z^{(i)}\leq\eta$, the level $\eta$ is reached in finite time almost surely.

\item As $\ell_i^{+}(0)=\mathfrak f_{ij}<0$ and $\eta\mapsto\ell_i^{+}(\eta)$ is continuous, we may choose $\eta>0$ smaller than both $\bar a_j/\beta_{ji}$ and $z^{(j)}_0$ with $\ell_i^{+}(\eta)<0$; (a) applied with this $\eta$ gives the claim, the requirement $z^{(j)}\geq\eta$ being satisfied since $ z^{(j)}_0\geq \eta$.

\item If $z^{(j)}_0=0$ the system is monomorphic and the claim follows from \eqref{eq:cesaro} with $\bar a_i>0$. If $z^{(j)}_0>0$, as $\ell_i^{-}(0)=\mathfrak f_{ij}>0$, by the continuity of $\ell_i^{-}$, there exists $\eta>0$ such that $\ell_i^{-}(\eta)>0$; by (b), $Z^{(i)}$ reaches $\eta$ a.s. from every initial condition, so that $\limsup_tZ^{(i)}_t\geq\eta>0$ a.s.
\end{enumerate}
\end{proof}

Recall the definitions of $M_1,M_2$ for the deterministic system in \eqref{def_LV_det}. Let
    \begin{equation}
    \label{subsetArandom}
    \mathscr{A}=\left\{(x,y)\in (\R_+)^2,~M_1(x,y)\geq 0 \text{ or }M_2(x,y)\geq 0\right\} \text{ and }\mathscr{A}^*=\mathscr{A}\cap (\R_+^*)^2.
    \end{equation}
The sets $\mathscr{A}$ and $\mathscr{A}^*$ are positively invariant under the dynamics \eqref{def_LV_det} and \eqref{eqdep}. That is, for any initial condition in $\mathscr{A}$ (resp. in $\mathscr{A}^*$), the solution to \eqref{def_LV_det} or \eqref{eqdep} remains in $\mathscr{A}$ (resp. in $\mathscr{A}^*$) a.s.

\begin{lemma}
\label{irred}
    Suppose that $Z_0\in\mathscr{A}$ almost surely and that one of the two following conditions is satisfied
\begin{align}&\label{2.sanssaut} \frac{a_1}{a_2} \leq \frac{\overline{a}_1}{\overline{a}_2} \quad \text{and} \quad \mathfrak{f}_{12}<0<\mathfrak{f}_{21}\\
&\label{2.avecsauts} \frac{a_1}{a_2} > \frac{\overline{a}_1}{\overline{a}_2}\end{align}
    Then, for any $x_0,y_0>0$ small enough, there exist $t_0\geq1$ and $\delta>0$ such that
    $$\P_z\big(\exists t\in[0,t_0],~Z_{t}^{(1)}\leq x_0 \text{ and }Z_{t}^{(2)}\geq y_0\big)\geq\delta
    \qquad\text{for every }z\in\mathscr{A}\text{ with }z^{(2)}\geq y_0.$$
\end{lemma}

Notice that by symmetry, the result also holds if we exchange $1$ and $2$.
Every case of Proposition \ref{3cas} relies on the same recurrence argument stated in the following lemma: a region that is reached almost surely from anywhere is visited at arbitrarily large times, and each visit gives an independent chance of entering a prescribed target region.

\begin{lemma}\label{lem_recurrence}
Let $B,C\subseteq\R_+^2$ be Borel sets and let $t_0\geq1$, $\delta>0$ be such that
\begin{equation}\label{hyp_recurrence}
\P_z\big(\exists t\geq0,\ Z_t\in B\big)=1\quad\forall z\in\mathscr{A}^*,
\quad
\P_z\big(\exists t\in[0,t_0],\ Z_t\in C\big)\geq\delta\quad\forall z\in B\cap\mathscr{A}^* .
\end{equation}
Then, $\P_z\big(\exists t\geq0,\ Z_t\in C\big)=1$ for every $z\in\mathscr{A}^*$.
\end{lemma}

\begin{proof}
Define $\sigma_1:=\inf\{t\geq0,\ Z_t\in B\}$ and, for $i\geq1$, $\sigma_{i+1}:=\inf\{t\geq\sigma_i+t_0,\ Z_t\in B\}$. As $\mathscr{A}^*$ is positively invariant, the first assumption in \eqref{hyp_recurrence} and the strong Markov property give $\sigma_i<\infty$ almost surely for every $i\geq1$; moreover, $\sigma_i\geq(i-1)t_0\to\infty$. Let $A_i:=\{\exists t\in[\sigma_i,\sigma_i+t_0],\ Z_t\in C\}$ and $\mathcal{G}_i:=\mathcal{F}_{\sigma_i}$. Since $\sigma_{i+1}\geq\sigma_i+t_0$, the event $A_i$ belongs to $\mathcal{G}_{i+1}$, and the strong Markov property at $\sigma_i$ together with the second assumption in \eqref{hyp_recurrence} yields $\P(A_i|\mathcal{G}_i)\geq\delta$ almost surely. Hence $\sum_{i\geq1}\P(A_i|\mathcal{G}_i)=\infty$ almost surely, and L\'evy's extension of the Borel--Cantelli lemma gives $\P(\limsup_iA_i)=1$. In particular $Z$ enters $C$ almost surely.
\end{proof}

\begin{proof}[Proof of Lemma \ref{irred}]
    We first assume that \eqref{2.sanssaut} holds. This latter may be rewritten
    \begin{equation*}
    \frac{a_1}{a_2}\leq \frac{\overline{a}_1}{\overline{a}_2}\quad \text{ and } \quad \frac{\overline{a}_1}{\overline{a}_2}<\frac{\beta_{12}}{\beta_{22}} \wedge \frac{\beta_{11}}{\beta_{21}}. 
    \end{equation*}
   Besides, 
   $$  f_{12}<0<f_{21} \quad  \Longleftrightarrow \quad \frac{a_1}{a_2}<\frac{\beta_{12}}{\beta_{22}} \wedge \frac{\beta_{11}}{\beta_{21}}. $$
  We deduce that \eqref{2.sanssaut} implies $ f_{12}<0<f_{21}$. Hence, applying Proposition \ref{compLVdet} (case $\exclref{2}$), we get that for any initial condition in $\mathscr{A}\setminus\{y=0\}$, the solution of the deterministic system is converging to $(0,a_2/\beta_{22})$. Let $\varphi_t$ be the flow associated to the solution of \eqref{def_LV_det}. Using that simple convergence to a point on a compact set implies uniform convergence, we get that there exists $t_0$ such that for any $t\geq t_0$, $\varphi_t(\mathscr{A}\cap\{y\geq y_0\})$ is arbitrarily close to $(0,a_2/\beta_{22})\in\mathscr{A}\cap\{x<x_0,~y>y_0\}$ for $x_0>0$ and $y_0<a_2/\beta_{22}$.
    If no jump occurs before $t_0$, the random system follows the ODE \eqref{def_LV_det}, so that
    $$        \P(Z_{t_0}^{(1)}\leq x_0 \text{ and }Z_{t_0}^{(2)}\geq y_0|Z_0^{(2)}\geq y_0)\\
        \geq  \P(\text{no jump in }[0,t_0])=e^{-rt_0}>0.$$
We now assume that \eqref{2.avecsauts} holds. The idea is to show that there is a non-zero probability that many jumps less favorable for $\Zus$ than for $\Zds$ occur in a short time and then no jump occurs until $\Zds$ reaches $y_0$. Recall that a jump with survival probabilities $(\eta_1,\eta_2)$ is favorable for $Z^{(i)}$ if \eqref{ADVi} is satisfied for $\Theta\equiv (\eta_1,\eta_2)$. As \eqref{2.avecsauts} is equivalent to $\ADVref{2}$, without loss of generality, we define $x_0$ and $y_0$ such that
    \begin{equation*} \label{choix_zlim_eps}
        \Esp[a_1\ln \Theta^{(2)} - (a_2-\beta_{21}x_0-\beta_{22}y_0)\ln \Theta^{(1)}] >0.
    \end{equation*}
    Note that this ensures $a_2-\beta_{21}x_0-\beta_{22}y_0 >0$. It implies that we can find $0<\eta_1,\eta_2<1$ s.t.
    \begin{equation}
        \label{eta2}
        \frac{a_1}{a_2-\beta_{21}x_0-\beta_{22}y_0} < \frac{\ln\eta_1}{\ln\eta_2} \quad \text{and} \quad  \Prob\left( \Theta^{(1)}\leq \eta_1 \text{ and } \Theta^{(2)}\geq\eta_2\right) >0. 
    \end{equation}
   Recall the definition of the sequence $(T_i, \Theta_i^{(1)},\Theta_i^{(2)})$ page \pageref{def_triplet} encoding the occurrence time of the $i$th catastrophe, and the associated survival probabilities.
   Let $(N_t,\ t\geq 0)$ be a Poisson process with intensity $r$ counting the number of catastrophes:
\begin{equation*}
    \label{def_N} \{ N_t=k \} \Longleftrightarrow \{ T_k \leq t < T_{k+1} \}, \quad k \geq 0.
\end{equation*}
Let us introduce for $ t_0\geq 1$ and $k \geq 0$ the event
\begin{equation*}
\label{manyjumps}
        A_{k,t_0} :=\left\{N_{1}=N_{t_0}=k\right\}  \cap\left\{\forall j\in[1,N_{1}], \Theta_{j}^{(1)}\leq \eta_1,\Theta_{j}^{(2)}\geq\eta_2\right\}.
\end{equation*}
Under $A_{k,t_0}$, there are $k$ jumps before time $1$, all of which are beneficial to type $2$ (as $\ADVref{2}$ is satisfied), and there are no jumps in $[1, t_0]$. Then, we prove the existence of $t_0 \geq 1$ and $k \in \N$ such that
\begin{equation}\label{eq:theclaim}
 \{Z_0^{(2)}> y_0\} \cap A_{k,t_0}  \subset \{Z_0^{(2)}\geq y_0\} \cap \{ \exists 0 \leq t \leq t_0,\ Z^{(1)}_t \leq x_0 \quad \text{and} \quad Z_t^{(2)}\geq y_0 \}.
\end{equation}
Note that $A_{k,t_0}$ has a positive probability for any $k$ and $t_0$. Thus, it remains to prove \eqref{eq:theclaim}. Let us define the stopping time $T^{(2)}:=\inf\{t\geq1, Z^{(2)}_t\geq y_0\}.$ Then, if $Z_0^{(2)}> y_0$ and $A_{k,t_0}$ is satisfied, we have that for all $1 \leq t \leq t_0$:
\begin{equation*}
\label{ublbcondAi}
\displaystyle\begin{array}{ll}
    Z^{(1)}_{1}\leq \frac{a_1}{\beta_{11}}\eta_1^ke^{a_1},&
    Z^{(1)}_{T^{(2)}\wedge t}\leq \frac{a_1}{\beta_{11}}e^{a_1 t}\eta_1^k,\\
    
  Z^{(2)}_{1}\geq  y_0e^{-\beta_{21}\frac{a_1}{\beta_{11}}}\eta_2^k,\quad
    &Z^{(2)}_{T^{(2)}\wedge t}\geq y_0\big[1 \wedge \big(e^{-\beta_{21}\frac{a_1}{\beta_{11}}} \eta_2^k e^{(a_2-\beta_{21}\frac{a_1}{\beta_{11}}e^{a_1 t}\eta_1^k-\beta_{22}y_0)(t-1)}\big)\big],
\end{array}\end{equation*}
using that $Z_0\in\mathscr{A}$ implies that $Z_0^{(1)}<a_1/\beta_{11}$. Thus, we want to choose $k$ and $t_0\geq 1$ such that

\begin{equation} \label{choix_k_t}
    \frac{a_1}{\beta_{11}}e^{a_1t_0}\eta_1^k \leq x_0 \quad \text{and} \quad
    \eta_2^ke^{-\beta_{21}\frac{a_1}{\beta_{11}}} e^{(a_2-\beta_{21}x_0-\beta_{22}y_0)(t_0-1)} \geq 1.\end{equation}
Note that such a choice ensures $T^{(2)}\leq t_0$ under $A_{k,t_0}$. Rearranging \eqref{choix_k_t}, we get the system

$$\displaystyle\left\{\begin{array}{l}
    \displaystyle\frac{k}{t_0} \geq -\frac{a_1}{\ln\eta_1} + \frac{C(a_1,\beta_{11},x_0,\eta_1)}{t_0},   \\
    \displaystyle\frac{k}{t_0} \leq -\frac{a_2-\beta_{21}x_0-\beta_{22}y_0}{\ln\eta_2}+ \frac{C(a_1,a_2,\beta_{11},\beta_{21},x_0,y_0,\eta_1,\eta_2)}{t_0}.
\end{array}\right.$$
As $\eta_1,\eta_2$ satisfy \eqref{eta2}, we can take $t_0$ such that the gap between the two bounds is at least $1/t_0$, then there exists at least one integer value for $k$ such that \eqref{choix_k_t} is satisfied.

In both cases the lower bound obtained, $e^{-rt_0}$ under \eqref{2.sanssaut} and $\P(A_{k,t_0})$ under \eqref{2.avecsauts}, depends only on $r,t_0,k,\eta_1,\eta_2$ and not on $z$. This gives the uniform constant $\delta$ of the statement.
\end{proof}

\begin{proof}[Proof of Proposition \ref{3cas}]\label{proof:3cas}

The case (\ref{coex}) is a direct consequence of Lemma \ref{lem:comp}(d). The strategies to prove the results in both cases (\ref{excli}) and (\ref{excl}) are similar. We show that, for any initial condition and any $p\in(0,1)$, the system will reach infinitely many times a subset of the state space from which one type will vanish with probability at least $p$  as $t$ goes to $\infty$.\\

We start with \eqref{excli}. Suppose without loss of generality that $\mathfrak{f}_{12}<0<\mathfrak{f}_{21}$. Note that Lemma \ref{irred} applies here without any additional assumption.
Let $p\in (0,1)$ and recall the definition of $\mathscr{A}^*$ in \eqref{subsetArandom}. 
Thanks to Lemmas \ref{lem:comp}(c), \ref{lem:comp}(d) and \ref{irred}, we can find $\eta>0$, $z^{(1)}_{\lim}>0$ and $t_0\geq 1$ such that for any initial conditions $(z^{(1)} ,z^{(2)} )\in \mathscr{A}^*$
\begin{enumerate}
        \item As $\ell_2^{-}(\eta)=\mathfrak{f}_{21}-\beta_{22}\eta>0$ for $\eta$ small enough, we have,
        \begin{equation} \label{choix_eps_prop} \Prob_{(z^{(1)},z^{(2)})}\left(\exists t\geq 0, \Zd \geq \eta\right)=1 \quad  \text{(Lemma \ref{lem:comp}(b)),}
        %, $\eta$ small enough for $\ell_2^{-}(\eta)>0$)},
        \end{equation}
        \item If $z^{(1)}\leq \zl$ and $z^{(2)}\geq \eta$ we have,
          \begin{equation} \label{choix_z1lim_prop}\Prob_{(z^{(1)},z^{(2)})}(\Zu \to 0)>p \quad  \text{(Lemma \ref{lem:comp}(c))},\end{equation}
        \item If $z^{(2)}\geq \eta$ we have
            \begin{equation}
                \label{choix_t0_prop}
                \Prob_{(z^{(1)},z^{(2)})}(\exists t\in[0,t_0],~Z_{t}^{(1)}\leq z^{(1)}_{\lim} \text{ and }Z_{t}^{(2)}\geq \eta)>0\quad\text{(Lemma \ref{irred})}.
            \end{equation}
    \end{enumerate}
    Lemma \ref{lem_recurrence}, applied with $B:=\R_+\times[\eta,+\infty)$ and $C:=[0,z^{(1)}_{\lim}]\times[\eta,+\infty)$ --- whose two assumptions
    \eqref{hyp_recurrence} are precisely \eqref{choix_eps_prop} and \eqref{choix_t0_prop} ---
    shows that $Z$ enters $C$ almost surely. The strong Markov property and \eqref{choix_z1lim_prop} then conclude.\\

It remains to treat the case (\ref{excl}). From the proof of Lemma \ref{irred}, we know that different events (many unfavorable jumps or no jump for a sufficiently long time) can drive one population to extinction depending on the parameters. In the following, we will consider both types of events leading to extinction: one for each population.

We first show that the two processes cannot go extinct together. Let $\eta>0$ be such that
\begin{equation}
\label{eta}
\left\{\begin{array}{ll}
   \overline{a}_1 -\eta(\beta_{11}+\beta_{12})  &>0,  \\
    \overline{a}_2 -\eta(\beta_{21}+\beta_{22})  &>0.
\end{array}\right.
\end{equation}

Similarly as in the proof of Lemma \ref{lem:comp}(d), we now prove that for any $z\in(0, \eta]^2$,
\begin{equation}
\label{Tetafinips}
    \Prob_{z}\left(\exists t>0 , (\Zu,\Zd)\notin [0,\eta]^2\right)=1,
\end{equation}
which will imply that $\Prob_{z}(\Zu \to 0 \text{ and }\Zd \to 0\text{ as } t\to \infty) =0$ for any $z\in \mathscr{A}^*$.

Let $(z^{(1)}_0,z^{(2)}_0)\in(0,\eta]^2$.
We first define a lower-bound process $Y^{(1)}$ for $Z^{(1)}$ on $[0,T_\eta]$, where $ T_{\eta}:=\inf\{t\geq 0,(\Zu,\Zd)\notin [0,\eta]^2 \}$. Let $(Y^{(1)}_t, t\geq 0)$ is the solution to
\begin{equation*} \label{eq_coupling_ext}
Y^{(1)}_t = z_0^{(1)}+ \int_0^t Y^{(1)}_s(a_1 -\eta(\beta_{11}+\beta_{12}))ds + \int_0^t\int_0^1(\theta^{(1)}-1)Y^{(1)}_{s-}S_1(ds,d\theta^{(1)}).\end{equation*}
%\begin{equation*}
 %   T_{\eta}:=\inf\{t\geq 0,(\Zu,\Zd)\notin [0,\eta]^2 \}.
%\end{equation*}
By definition,
\begin{equation}
\label{lbnondoublext}
    \ind_{\{T_{\eta}>t\}}Y^{(1)}_t \leq \ind_{\{T_{\eta}>t\}}{Z}^{(1)}_t \quad a.s.
\end{equation}
Recall the definition of $K$ in \eqref{defK}. From \eqref{Z_explicite} with $\beta_{11}=0$, we get $ Y^{(1)}_t=z^{(1)}_0 e^{K^{(1)}_t -\eta(\beta_{11}+\beta_{12})t }$. As $\eta$ satisfies \eqref{eta}, the theory of Lévy processes (see \cite[Theorem 7.2]{kyprianou2006introductory}) implies
$$\lim_{t \to \infty}Y^{(1)}_t=+\infty \quad \textit{a.s.}$$
Using \eqref{lbnondoublext} and the fact that the process $Z^{(1)}$ is almost surely bounded, we get that 
$$\P(T_{\eta}=\infty)\leq\P(\lim_{t \to \infty}{Z}^{(1)}_t=+\infty)=0.$$
This proves \eqref{Tetafinips}.

We now focus on the probability of extinction of one of the two processes. As before, we will show that for any $p\in(0,1)$, we have $\Prob(\Zu \to 0\text{ or }\Zd \to 0) \geq p$ using Lemma \ref{lem:comp}(c).

Fix $p\in(0,1)$ and $(z_0^{(1)},z_0^{(2)})\in \mathscr{A}^*$. Let $\eta>0$ be such that \eqref{eta} is satisfied. Recall that in this case we assume $a_1\bar a_2 \neq a_2\bar a_1$, and without loss of generality suppose that $a_1/a_2<\overline{a}_1/\overline{a}_2$.
From Lemma \ref{lem:comp}(c), there exist $z_{\lim}^{(1)}<\eta$, $z_{\lim}^{(2)}>0$ and $t_0 \geq 1$ such that  
$$\left\{\begin{array}{ll}
    \forall z^{(1)}\leq z_{\lim}^{(1)}, z^{(2)}\geq \eta,~
    &\Prob_{(z^{(1)},z^{(2)})}(\Zu \to 0)>p,\\
    \forall z^{(1)}\geq z^{(1)}_{\lim}, z^{(2)}\leq z_{\lim}^{(2)},~
    &\Prob_{(z^{(1)},z^{(2)})}(\Zd \to 0)>p.
\end{array}
\right.$$
As $a_1/a_2<\overline a_1/\overline a_2$, \eqref{2.avecsauts} holds after exchanging $1$ and $2$, so the symmetric version of Lemma \ref{irred} provides $\delta>0$ with
\begin{equation}\label{irred_swapped}
\Prob_{z}\big(\exists t\in[0,t_0],\ \Zu\geq z_{\lim}^{(1)}\text{ and }\Zd\leq z_{\lim}^{(2)}\big)\geq\delta
\quad\text{for every }z\in\mathscr{A}\text{ with }z^{(1)}\geq z^{(1)}_{\lim} .
\end{equation}
Set $C_1:=(0,z_{\lim}^{(1)}]\times[\eta,+\infty)$ and $C_2:=[z_{\lim}^{(1)},+\infty)\times(0,z_{\lim}^{(2)}]$, and apply Lemma \ref{lem_recurrence} with $B:=\R_+^2\setminus[0,\eta)^2$ and $C:=C_1\cup C_2$. Its first assumption is \eqref{Tetafinips}. For the second one, let $z\in B\cap\mathscr{A}^*$: since
$z_{\lim}^{(1)}<\eta$, either $z^{(1)}\leq z^{(1)}_{\lim}$, and then $z^{(2)}\geq\eta$ by definition of $B$, so that $z\in C_1$ and the condition holds; or $z^{(1)}>z^{(1)}_{\lim}$, and \eqref{irred_swapped} gives the condition. Hence $Z$ enters $C_1\cup C_2$ almost surely, and Lemma \ref{lem:comp}(c) ends the proof.
\end{proof}

\subsubsection{Proof of Proposition \ref{ergofaible}}

The proof will rely on the following technical lemma.

\begin{lemma} \label{lnZt/t}
    Under the assumptions of Proposition \ref{ergofaible}, if there exists $\eta>0$ such that
    \begin{equation} \label{above_eps} \P \big( \limsup_t Z^{(i)}_t\geq \eta \big|\limsup_tZ^{(i)}_t>0\big)=1 \end{equation}
and
    $$ \left\{ \begin{array}{ll}
\bar{a}_i-\beta_{ii}\eta>\frac{\beta_{ij}}{\beta_{jj}}\bar{a}_j  & \text{if \eqref{coex} holds}   \\
\bar{a}_i-\beta_{ii}\eta>0 & \text{if \eqref{excl} or \eqref{excli} holds},
\end{array} \right. $$
then
$$  \ind_{\{\limsup_tZ^{(i)}_t>0\}} \lim_{t \to \infty}\frac{\ln Z_t^{(i)}}{t}= 0 \quad \text{almost surely}. $$
\end{lemma}

\begin{proof}
    Let us introduce the sequences of stopping times $(T^{(i,\pm)}_k, k\geq 1)$ as follows:
$$\left\{\begin{array}{ll}
   T^{(i,-)}_1 = \inf\{t\geq 0, Z^{(i)}_t< \eta/2\}, & \\
   T^{(i,-)}_k = \inf\{t\geq T^{(i,+)}_{k-1}, Z^{(i)}_t< \eta/2\}, & k\geq 2,\\
   T^{(i,+)}_k = \inf\{t\geq T^{(i,-)}_{k}, Z^{(i)}_t\geq \eta\}, & k\geq 1.
\end{array}\right.$$
By definition of the process $Z^{(i)}$, $\liminf_{t \to \infty}Z^{(i)}_t=0$ almost surely. Combining this with \eqref{above_eps}, we get that for any $k\geq 1,~T^{(i,\pm)}_k<\infty$ and $\lim_k T^{(i,\pm)}_k=+\infty$ almost surely. As $Z^{(i)}$ is bounded, $\limsup_{t \to \infty}(\ln Z^{(i)}_t/t)=0$.
Since the growth rate of $Z^{(i)}$ is almost surely bounded by $a_i$, for any $k\geq 1$ we have $T_k^{(i,+)}-T_k^{(i,-)}\geq \ln 2/a_i$. Let $I_n^{(i)} = [T^{(i,-)}_n,T^{(i,+)}_{n}]$. Fix $\eta<1$ so that $\ln Z^{(i)}_t\leq 0$ for all $t\in I_n^{(i)}$ for all $n\in\N^*$. We have
\begin{equation}
\label{liminflogZ}
    0\geq\liminf_{t\to \infty}\frac{\ln Z^{(i)}_t}{t}=\liminf_{n\to \infty}\inf_{t\in I_n^{(i)}}\frac{\ln Z^{(i)}_t}{t}\geq \liminf_{n \to \infty}\frac{\inf_{t\in I_n^{(i)}}\ln Z^{(i)}_t}{n\ln 2/{a_i}},
    %=:\liminf_{n \to \infty}\frac{X_n}{n},
\end{equation}

The variables $\inf_{t\in I_n^{(i)}}\ln Z^{(i)}_t$ correspond to the logarithm of the minimum value attained by $Z^{(i)}$ during its $n$th excursion below $\eta/2$, and since $Z^{(i)}$ is not a Markov process, those variables are not independent. In what follows, we will couple them with a sequence of independent and identically distributed random variables $(X_n,n\geq1)$ such that almost surely for all $n\geq 0$, $X_n\leq\inf_{t\in I_n^{(i)}}\ln Z^{(i)}_t$. We will not use the same definition of $X$ in the case \eqref{coex} and in the cases \eqref{excl} and \eqref{excli}, but in both situations, we will use the following characterization for a sequence of independent and identically distributed random variables:
$$\limsup_{n}\frac{|X_n|}{n}=0 ~\text{a.s}\iff \E[|X_1|]<\infty.$$

\begin{enumerate}
    \item[1)] First assume that \eqref{coex} holds.
For $j \in \{1,2\}$, we define the process $(\widehat{Z}^{(j)}_t,t\geq 0)$ as the strong solution to \eqref{upper_bound_no_compet} with $\widehat{Z}^{(j)}_0=a_j/\beta_{jj}$.
It is an upper-bound for the process $Z^{(j)}$. As previously, we consider a non-Lévy random environment $\widecheck L^{(i)}$, defined as in \eqref{eq:focal_lower}. %While $Z^{(i)}$ is lower than $\eta$, $L^{(i)}$ is a lower-bound for the growth rate of $Z^{(i)}$.
For any initial condition $Z_0^{(i)}<\eta$:
$$\ind_{T^{(i,+)}_1>t}Z_0^{(i)}e^{\widecheck L^{(i)}_t}\leq Z_t^{(i)}.$$
Since we have the stochastic domination 
$$Z^{(i)}_{T_n^{(i,-)}}\succeq\eta\Theta^{(i)}_n/2,$$
we can write
$$\inf_{t\in I_n^{(i)}}Z_t^{(i)}\succeq\eta\Theta^{(i)}_n\inf_{t \in I_n^{(i)}}e^{\widecheck L^{(i)}_t - \widecheck L^{(i)}_{T_n^{(i,-)}}}/2,$$
where $\Theta^{(i)}_n$ denotes the survival probability of the catastrophe occurring at time $T^{(i,-)}_n$ (with the convention $\Theta^{(i)}_n=1$ if $Z^{(i)}$ crosses $\eta/2$ continuously), and where, by the strong Markov property at $T^{(i,-)}_n$, the environment driving $Z^{(i)}$ after that time is the shifted environment. We can thus define, on the same probability space, a sequence a processes indexed by $n\geq 1$:
$$L^{(i),n}_t:=\widecheck L^{(i)}_{T^{(i,-)}_n+t}- \widecheck L^{(i)}_{T^{(i,-)}_n}, t\geq 0,$$
with a fixed initial condition $\widehat{Z}^{(j)}_{T^{(i,-)}_n}=a_j/{\beta_{jj}}$ for all $n\geq 1$. Therefore, $(L^{(i),n},n\geq 1)$ are independent and identically distributed.

Only this post-jump environment is involved in the bound. Since $\widehat{Z}^{(j)}_{T^{(i,-)}_n}\leq a_j/\beta_{jj}$ and $\widecheck L^{(i)}$ is non-increasing in $\widehat{Z}^{(j)}$, the shifted environment $ L^{(i),n}$ stochastically dominates a copy of $\widecheck L^{(i)}$ which is independent of $\Theta^{(i)}_n$; this is all we need, since only a lower bound on $\inf_{t\geq0}e^{\widecheck L^{(i)}_t}$ is used.
Recall \eqref{liminflogZ} and let $(\underline{L}^{(i)}_n, n \in \N)$ be a sequence of independent identically distributed random variables with law $\inf_{t\in I_n^{(i)}}L_{t-T^{(i,-)}_n}^{(i),n}$ and $(\Theta_{n}^{(i)}, n \in \N)$ be a sequence of independent identically distributed random variables with law $\Theta^{(i)}$, independent of the first sequence. We thus have to prove that
$$\liminf_{n\to\infty}\frac{\ln ( \eta\Theta_{n}^{(i)}e^{\underline{L}^{(i)}_n})}{n}=0  \Longleftrightarrow \E[|\ln\Theta^{(i)}|]+\E[|\inf_{t \geq 0} \widecheck L^{(i)}_t|]<+\infty. $$
According to \eqref{Z_explicite}, we have
\begin{align*}
    \int_0^{t}\widehat{Z}^{(j)}_sds & = \int_0^{1}\widehat{Z}^{(j)}_sds-\frac{a_j}{\beta_{jj}}\ind_{\{t>1\}}\int_1^t\frac{e^{K^{(j)}_s}}{1+a_j\int_0^se^{K^{(j)}_u}du}ds\\
    & \geq \int_0^{1}\widehat{Z}^{(j)}_sds-\ind_{\{t>1\}}\frac{1}{\beta_{jj}}\int_1^t\frac{e^{K_s^{(j)}}}{\int_0^s e^{K^{(j)}_u}du}ds\\
     &=\int_0^{1}\widehat{Z}^{(j)}_sds-\ind_{\{t>1\}}\frac{1}{\beta_{jj}}\left[\ln\left(\int_0^te^{K^{(j)}_s}ds\right)-\ln\left(\int_0^1e^{K^{(j)}_s}ds\right)\right],
\end{align*}
which gives a lower bound for $\widecheck L^{(i)}_t$ using \eqref{eq:focal_lower}.
Since $(\widehat{Z}^{(j)}_t,t\geq 0)$ is bounded, $\int_0^{1}\widehat{Z}^{(j)}_sds$ has a finite mean, and from Jensen's inequality $\ln\int_0^1e^{K^{(j)}_s}ds\geq\int_0^1K_s^{(j)}ds\geq\min\limits_{s\in [0,1]}K^{(j)}_s$ has also a finite mean according to Lemma~\ref{momentinf}. Moreover, for any $\eps>0$,
\begin{align*}
 \ind_{\{t>1\}}\left[\ln\left(\int_0^te^{K^{(j)}_s}ds\right)\right] 
  % \leq\ind_{\{t>1\}}\left[\ln \left(e^{(1+\varepsilon)\E[K^{(j)}_1]t}\int_0^te^{K^{(j)}_s-(1+\varepsilon)\E[K^{(j)}_1]s}ds\right)\right] \\
  % &\leq \ind_{\{t>1\}}\left[(1+\varepsilon)\E[K^{(j)}_1]t+\ln\left(\int_0^te^{K^{(j)}_s-(1+\varepsilon)\E[K^{(j)}_1]s}ds\right)\right]\\
   &  \leq (1+\varepsilon)\E[K^{(j)}_1]t+\left[\ln\left(\int_0^{\infty}e^{K^{(j)}_s-(1+\varepsilon)\E[K^{(j)}_1]s}ds\right)\right]^+
\end{align*}
where $[\cdot]^+$ stands for the positive part. We now prove that the last term, which does not depend on $t$, has finite mean.
From Jensen's inequality, we get for $p >0$
$$\E\left[\ln\left(\int_0^{\infty}e^{\widetilde{K}^{(\eps)}_s}ds\right)\right]=\frac{1}{p}\E\left[\ln\left(\int_0^{\infty}e^{\widetilde{K}^{(\eps)}_s}ds\right)^p\right]\leq\frac{1}{p}\ln\E\left[\left(\int_0^{\infty}e^{\widetilde{K}^{(\eps)}_s}ds\right)^p\right],$$
where $(\widetilde{K}^{(\eps)}_t,t\geq 0):= (K^{(j)}_t-(1+\varepsilon)\E[K^{(j)}_1]t,t\geq 0)$ is spectrally negative and goes to $-\infty$ when $t\to\infty$. Then, there exists $0<p<1$ such that $\E[e^{p\widetilde{K}^{(\eps)}_1}]<1$.
According to \cite[Lemma 2.1]{maulik2006tail}, this ensures that $\E[(\int_0^{\infty}e^{\widetilde{K}^{(\eps)}_s}ds)^p]<\infty$.
To summarize, we have proven that
$$ \widecheck L^{(i)}_t\geq K_t^{(i)}-\eta\beta_{ii}t-\frac{\beta_{ij}}{\beta_{jj}}(1+\varepsilon)\E[K^{(j)}_1]t+ \mathscr{R} =: K^{(\eps,\eta)}_t + \mathscr{R}, $$
where $\mathscr{R}$ has a finite mean and $\left(K^{(\eps,\eta)}_t,t\geq 0\right)$ is a Lévy process with finite variance which drifts to $+\infty$ if $\eps$ is small enough. According to Lemma \ref{momentinf}, its infimum has a finite mean. It ends the proof of Lemma \ref{lnZt/t} under assumption \eqref{coex}. \\

\item[2)]Assume that \eqref{excl} or \eqref{excli} holds.
In this case, from Proposition \ref{3cas} 
$$
\{\limsup_{t \to \infty} Z^{(i)}_t>0\}=\{\lim_{t \to \infty} Z^{(j)}_t=0\},
$$
almost surely. If $\eps>0$ satisfies $\bar{a}_i-\eta\beta_{ii}-\eps\beta_{ij}>0,$ then the time 
$$
\tau^{(j)}_{\eps}:=\inf\{t\geq 0, \forall s\geq t, Z^{(j)}_s<\eps\},
$$
is almost surely finite on the event $\{\lim_{t \to \infty} Z^{(j)}_t=0\}$. Define now the Lévy process $M^{(i)}$ via
$$M^{(i)}_t=K^{(i)}_t-\eta\beta_{ii}t-\eps\beta_{ij}t, \quad t \geq 0.$$
Using the same arguments as previously we have for any $t\geq0$ and any integer $n$
$$\ind_{t \in [T^{(i,-)}_n,T^{(i,+)}_n],t>\tau^{(j)}_{\eps}}Z_{T^{(i,-)}_n}^{(i)}e^{M^{(i)}_t-{M^{(i)}_{T^{(i,-)}_n}}}\leq Z_t^{(i)},$$
and we can write a similar stochastic domination as before
$$\ind_{t>\tau^{(j)}_{\eps}}\inf_{[T_n^{(i,-)},T_{n}^{( i,+)}]}Z_t^{(i)}\succeq\inf_{t\geq 0}\varepsilon\Theta^{(i)}e^{M^{(i)}_t}/2.$$
As a consequence we just have to check that $|\inf_{t\geq 0}M^{(i)}_t|$ has a finite mean, which stems from a direct application of Lemma \ref{momentinf}.
\end{enumerate}
\end{proof}

\begin{proof}[Proof of Proposition \ref{ergofaible}]
    We will prove that for each $i\in \{1,2\}$, we have:
\begin{equation} \label{to_show}
    \ind_{\{\limsup_tZ^{(i)}_t>0\}}\lim_{t \to \infty}\frac{1}{t}\int_0^t(\beta_{ii}Z^{(i)}_s+\beta_{ij}Z^{(j)}_s)ds=\ind_{\{\limsup_tZ^{(i)}_t>0\}}\bar{a}_i\quad a.s.,
\end{equation}
where $j=3-i$. Let $\widetilde{Z}^{(i)}_t:=  {Z}^{(i)}_t e^{-K^{(i)}_t}$, $t \geq 0$, where $K$ is defined in \eqref{defK}. 
From \eqref{Ito}, we get
\begin{equation}
\label{preuvequant1}
    \frac{1}{t}\int_0^t(\beta_{ii}Z^{(i)}_s+\beta_{ij}Z^{(j)}_s)ds=\frac{\ln Z^{(i)}_0}{t}+\frac{K^{(i)}_t}{t}-\frac{\ln Z^{(i)}_t}{t}.
\end{equation}
The first term of the right hand side goes to $0$ almost surely. Moreover, %Strong law of large numbers for Lévy processes (see for instance Ex. 7.2 in \cite{kyprianou2006introductory}) ensures that
$$ \lim_{t \to \infty}\frac{K^{(i)}_t}{t}=\E[K^{(i)}_1]= \bar{a}_i \quad \text{a.s.}, $$
so that \eqref{to_show} holds if the last term on the right hand side of \eqref{preuvequant1} goes to $0$ a.s., which is the content of Lemma \ref{lnZt/t}. It thus only remains to check that its assumptions are satisfied, that is, to exhibit some $\eta>0$ satisfying both $\overline a_i-\beta_{ii}\eta>0$ and \eqref{above_eps}. We may assume that $\P(\limsup_tZ^{(i)}_t>0)>0$, since otherwise \eqref{to_show} is trivial.
\begin{enumerate}
\item[1.] If \eqref{excl} or \eqref{excli} holds, we choose $\eta>0$ and $\eps>0$ such that $\overline a_i-\beta_{ii}\eta-\beta_{ij}\eps>0$, ensuring that $\overline a_i-\beta_{ii}\eta>0$. By Proposition \ref{3cas}, almost surely $\{\limsup_tZ^{(i)}_t>0\}=\{\lim_tZ^{(j)}_t=0\}$, so that on this event the time $\tau^{(j)}_{\eps}:=\inf\{t\geq0,\ \forall s\geq t,\ Z^{(j)}_s<\eps\}$ is finite. For $t\geq\tau^{(j)}_{\eps}$, and as long as $Z^{(i)}_t\leq\eta$, the comparison of Lemma \ref{lem:comp}(ii) (with $\widehat Z^{(j)}$ replaced by the constant $\eps$ in \eqref{eq:focal_lower}) yields
$$Z^{(i)}_t\geq Z^{(i)}_{\tau^{(j)}_{\eps}}e^{M^{(i)}_t-M^{(i)}_{\tau^{(j)}_{\eps}}},
\qquad M^{(i)}_t:=K^{(i)}_t-(\beta_{ii}\eta+\beta_{ij}\eps)t,$$
and the Lévy process $M^{(i)}$ drifts to $+\infty$ by the choice of $\eta$ and $\eps$. Then, $Z^{(i)}$ reaches $\eta$ at arbitrarily large times on the event $\{\lim_tZ^{(j)}_t=0\}$ and \eqref{above_eps} holds.
\item[2.] If \eqref{coex} holds, we choose $0<\eta<\mathfrak{f}_{ij}/\beta_{ii}$, which is possible since $\mathfrak{f}_{ij}>0$. This choice ensures $\overline a_i-\beta_{ii}\eta>0$, and so that $\ell^-_i(\eta)>0$, with the notation of Lemma \ref{lem:comp}. Hence, by Lemma \ref{lem:comp}(b), the process $Z^{(i)}$ reaches the level $\eta$ in finite time almost surely, from any initial condition; applying this successively at the times $\sigma_{n+1}:=\inf\{t\geq\sigma_n+1,\ Z^{(i)}_t\geq\eta\}$ and using the strong Markov property, we get that $Z^{(i)}$ reaches $\eta$ at arbitrarily large times, so that $\limsup_tZ^{(i)}_t\geq\eta$ almost surely and \eqref{above_eps} holds. This ends the proof.
\end{enumerate}
\end{proof}

\subsection{Proofs of the results of Section \ref{newcomer_section}}

We now consider that the population with type-$1$ individuals is established, and that a mutant with type-$2$ tries to invade.
\subsubsection{Branching Process in Varying Environment (BPVE)}
\label{subsec_introofbpre}
First, we define a Markov process of successive environments $E:=(E_n,n\geq0)= (t_n, \theta_n, z_n)$ on $\mathcal{E}:=\R^+\times[0,1]\times\R^+$, where $t_n$ is the time between the $n$th and $(n+1)$th catastrophe, $\theta_n$ is the probability of survival of the $n$th catastrophe for the invading population, and $z_n$ is the quantity of individuals of the established type just after the $n$th catastrophe. Let $(\phi_t:\R^+\to\R^+,t\geq0)$ be a collection of continuously differentiable functions which are upper bounded, and satisfy $\phi_0=\operatorname{Id}$ and $\phi_s\circ\phi_t=\phi_{s+t}~~\forall s,t\geq0$.
Conditionally on $E_n=(t_n,\theta_n,z_n)$, the law of $E_{n+1}$ is defined as the law of
$$(T,\Theta^{(2)},\Theta^{(1)}\phi_{t_n}(z_n))$$
where $T$ follows an exponential law of parameter $r$, $(\Theta^{(1)},\Theta^{(2)})$ follows a law $\kappa$ on $(0,1)^2$ and $(T,\Theta^{(1)},\Theta^{(2)})$ is independent of $(E_0,E_1,...,E_n)$.
We impose the initial state $E_0=(T_1,1,z_0^{(1)})$ with $T_1$ an exponential variable of parameter $r$ and $z_0^{(1)}>0$.
 This uniquely defines a law for the random catastrophes process, which we will denote $\operatorname{RC}(r,\kappa,z_0^{(1)},\phi)$.

To define the BPVE, we fix a reproduction law as a function of a given environment. Given $E=(t,\theta, z)\in\mathcal{E}$, 
each type-2 individual is killed by the catastrophe with probability $1-\theta$ and, with probability $\theta$, it survives and has a number $n_t^{(z)}$ of descendants, where $(n_s^{(z)},s\leq t)$ is a birth-death process with birth rate $\lambda_2$, death rate (at time $s\geq 0$) given by $\mu_2+ \beta_{21}\phi_s(z)$ and initial condition $n_0 = 1$. 
Thus, if $B$ is a Bernoulli random variable with parameter $\theta$, independent of $n_t^{(z)}$, the reproduction law is the law of $Bn_t^{(z)}$, denoted by $\mathcal{L}(E)$.

We can now define the law of our branching process with random catastrophes:
$$X=(X_n,n\geq0)\sim\operatorname{BPRC}(r,\kappa,z_0^{(1)},\phi,\lambda_2,\mu_2,\beta_{21})$$
if there exist a process $(E_n,n\geq 0)\sim\operatorname{RC}(r,\kappa,z_0^{(1)},\phi)$ and random variables $(Y_{n,i,e},n\geq1,i\geq 1,e\in \mathcal{E})$ such that $Y_{n,i,e}$ has law $\mathcal{L}(e)$, and for any $n\neq n',i\neq i',e,e'$, $Y_{n,i,e}$ and $Y_{n',i',e'}$ are independent and satisfy
$$X_{n+1}=\sum_{i=1}^{X_n}Y_{n,i,E_n}.$$

\subsubsection{Results on Branching processes with random catastrophes}
In a recent paper, Kersting gave general properties of the long-time behaviour of discrete branching processes in varying environment under very mild conditions \cite{Kersting19}. In order to apply those results, we first prove that a process $X$ with law $\operatorname{BPRC}(r,\kappa,z_0^{(1)},\phi,\lambda_2,\mu_2,\beta_{21})$ satisfies condition $(A)$ in \cite{Kersting19}. Let $E:=(E_n,n\in\N)$. In our setting, this condition $(A)$ reads as follows.
\begin{equation} \label{ass_A_Kersting} \P\left(\sup_{n}\frac{(\E[Y_{n,1,E_n}^2|E]-\P(Y_{n,1,E_n}=1|E))\P(Y_{n,1,E_n}>0|E)}{(\E[Y_{n,1,E_n}|E]-\P(Y_{n,1,E_n}=1|E))\E[Y_{n,1,E_n}|E]} < \infty \right)=1.\end{equation}
In the next lemma, we prove that this condition is satisfied.

\begin{lemma} \label{lemma_condA}
    Let $X$ be a $\operatorname{BPRC}(r,\kappa,z_0^{(1)},\phi,\lambda_2,\mu_2,\beta_{21})$. 
Then $X$ satisfies \eqref{ass_A_Kersting}.
\end{lemma}

\begin{proof}[Proof of Lemma \ref{lemma_condA}]
A probability law on $\N$ is called linear fractional with parameters $(m,b)$, $m>0$, $b\ge0$, $m\le1+b$, if its probability generating function writes
\begin{equation} \label{shape_linfrac}
f(s)=1-\frac{m(1-s)}{1+b(1-s)}, \quad 0\leq s <1 .
\end{equation}
Let $n\in \N$, and let $E_n=(t_n,\theta_n,z_n)$ be the environment. We introduce the notations
$$\psi_{E_n}(t):=\int_0^{t}\big(\lambda_2-\mu_2-\beta_{21}\phi_u(z_n)\big)du,\quad
V_{E_n}(t):=\int_0^{t}\lambda_2e^{-\psi_{E_n}(u)}du,$$
$$
W_{E_n}(t):=\int_0^t(\mu_2+\beta_{21}\phi_u(z_n))e^{-\psi_{E_n}(u)}du .
$$
According to \cite{kendall1948generalized}, for $s\in[0,1)$,
\begin{equation}\label{eq:kendall}
F^{(n)}_{t_n}(s):=\E\big[s^{n^{(z_n)}_{t_n}}|E_n\big]
=1-\frac{1-s}{e^{-\psi_{E_n}(t_n)}+(1-s)V_{E_n}(t_n)} .
\end{equation}
Besides, by independence of $B$ and $n^{(z)}$, 
$$\E[s^{Y_{n,1,E_n}}|E_n]= (1-\theta_n)s^0+\theta_n F^{(n)}_{t_n}(s).$$
Substituting \eqref{eq:kendall} yields
$$ \E[s^{Y_{n,1,E_n}}|E_n]=1-\frac{\theta_n(1-s)}{e^{-\psi_{E_n}(t_n)}+(1-s)V_{E_n}(t_n)}
=1-\frac{\theta_n e^{\psi_{E_n}(t_n)}(1-s)}{1+(1-s)e^{\psi_{E_n}(t_n)}V_{E_n}(t_n)}.
$$
This entails that  the law of $Y_{n,1,E_n}$ conditionally on $E_n$ is linear fractional with parameters
$$
m_{E_n}=\theta_ne^{\psi_{E_n}(t_n)},\qquad b_{E_n}=e^{\psi_{E_n}(t_n)}V_{E_n}(t_n).
$$
In Example 8 in \cite{Kersting19}, Kersting has proven that linear fractional laws uniformly satisfy condition \eqref{ass_A_Kersting} whatever their parameters. This ends the proof of the Lemma.
\end{proof}

 We will apply Lemma \ref{lemma_condA} with $\phi$ being the flow of the solution to the logistic equation.
More precisely, for $a \in \R$, $\beta > 0$, we denote by $ \phi^{a,\beta}$ the solution to 
$$ \dot{\phi}^{a,\beta}= \phi^{a,\beta}(a-\beta \phi^{a,\beta}), $$
and for all $t\geq 0$,
\begin{equation}
\label{eqdep-abeta}
Z^{a,\beta}_t = z_0^{(1)}+ \int_0^tZ^{a,\beta}_s\left(a - \beta Z^{a,\beta}_s \right)ds + \int_0^t \int_0^1 (\theta^{(1)}-1) Z^{a,\beta}_{s-}S_1(ds,d\theta^{(1)}).
\end{equation}
We may now apply \cite[Theorem 1]{Kersting19}.
\begin{prop}
\label{prop_Kersting_aleatoire}
    Let $(a,\beta)\in \R \times \R_+^*$, $X\sim\operatorname{BPRC}(r,\kappa,z_0^{(1)},\phi^{a,\beta},\lambda_2,\mu_2,\beta_{21})$, and
    $$ \mathfrak{f}:=\lambda_2-\mu_2+r\E[\ln \Theta^{(2)}]-\beta_{21}\frac{a+r\E[\ln \Theta^{(1)}]}{\beta}. $$
Let $K \in \N$, $\eps_K>0$ be such that $\eps_K K \to \infty$ and $\eps_K=o(1)$ as $K$ goes to $\infty$, and if $\mathfrak{f}>0$, assume also that $\ln \eps_K =o(\ln K)$. Let $T_{\eps_K}^{X,K}:=\inf\{T_n\geq 0, X_n\geq K \eps_K\}$ and $T_{0}^{X}:=\inf\{T_n \geq 0, X_n=0\}$.
    Then we have the convergence in law
    $$\left(\frac{T_0^{X}\wedge {T}^{X,K}_{\eps_K}}{\ln K},\ind_{\{{T}_0^{X}>{T}^{X,K}_{\eps_K}\}}\right)\overset{\mathcal{L}}{\longrightarrow}\left\{\begin{array}{ll}
    B \cdot\left(\mathfrak{f}^{-1},1\right) &\text{if } \mathfrak{f}>0,\\
    (0,0)&\text{if }\mathfrak{f}<0.
    \end{array}\right.$$
    where $B$ is a Bernoulli variable with parameter $p=\P(\lim X_n=\infty)>0$ when $\mathfrak{f}>0$.
    \end{prop}

\begin{proof}
Let us first introduce some notation for the sake of readability. For $1 \leq k \leq n$,
\begin{equation} \label{notations_mMR}
  M_n := \prod_{k=1}^n m_{E_k} \quad \text{and} \quad R_n:=\frac{\var[Y_{n,1,E_n}|E_n]}{\E[Y_{n,1,E_n}|E_n]^2}.
\end{equation}
\noindent $\bullet$ Case  $\mathfrak{f}<0$. According to Lemma \ref{lemma_condA}, we can apply \cite[Theorem 1]{Kersting19} and we obtain from condition (iv) the following inclusion of events of the $\sigma$-algebra $\sigma(E)$:
$$\left\{\lim_{n\to \infty}M_n =0\right\} = \left\{\lim_{n\to \infty}\E[X_n|E]=0\right\} \subset \big\{\E[\ind\{\lim X_n=0\}|E]=1\big\}.$$
Let us show that the event on the left-hand side has probability $1$. For $n \in \N$ we have
\begin{align} \label{as_equiv}
    \ln M_n&=\sum_{k=1}^n \ln \Theta_k^{(2)} +\int_{0}^{T_n}(a_2-\beta_{21}Z^{a,\beta}_s)ds\nonumber 
    \\&\sim n\left(\E[\ln \Theta^{(2)}]+\frac{1}{r}\left(a_2-\beta_{21}\frac{a + r\E[\ln\Theta^{(1)}]}{\beta}\right)\right)=n\frac{\mathfrak{f}}{r}\to -\infty,
\end{align}
where we used the law of large numbers, and a combination of \eqref{Z_explicite} and \eqref{eq:cesaro} on $Z^{a,\beta}$. This shows that $\P(X_n\to0)=1$ and implies that almost surely $\sup X_n<\infty$. Moreover $\{T_{\eps_K}^{X,K}<\infty\}=\{\sup X_n\geq K\eps_K\}$, and $T_0^X$ does not depend on $K$ so that, in probability, $\lim T_0^X/\ln K=0$. We deduce that:
$$\lim_{K\to\infty}\left(\frac{{T}^X_0\wedge {T}^{X,K}_{\eps_K}}{\ln K},\ind_{\{{T}^X_0>{T}^{X,K}_{\eps_K}\}}\right)=(0,0), \quad \text{in probability}.$$

\noindent $\bullet$ Case  $\mathfrak{f}>0$. Recall notations given in \eqref{notations_mMR}.
Using condition (vii) in  \cite[Theorem 1]{Kersting19} yields:
$$\P(\lim X_n=\infty)>0 \iff \P\left(\sum_{k=1}^{\infty}\frac{R_k}{M_{k-1}}<\infty\right)>0.$$
From \eqref{eq:kendall} we may compute 
$$ R_n = \frac{m_{E_n}(2b_{E_n}+1)-m^2_{E_n}}{m^2_{E_n}} = \frac{1}{\Theta^{(2)}_n} \left( 2\lambda_2 \int_0^{T_n-T_{n-1}}e^{-\psi_{E_n}(s)}ds + e^{-\psi_{E_n}(T_n-T_{n-1})}\right)-1. $$

Using that $\psi_{E_n}(t)\geq \delta t$ for every $t\geq 0$, where
\begin{equation*} \label{def_delta}
  \delta:=\Big(a_2-\beta_{21}\big((a/\beta)\vee z_0^{(1)}\big)\Big)\wedge(-1)\;<\;0
\end{equation*}
(any strictly negative upper bound would do), we get
$$\int_0^{t}e^{-\psi_{E_n}(s)}ds\leq\int_0^{t}e^{|\delta| s}ds\leq\frac{e^{|\delta| t}}{|\delta|}
\qquad\text{and}\qquad e^{-\psi_{E_n}(t)}\leq e^{|\delta| t},$$
so that, with $C:=2\lambda_2/|\delta|+1$,
$$R_n\leq \frac{C}{\Theta^{(2)}_n}e^{|\delta|(T_n-T_{n-1})}.$$
Combining this last inequality with \eqref{as_equiv}, we get that almost surely for $n$ large enough, 
$$ \ln \Big( \frac{R_n}{M_{n-1}} \Big) \leq - n \Big[ \frac{\mathfrak{f}}{2r}- \frac{1}{n}\big(|\delta|(T_n-T_{n-1})+ \ln C+ |\ln \Theta^{(2)}_n|\big) \Big].  $$
Since $(T_n-T_{n-1},\ln{\Theta^{(2)}_n},n \in \N)$ is a sequence of independent identically distributed variables with finite means, $\limsup n^{-1}[|\delta|(T_n-T_{n-1})+|\ln \Theta^{(2)}_n|]=0$, and almost surely, $R_n/M_{n-1}$ has exponential decay: condition (vii) is satisfied.

The process $(X_{n}/M_n, n \in \N)$ is a non-negative martingale, thus converges almost surely to some random variable $W\geq 0$, which satisfies $\P(W=0)=\P(\lim X_n=0)<1$ according to \cite[Theorem 2]{Kersting19}. We deduce that
    $$\lim_{K\to \infty}\P(T_{\eps_K}^{X,K}<T_0^X)=\P(W>0).$$
Indeed, $\lim_{K\to \infty}\P(T_{\eps_K}^{X,K}<T_0^X)=\P(\limsup X_n=\infty),$
and by the previous dichotomy, we get that $\P(\limsup X_n=\infty)=\P(W>0)>0$. This gives the convergence in law of the second coordinate to a Bernoulli random variable of parameter $p:=\P(\lim X_n= \infty)>0$. To show the convergence of the pair given in the proposition, similarly to the case $\mathfrak{f}<0$, we first remark that $T_0^X$ does not depend on $K$, so that
$$\ind_{\{W=0\}}\frac{T_0^X\wedge T_{\eps_K}^{X,K}}{\ln K}=\ind_{\{T_0^X<\infty\}}\frac{T_0^X \wedge T^{X,K}_{\eps_K}}{\ln K}\to 0 \quad \text{in probability}.$$
Now under the event $W>0$, we know that $T_0^X=\infty$, and we just have to estimate $T_{\eps_K}^{X,K}$. Almost sure convergence to $W$ gives
$$ \ind_{\{W>0\}}\big( \ln X_n - \ln M_n \big) \xrightarrow[n\to\infty]{}\ind_{\{W>0\}} \ln W \quad \text{almost surely.}$$
Adding \eqref{as_equiv} allows to conclude that
$$ \lim_{K \to \infty }\ind_{\{W>0\}}\frac{\ln (\eps_K K)}{T_{\eps_K}^{X,K}} = \lim_{n \to \infty }\ind_{\{W>0\}}\frac{\ln X_n}{T_n} = \ind_{\{W>0\}} \mathfrak{f} \quad \text{a.s.}$$
This ends the proof.
\end{proof}

\subsubsection{Proof of Proposition \ref{possible_invasion}} \label{subsec_proofinvasion}

The proof consists in sandwiching $\Gamma^{(2,K)}$ between two branching processes in random environments belonging to the class introduced above, and which are sufficiently close to each other to exhibit the same asymptotic behavior. To this aim, we  define $(X^v)_{v\in (-\eps_0,\eps_0)}$ as a family of BPVE , living on the same probability space as $\Gamma^{(K)}$, such that for any $\eps\in(0,\eps_0)$ and for a sufficiently long-time $X^{-\eps}_n\leq\Gamma^{(2,K)}_{T_n}\leq X^{\eps}_n$, where $T_n$ (defined page \pageref{def_triplet}) is the time of the $n$th catastrophe. 

Recall that $\Gamma^{(1,K)}$ and $\Gamma^{(2,K)}$ are defined in \eqref{strongSDE} using five independent Poisson point processes.
%and the notation $S_i(ds,d\theta^{(i)})=\int_0^1\bar N_i(ds,d\theta^{(i)},dp)$ which only concerns times and intensities of the jumps affecting $\Gamma^{(i,K)}$ for $i\in\{1,2\}$.
We first introduce bounds for the process $\Gamma^{(1,K)}$. For $a>0$, we define the process $Z^{(1,a)}$ to be the strong solution to
$$Z^{(1,a)}_t = Z_0^{(1,a)}+ \int_0^tZ^{(1,a)}_s\left(a - \beta_{11}Z^{(1,a)}_s\right)ds + \int_0^t \int_0^1 (\theta^{(1)}-1) Z^{(1,a)}_{s-}S_1(ds,d\theta^{(1)}).$$

 Recall the notation $\widetilde{T}^K_{\eps}$ in \eqref{TKtilde}. The following result is a consequence of Proposition \ref{prop_cvglongue}.

\begin{lemma} \label{lem_cvglongue}
   Assume that \eqref{H3} holds. For any $\tilde{\alpha}>1$, $\eps,\eps'>0$ such that $\beta_{12}^{-1}\overline{a}_1-\eps>0$, we have
    $$ \lim_{K \to \infty}\P\left( \forall t<\widetilde{T}_{\eps}^K\wedge\ln^{\tilde{\alpha}}K,~Z^{(1,a_1-\beta_{12}\eps)}_t-\eps'\leq\frac{\Gamma^{(1,K)}_t}{K}\leq Z^{(1,a_1)}_t+\eps'\right)=1.$$
\end{lemma}
\begin{proof}
We can sandwich the process $\Gamma^{(1,K)}$ (which is not markovian) between two Markov processes $\Gamma^{u}$ and $\Gamma^{l}$ solution to
$$\begin{aligned}
\Gamma_t^{k}
=&\Gamma^{(1,K)}_0+\int_0^t
\left(\int_0^{\lambda_1\Gamma^{k}_{s-}}N_1^+(ds,dx)-\int_0^{(\mu_1 +\beta_{11}\Gamma_{s-}^{k}/K+\delta_{k=l}\beta_{12}\eps)\Gamma^{k}_{s-}}N_1^-(ds,dx)\right)\\
&+\int_0^t\int_0^1\int_0^1\Big(F^{-1}_{\text{Bin}(\Gamma^{k}_{s-},\theta^{(1)})}(p) - \Gamma^{k}_{s-}\Big)\bar{N}_1(ds,d\theta^{(1)},dp),
\end{aligned}$$
where $\delta_{k=l}$ is $1$ if $k=l$ and $0$ if $k=u$. By definition, $\forall t< \widetilde{T}_{\eps}^K,~\Gamma_t^l\leq \Gamma^{(1,K)}_t\leq\Gamma^u_t$ almost surely. Applying Proposition \ref{prop_cvglongue} to the two processes $\Gamma^u$ and $\Gamma^l$ with $\alpha>1$, we get that for any $\eps''>0$, there exists $K$ large enough for which with probability at least $1-\eps''$ we have:
$$\begin{aligned}
    \forall t\leq (\ln K)^{\alpha-\eps''},\quad & Z^{(1,a_1-\beta_{12}\eps)}_t-K^{2\eps''-1/2}\leq\Gamma^l_t/K \leq \Gamma^u_t/K\leq Z^{(1,a_1)}_t+K^{2\eps''-1/2}.
    \end{aligned}$$
%where $Z^{(1,a_1)}$ (respectively $Z^{(1,a_1-\beta_{12}\eps)}$) is the PDMP associated to $\Gamma^u$ (respectively $\Gamma^l$) in the sense of Equations \eqref{strongSDE} and \eqref{eqdep} when $\Gamma^{(2,K)}\equiv0$.
This concludes the proof.
\end{proof}

We now suppose $\mu_2>0$ and fix
\begin{equation}\label{def_eps0}
0<\eps_0<\frac{\overline{a}_1}{\beta_{12}}\wedge\frac{\mu_2}{\beta_{21}}\wedge \frac{|\mathfrak{f}_{21}|}{\beta_{21} + \beta_{22} + \beta_{21}\beta_{12}\beta_{11}^{-1}}.
\end{equation}
The first two constraints ensure that $\mu^v$ defined in \eqref{def_mu} below is well-defined. The last constraint ensures that $\mathfrak{f}_{v}$, defined in \eqref{fitness_BPRE} below, have the same sign as $\mathfrak{f}_{21}$ for every $v\in [-\eps_0,\eps_0]$.
%this is possible because $v\mapsto\mathfrak{f}_{v}$ is continuous, increasing, and equal to $\mathfrak{f}_{21}\neq0$ at $v=0$. The second constraint in \eqref{def_eps0} ensures that the rates $\mu^{v}$ defined in \eqref{def_mu} below are nonnegative (when $\mu_2=0$, one replaces $\mu^{v}$ by $\mu^{v}\vee0$ in \eqref{def_mu}, and none of the arguments below is modified).
For $v\in (-\eps_0,\eps_0)$ and $x_0 \in \N$, let $(X^{v}_t, t\geq0)$ be the strong solution to
\begin{align}\nonumber\label{def_Xeps}
    X_t^{v} &= x_0+\int_0^t
\left(\int_0^{\lambda_2X^{v}_{s-}}N_2^+(ds,dx)-\int_0^{\mu^{v}_sX^{v}_{s-}}N_2^-(ds,dx)\right)\\
&+\int_0^t\int_0^1\int_0^1\Big(F^{-1}_{\text{Bin}(X^{v}_{s-},\theta^{(2)})}(p) - X^{v}_{s-}\Big)\bar{N}_2(ds,d\theta^{(2)},dp),
\end{align}
with $\bar{N}_2$ and $N_2^\pm$ introduced in \eqref{strongSDE} and \begin{equation} \label{def_mu}
\mu^{v}_s=\left(\mu_2+ \ind_{\{v \geq 0\}} \Big[\beta_{21}\big(Z^{(1,a_1-\beta_{12}v)}_s-v\big)\Big]+  \ind_{\{v < 0\}} \Big[ |v|\beta_{22}+\beta_{21}\big(Z^{(1,a_1)}_s+|v|\big)\Big]\right).
\end{equation}
Eq. \eqref{def_eps0} ensures that $\mu_s^v\geq 0$ for all $v\in (-\eps_0, \eps_0)$ if $\mu_2>0$, but the particular case $\mu_2=0$ requires to be considered separately since we cannot define $\eps_0>0$ such that $\mu^{v}_s\geq0$ for any $0<v<\eps_0$. Consequently, we will first show Proposition \ref{possible_invasion} for $\mu_2>0$, and extend the result to the case $\mu_2=0$ noticing that a process $\Gamma^{(K)}$ with $\mu_2=0$ is the almost sure uniform limit of a sequence of processes satisfying $\mu_2\to0$.\\

All these processes are constructed with the same Poisson point measures, which entails the following properties.

\begin{lemma}\label{lem_monotone_X}
Suppose $\mu_2>0$. Almost surely, $(i)$ for every $t\geq0$ the map $v\mapsto X^{v}_t$ is non-decreasing on $(-\eps_0,\eps_0)$; $(ii)$ $\sup_{t\leq T}X^{v}_t<\infty$ for every $T<\infty$ and every $v\in(-\eps_0,\eps_0)$.
\end{lemma}

\begin{proof}
$(i)$ Let $-\eps_0<v\leq v'<\eps_0$. If $0\leq v\leq v'$, the drift of \eqref{eqdep-abeta} is non-decreasing in $a$, so that a comparison argument gives $Z^{(1,a_1-\beta_{12}v')}\leq Z^{(1,a_1-\beta_{12}v)}$ almost surely and we get $\mu^{v'}_s\leq\mu^{v}_s$ from \eqref{def_mu}. From \eqref{def_mu}, the same inequality holds if $v\leq v'\leq0$, and, combining both cases, if $v\leq 0\leq v'$. The processes $X^{v}$ and $X^{v'}$ have the same birth rate, the same effect of catastrophes, and are driven by the same Poisson point measures, so that a pathwise comparison argument entails $X^{v}_t\leq X^{v'}_t$ for every $t\geq0$ if $v\leq v'$. 

$(ii)$ Neglecting deaths and catastrophes, $X^{v}$ is dominated by a Yule process of rate $\lambda_2$ started at $x_0$, which is almost surely finite at all times. As $X^{v}$ is càdlàg with values in $\N$, its supremum on any compact time interval is finite almost surely.
\end{proof}

An application of Lemma \ref{lem_cvglongue} with $\eps'=\eps$ ensures that the death rates of $X^{-\eps},\ \Gamma^{(2,K)},\ X^{\eps}$ stay ordered. The above coupling thus entails that for any $\eps\in(0,\eps_0)$, $\tilde\alpha>1$, and $x_0\in\N$,
\begin{equation} \label{prop_Xulbound} \lim_{K \to \infty}\P\big(\forall t<\widetilde{T}_{\eps}^K\wedge\ln^{\tilde\alpha}K,\ X^{-\eps}_t\leq \Gamma^{(2,K)}_t\leq X^{\eps}_t\ \big|\ X^{-\eps}_0=\Gamma^{(2,K)}_0=X^{\eps}_0=x_0\big)=1.
\end{equation}

Finally, a direct consequence of the construction of the sandwiching processes is that their discretizations belong to the class of BPVE introduced in Section \ref{subsec_introofbpre}. More precisely,
    \begin{itemize}
        \item if $v\leq0$:
    $$(X^{v}_{T_n-},n\in\N)\sim\operatorname{BPRC}(r,\kappa,z^{(1)}_0,\phi^{a_{1},\beta_{11}},\lambda_2,\mu_2+(\beta_{22}+\beta_{21})|v|,\beta_{21}),$$
    \item if $v\geq0$:
    $$(X^{v}_{T_n-},n\in\N)\sim\operatorname{BPRC}(r,\kappa,z^{(1)}_0,\phi^{a_{1}-\beta_{12}v,\beta_{11}},\lambda_2,\mu_2-\beta_{21}v,\beta_{21}).$$
    \end{itemize}
We denote $E^v$ the random catastrophes process associated to $X^{v}$ and $\mathfrak{f}_{v}$ its fitness, given by:
\begin{equation}
\label{fitness_BPRE}
\mathfrak{f}_{v}=\left\{\begin{array}{cc}
    \mathfrak{f}_{21}-(\beta_{22}+\beta_{21})|v| & \text{if } v \leq 0,\\
    \mathfrak{f}_{21}+\big(\beta_{21}+\frac{\beta_{21}\beta_{12}}{\beta_{11}}\big)v & \text{if } v\geq 0.
\end{array}\right.
\end{equation}
In particular $\mathfrak{f}_0=\mathfrak{f}_{21}$, and $v\mapsto\mathfrak{f}_{v}$ is continuous and increasing; notice from \eqref{def_eps0} that $\eps_0$ has been chosen so that $\mathfrak{f}_{v}$ has the sign of $\mathfrak{f}_{21}$ for every $v\in(-\eps_0,\eps_0)$.

The linear fractional structure exhibited in the proof of Lemma \ref{lemma_condA} allows us to compute explicitly the survival probability of $X^{v}$ conditionally on the catastrophes, and thus to compare the invasion probabilities associated with different values of $v$. Recall the $\sigma$-field $\mathcal{G}$ generated by the catastrophes, introduced in the proof of Lemma \ref{lem_Kurtz}, and recall the notation $E_k=(t_k,\theta_k,z_k)$ of Section \ref{subsec_introofbpre}, where $t_k=T_{k+1}-T_k$ and $\theta_k=\Theta^{(2)}_k$ (with $T_0=0$ and $\theta_0=1$). For $v\in(-\eps_0,\eps_0)$ and $k\geq0$, we denote by $(m^{v}_k,b^{v}_k)$ the parameters of the linear fractional law $\mathcal{L}(E^{v}_k)$ obtained in the proof of Lemma \ref{lemma_condA}, that is
$$\begin{aligned}&m^{v}_k=\theta_ke^{\psi^{v}_k(t_k)},&b^{v}_k=e^{\psi^{v}_k(t_k)}V^{v}_k(t_k),\\
&\psi^{v}_k(t)=\int_0^t\big(\lambda_2-\mu^{v}_{T_k+u}\big)du,
&V^{v}_k(t)=\lambda_2\int_0^te^{-\psi^{v}_k(u)}du,\end{aligned}$$
and we denote $\mathcal{M}^{v}_j:=\prod_{k=0}^{j}m^{v}_k$, which coincides with the product $M_j$ of \eqref{notations_mMR} up to the factor $m^{v}_0$. Since the $\psi^{v}_k$ concatenate, $\mathcal{M}^{v}_.$ is defined as in \eqref{formule_p_bis}.

\begin{lemma}\label{lem_continuity_p}
Suppose $\mu_2>0$. Assume $\mathfrak{f}_{-\eps_0}>0$. Then, $v\mapsto p_{v}=\P(\lim X^v=\infty)$ is continuous at $0$ on $(-\eps_0,\eps_0)$.
\end{lemma}

\begin{proof}
Conditionally on $\mathcal{G}$, $(X^{v}_{T_n-},n\in\N)$ is a branching process in the varying environment ${E}^{v}$, which is $\mathcal{G}$-measurable, and whose offspring laws are linear fractional with parameters $(m^{v}_k,b^{v}_k)$ (recall \eqref{shape_linfrac}).
Equation (16) in \cite{Kersting19} therefore reads, in our setting, $$\P\big(X^{v}_{T_n-}>0\,\big|\,\mathcal{G}\big)^{-1}
=\frac{1}{\mathcal{M}^{v}_{n-1}}+\sum_{j=0}^{n-1}\frac{b^{v}_j}{\mathcal{M}^{v}_j}.$$
Letting $n$ go to infinity and using that $\{X^{v}_{T_n-}>0\ \forall n\}=\{\lim X^{v}=\infty\}$ a.s. (Theorem 1 in \cite{Kersting19}), we get 
\begin{equation}\label{agresti}
p_{v}=\E\big[\pi_{v}\big],\qquad
\pi_{v}:=\frac{1}{S_{v}},\qquad S_{v}:=\sum_{j\geq0}\frac{b^{v}_j}{\mathcal{M}^{v}_j}\in[1,\infty],
\end{equation}
with the convention $1/\infty=0$. Since $0\leq\pi_{v}\leq1$, it is enough to prove that $S_{v}\to S_0$ almost surely as $v$ goes to $0$.
First, as for all $s\geq 0$, $v\mapsto \mu_s^v$ is non-increasing (see the proof of Lemma \ref{lem_monotone_X}$(i)$), we obtain that $v\mapsto\pi_{v}$ is non-decreasing and $v\mapsto S_{v}$ is non-increasing, almost surely. Second, for every fixed $j$ the map $v\mapsto (m^{v}_j,b^{v}_j)$ is almost surely continuous: from \eqref{def_mu}, this is clear for $v\leq0$, and for $v\geq0$, combining
$$d\big(Z^{(1,a_1)}_t-Z^{(1,a_1-\beta_{12}v)}_t\big)\leq a_1\big(Z^{(1,a_1)}_t-Z^{(1,a_1-\beta_{12}v)}_t\big)dt+\beta_{12}v Z^{(1,a_1-\beta_{12}v)}_tdt$$
and $Z^{(1,a_1-\beta_{12}v)}_t\leq(a_1-\beta_{12}v)/\beta_{11}$, yields the almost sure bound
\begin{equation*}\label{diff_mu}
0\leq Z^{(1,a_1)}_t-Z^{(1,a_1-\beta_{12}v)}_t\leq\frac{\beta_{12}v}{\beta_{11}}\big(e^{a_1t}-1\big),
\qquad t\geq0,
\end{equation*}
hence the convergence of $\mu^{v}_s$ to $\mu^0_s$ for every $s\geq0$, and, by dominated convergence on the compact interval $[0,t_j]$, that of $\psi^{v}_j(t_j)$ and $V^{v}_j(t_j)$.\\

$\bullet$ {Right-continuity ($v\geq 0$)}. By monotonicity $S_{v}\leq S_0$, while for every $N$,
$$\liminf_{v\downarrow0}S_{v}\geq\lim_{v\downarrow0}\sum_{j\leq N}\frac{b^{v}_j}{\mathcal{M}^{v}_j}
=\sum_{j\leq N}\frac{b^{0}_j}{\mathcal{M}^{0}_j},$$
so that $\liminf_{v\downarrow0}S_{v}\geq S_0$ by letting $N$ go to infinity.\\

$\bullet$ {Left-continuity ($v\leq 0$).} Here the dependence in $v$ is explicit: from \eqref{def_mu}, $\mu^{v}_s=\mu^0_s+c|v|$ with $c:=\beta_{21}+\beta_{22}$, so that $\psi^{v}_j(t)=\psi^0_j(t)-c|v|t$ and,
$$\mathcal{M}^{v}_j=\mathcal{M}^{0}_je^{-c|v|T_{j+1}},\qquad e^{-c|v|t_j}b^{0}_j\leq b^{v}_j\leq b^{0}_j,$$
the last two inequalities coming from $V^{0}_j(t_j)\leq V^{v}_j(t_j)\leq e^{c|v|t_j}V^{0}_j(t_j)$. Hence, for $v\leq0$,
\begin{equation}\label{encadrement_Sv}
\sum_{j\geq0}e^{c|v|T_j}\frac{b^{0}_j}{\mathcal{M}^{0}_j}\ \leq\ S_{v}\ \leq\ \Sigma(c|v|),
\qquad\text{where }\ \Sigma(\beta):=\sum_{j\geq0}e^{\beta T_{j+1}}\frac{b^{0}_j}{\mathcal{M}^{0}_j}.
\end{equation}
Fix $\eps_1\in(0,\eps_0)$. As $\mathfrak{f}_{-\eps_1}>0$, following the proof of Proposition \ref{prop_Kersting_aleatoire}, we get that the environment ${E}^{-\eps_1}$ almost surely satisfies condition (vii) of Theorem 1 in \cite{Kersting19}, so that $\pi_{-\eps_1}>0$, that is $S_{-\eps_1}<\infty$, almost surely. By \eqref{encadrement_Sv}, $\sum_{j}e^{c\eps_1T_j}b^{0}_j/\mathcal{M}^{0}_j<\infty$ almost surely. Combining that $T_j/j\xrightarrow[]{\text{a.s.}}1/r$ by the law of large numbers and that $t_j\overset{\text{a.s}}{=}O(\ln j)$ (by the Borel-Cantelli lemma as $\sum_j\P(t_j>2r^{-1}\ln j)=\sum_jj^{-2}<\infty$), we obtain
\begin{equation}\label{tj_Tj}
\lim_{j\to\infty}\frac{t_j}{T_j}=0\qquad\text{almost surely.}
\end{equation}
Hence, for $\beta_1:=c\eps_1/2$, we have $\beta_1T_{j+1}=\beta_1T_j+\beta_1t_j\leq c\eps_1T_j$ for $j$ large enough by \eqref{tj_Tj}, and $\Sigma(\beta_1)<\infty$ almost surely. Dominated convergence in $\Sigma(c|v|)$ then gives $\Sigma(c|v|)\to\Sigma(0)=S_0$ as $v$ increases to $0$, and \eqref{encadrement_Sv} yields $S_{v}\to S_0$, which ends the proof.
\end{proof}

We are now ready to prove the main result of this section.

\begin{proof}[Proof of Proposition \ref{possible_invasion}]
We first assume that $\mathfrak{f}_{21}<0$. Fix $\tilde\alpha>1$, $\mathfrak{a}>0$, and let $\eps\in(0,\eps_0\wedge\mathfrak{a})$, so that $\mathfrak{f}_{\eps}<0$. Applying Proposition \ref{prop_Kersting_aleatoire} to $X^{\eps}$ (see \eqref{as_equiv}) ensures that $(X^{\eps}_{T_n-},n\in\N)$ becomes extinct almost surely. Combining this with Lemma \ref{lem_monotone_X}$(ii)$, we get
\begin{equation*}\label{order_at_cata} T_0^{X^\eps}<\infty\quad\text{and}\quad \mathscr{M}^{\eps}:=\sup_{t\leq T_0^{X^\eps}}X^{\eps}_t<\infty\qquad\text{almost surely,}
\end{equation*}
As neither $T_0^{X^\eps}$ nor $\mathscr{M}^{\eps}$ depend on $K$, we deduce
\begin{equation}\label{event_subcritical}
\lim_{K\to\infty}\P\big(T_0^{X^\eps}<\ln^{\tilde\alpha}K \ \text{ and }\
\mathscr{M}^{\eps}<\lceil\eps K\rceil-1\big)=1 .
\end{equation}
Let us work on the intersection of the events in \eqref{prop_Xulbound} (with $x_0=1$) and in \eqref{event_subcritical}, denoted by $\mathfrak{D}_\eps^K$, whose probability goes to one when $K$ goes to infinity. Assume first that $\widetilde{T}^K_{\eps}\leq T_0^{X^\eps}$. Then $\widetilde{T}^K_{\eps}<\ln^{\tilde\alpha}K$, and as $\Gamma^{(2,K)}$ increases by jumps of size one,
$$\lceil\eps K\rceil-1=\Gamma^{(2,K)}_{\widetilde{T}^K_{\eps}-}\leq X^{\eps}_{\widetilde{T}^K_{\eps}-}\leq\mathscr{M}^{\eps},$$ which contradicts \eqref{event_subcritical}. Hence $\widetilde{T}^K_{\eps}>T_0^{X^\eps}$, and \eqref{prop_Xulbound} taken at time $t=T_0^{X^\eps}$ gives $\Gamma^{(2,K)}_{T_0^{X^\eps}}\leq X^{\eps}_{T_0^{X^\eps}}=0$. As $0$ is an absorbing state for $\Gamma^{(2,K)}$, we conclude that 
$$\widetilde{T}^K_0\leq T_0^{X^\eps}\quad\text{and}\quad\widetilde{T}^K_{\eps}=\infty .$$
Finally, $\eps\leq\mathfrak{a}$ implies $\widetilde{T}^K_{\mathfrak{a}}\geq \widetilde{T}^K_{\eps}$ for $K$ large enough, so that on $\mathfrak{D}_\eps^K$, we have $\ind_{\{\widetilde{T}^K_0>\widetilde{T}^K_{\mathfrak{a}}\}}=0$ and $\widetilde{T}^K_0\wedge\widetilde{T}^K_{\mathfrak{a}}=\widetilde{T}^K_0\leq T_0^{X^\eps}$. Since $T_0^{X^\eps}$ is a finite random variable which does not depend on $K$, this proves \eqref{prop_inv_debut} and ends the proof in the case $\mathfrak{f}_{21}<0$.

 We suppose now that $\mathfrak{f}_{21}>0$. We first show \eqref{prop_inv_intermediaire} and then we give an argument to strengthen this limit in $\eps \to 0$ to an equality for $\eps>0$ small enough. By \eqref{def_eps0}, $\mathfrak{f}_{-\eps_0}>0$. Proposition \ref{prop_Kersting_aleatoire} gives a positive probability of invasion for $X^{v}$:
    $$p_{v}=\P(\lim X^{v}=\infty)>0\quad\text{for }v\in(-\eps_0,\eps_0).$$
By Lemma \ref{lem_continuity_p}, the function $v\mapsto p_{v}$ is continuous at $0$: the parameter of the Bernoulli random variable obtained in the limit is $p=p_0$. Evaluating \eqref{agresti} at $v=0$ with $\mu^{0}_s=\mu_2+\beta_{21}Z^{(1,a_1)}_s$ give $\mathcal{M}^{0}_j=\mathcal{M}_j$ and $b^{0}_j=b_j$ thus proving \eqref{formule_p}. Finally $p<1$, since the type-$2$ individual present at time $0$ dies before reproducing with positive probability.

Applying Proposition \ref{prop_Kersting_aleatoire} to $X^{\pm\eps}$ and using the sandwiching \eqref{prop_Xulbound}, the continuity at $0$ of $v\mapsto p_{v}$ and of $v\mapsto\mathfrak{f}_{v}$ now yield \eqref{prop_inv_intermediaire}, provided we check that the time needed by $\Gamma^{(2,K)}$ to go from
$\lceil\eps K/2\rceil$ to a value larger than $\lceil\eps K\rceil$ just after a catastrophe (at time $T_k$ for some $k\in\N$),
%as well as the time elapsed between the last catastrophe and this event, are
is negligible with respect to $\ln K$. This is the purpose of the next paragraph. Applying \eqref{prop_Xulbound} with $x_0=\lceil \eps K/2\rceil$ together with the strong Markov property at time
$\widetilde{T}^K_{\eps/2}$, we know that
$$ \lim_{K \to \infty}\P\big(\forall t<\widetilde{T}_{\eps}^K\wedge\ln^{\tilde\alpha}K,X^{-\eps}_t\leq \Gamma^{(2,K)}_t\big|X^{-\eps}_0 = \Gamma^{(2,K)}_0 = \lceil \eps K/2\rceil \big)=1.
$$
Besides, by definition of $X^{-\eps}$ in \eqref{def_Xeps} and Theorem 2 in \cite{Kersting19}, we know that for $k \in \N$
$$\Big( X^{-\eps}_{T_{k}}e^{-\int_0^{T_k}(a_2-(\beta_{22}+\beta_{21})\eps - \beta_{21}Z_s^{(1,a_1)})ds - \sum_{i=1}^k\ln(\Theta_i^{(2)})},k \in \N \Big) $$
is a uniformly integrable martingale and thus converges almost surely. Moreover, from Proposition \ref{ergofaible}(1), 
$$ \int_0^{T_k}(a_2-(\beta_{22}+\beta_{21})\eps - \beta_{21}Z_s^{(1,a_1)})ds + \sum_{i=1}^k\ln(\Theta_i^{(2)}) \sim \mathfrak{f}_{-\eps}T_k$$
for large $k$. This ensures that if $X^{-\eps}_0 = \lceil \eps K/2\rceil$, with high probability for large $K$
$$  X^{-\eps}_{T_{\sqrt{\ln K}}} \geq X_0^{-\eps} e^{\mathfrak{f}_{-\eps}T_{\sqrt{\ln K}}/2} \geq \eps K. $$
Hence, the time interval between the first time when $\Gamma^{(2,K)}$ reaches $\lceil\eps K/2\rceil$ and the first $T_k$ such that $\Gamma^{(2,K)}_{T_k}\geq \eps K$ is negligible with respect to $\ln K$, since it is dominated by $\sqrt{\ln K}$. 
This ends the proof of \eqref{prop_inv_intermediaire}.

It only remains to show that \eqref{prop_inv_intermediaire} implies \eqref{prop_inv_final}.
Proposition \ref{ergofaible} applied to the process $Z$ with initial condition $Z^{(1)}_0>0$ and $Z_0^{(2)}=\eps$ entails that almost surely,
    $$\lim_{t \to \infty}\frac{1}{t}\int_0^tZ^{(2)}_sds=\overline{z}_2.$$
In particular $Z^{(2)}$ reaches any $z<\overline{z}_2$ in finite time almost surely.
Combining this with Lemma \ref{lem_Kurtz} applied to the processes $Z^{(K)}$ and $Z$ with $Z^{(1,K)}_0=Z_0^{(1)}$ and $Z_0^{(2,K)}=\eps$ ends the proof.

We can now end the proof considering the case where $\Gamma$ is defined with $\mu_2=0$: take a sequence $(\mu_n, n\geq0)$ such that $\lim_{n\to\infty}\mu_n=0$ and the associated sequence of processes $\Gamma_n$, we see from the continuity of the formula \eqref{formule_p} that for fixed $K$ $\lim_n\P((\widetilde{T}^K_0)_n<(\widetilde{T}^K_\eps)_n)=\P(\widetilde{T}^K_0<\widetilde{T}^K_\eps)$, and using finally \eqref{prop_inv_final} gives the result by monotone convergence.
\end{proof}

\subsubsection{Proofs of Corollary \ref{coro_invasion} and Theorem \ref{theorem2}}

\begin{proof}[Proof of Corollary \ref{coro_invasion}]
    Since $0$ is absorbing for $\Gamma^{(2,K)}$, which moreover increases by jumps of size one, for every $y\in(0,\overline{z}_2)$ we have
    $$\big\{\widetilde{T}^K_y<\infty\big\}=\big\{\widetilde{T}^K_0>\widetilde{T}^K_y\big\}
    \qquad\text{and}\qquad
    \big\{\widetilde{T}^K_y\leq t_K\big\}=\Big\{\sup_{s\leq t_K}Z^{(2,K)}_s\geq y\Big\} .$$
    If $\mathfrak{f}_{21}<0$, according to \eqref{prop_inv_debut}, $\P\big(\widetilde{T}^K_y<\infty\big)\xrightarrow[K\to\infty]{}0$ for every $y>0$, hence so do $\P(E^K_1)$, $\P(E^K_2)$ and $\P(E^K_1\triangle E^K_2)$; and $p(C)=0$ in that case.

    Assume now $\mathfrak{f}_{21}>0$ and fix $y\in(0,\overline{z}_2)$. By     \eqref{prop_inv_final}, $\left(\widetilde{T}^K_0\wedge\widetilde{T}^K_y\right)(\ln K)^{-1}$ converges in total variation to $B\mathfrak{f}_{21}^{-1}$ with $B\sim \mathcal{B}(p(C))$, jointly with $\ind_{\{\widetilde{T}^K_0>\widetilde{T}^K_y\}}$ converging to $B$; since $\liminf_Kt_K/\ln K>\mathfrak{f}_{21}^{-1}$, this yields both
    $$\lim_{K\to\infty}\P\big(\widetilde{T}^K_y<\infty\big)=p(C)
    \qquad\text{and}\qquad
    \lim_{K\to\infty}\P\big(\widetilde{T}^K_y\leq t_K\big)=p(C) .$$
    Writing $y_-:=z\wedge z'$ and $y_+:=z\vee z'$, and using that $y\mapsto\widetilde{T}^K_y$ is non-decreasing, we get
    $$\big\{\widetilde{T}^K_{y_+}\leq t_K\big\}\subseteq E^K_1\cap E^K_2
    \subseteq E^K_1\cup E^K_2\subseteq\big\{\widetilde{T}^K_{y_-}<\infty\big\} .$$
    The probabilities of the two extreme events both converge to $p(C)$, hence $\P\big(E^K_1\triangle E^K_2\big)\to0$ and $\P(E^K_1)\to p(C)$. Finally $p(C)>0$ when $\mathfrak{f}_{21}>0$, by Proposition \ref{possible_invasion}.
\end{proof}

\begin{proof}[Proof of Theorem \ref{theorem2}]
    If $\mathfrak{f}_{21}<0$, for fixed $t>0$, Lemma \ref{lem_Kurtz} gives the first part of \eqref{th2_sub}; since $Z_0=(C,0)$, the process $Z$ is monomorphic: \eqref{eq:cesaro} concludes this case.

    Suppose now that $\mathfrak{f}_{21}>0$ and let $z\in(0,\overline{z}_2)$. We will apply Lemma \ref{lem_Kurtz} to $(Z^{(K)}_{\widetilde{T}^K_z+t},t\geq0)$, and then conclude with Proposition \ref{ergofaible}. To use Lemma \ref{lem_Kurtz}, according to Remark \ref{rem_Kurtz_random}, we first prove that $\big(\ind_{\{\widetilde{T}^{K}_z<\infty\}}Z^{(1,K)}_{\widetilde{T}^{K}_z},K\geq1\big)$ is tight in $\R^*_+$. Let $\eps>0$ be such that $\overline{a}_1-\beta_{12}\eps>0$. From a coupling argument combined with Proposition \ref{prop_cvglongue}, $Z^{(1,K)}_{\widetilde{T}^{K}_\eps}\geq\widecheck{Z}^{(1)}_{\widetilde{T}^{K}_\eps}$ with $\widecheck{Z}^{(1)}$ the strong solution to \eqref{eq:rival_lower} with $j=1$, $\widecheck{Z}^{(1)}_0=z^{(1)}_0$ and $\eta=\eps$. Moreover, $Z^{(1,K)}$ is dominated by the rescaled birth-death process without the competition exerted by type-$2$ individuals, whose supremum over $\R_+$ is tight, thus the family $\big(\ind_{\{\widetilde{T}^{K}_\eps<\infty\}}Z^{(1,K)}_{\widetilde{T}^{K}_\eps},K\geq1\big)$ is tight in $\R_+^*$.

    We now transfer this property from time $\widetilde{T}^{K}_\eps$ to time $\widetilde{T}^{K}_z$. By the strong Markov property at $\widetilde{T}^{K}_\eps$ and Remark \ref{rem_Kurtz_random}, for every $S<\infty$,
    $$\sup_{s\leq S}\big\Vert Z^{(K)}_{\widetilde{T}^{K}_\eps+s}-\mathcal{Z}^{K}_s\big\Vert
    \underset{K\to\infty}{\longrightarrow}0\quad\text{in probability,}$$
    where $\mathcal{Z}^{K}$ denotes the solution to \eqref{eqdep} started at $Z^{(K)}_{\widetilde{T}^{K}_\eps}\in(\R_+^*)^2$. By Proposition \ref{ergofaible} applied to $\mathcal{Z}^K$, the level $z$ is reached by $\mathcal{Z}^{K,(2)}$ in an almost surely finite time whose law is tight in $K$, since the initial conditions are tight in a compact subset of $(\R_+^*)^2$. Consequently the family $\big(\ind_{\{\widetilde{T}^{K}_z<\infty\}}(\widetilde{T}^{K}_z-\widetilde{T}^{K}_\eps), K\geq1\big)$ is tight, and so is $\big(\ind_{\{\widetilde{T}^{K}_z<\infty\}}Z^{(1,K)}_{\widetilde{T}^{K}_z},K\geq1\big)$ in $\R^*_+$.

    Applying the strong Markov property at $\widetilde{T}^{K}_z$ together with Remark \ref{rem_Kurtz_random} for every fixed $t>0$, we get
    $$\frac{1}{t}\int_0^tZ^{(K)}_{\widetilde{T}^{K}_z+s}ds -\frac{1}{t}\int_0^t\mathcal{Z}'^{K}_sds\underset{K\to\infty}{\longrightarrow}0
    \qquad\text{in probability,}$$
    where $\mathcal{Z}'^{K}$ is the solution to \eqref{eqdep} driven by the same Poisson measures and started at $Z^{(K)}_{\widetilde{T}^{K}_z}=\big(Z^{(1,K)}_{\widetilde{T}^{K}_z},z\big)$. Fix $\eta,\gamma >0$. By the tightness obtained above, there are $0<\varrho<M<\infty$ such that $\limsup_K\P^{K}_z\big(Z^{(1,K)}_{\widetilde{T}^{K}_z}\notin[\varrho,M]\big)<\gamma$. Write $\zeta\preceq\zeta'$ when $\zeta^{(1)}\leq\zeta'^{(1)}$ and $\zeta^{(2)}\geq\zeta'^{(2)}$, and let $Z^{\zeta}$ be the solution to \eqref{eqdep} started at $\zeta$ and driven by the
    same measures. The drift of \eqref{eqdep} is quasi-monotone for $\preceq$ and the jumps are multiplicative, hence order preserving, so a pathwise comparison as in Lemma \ref{lem_monotone_X}$(i)$ gives $Z^{(\varrho,z)}_s\preceq\mathcal{Z}'^{K}_s\preceq Z^{(M,z)}_s$ for every $s\geq0$, on the event $\{Z^{(1,K)}_{\widetilde{T}^{K}_z}\in[\varrho,M]\}$. Both bounds have the same Ces\`aro limit $\mathfrak{z}$ almost surely by Proposition \ref{ergofaible}, whose conclusion does not depend on the initial condition in $(\R_+^*)^2$, so that
    $$\limsup_{K\to\infty}\P^{K}_z\Big(\Big\Vert\frac1t\int_0^tZ^{(K)}_{\widetilde{T}^{K}_z+s}ds-\mathfrak{z}\Big\Vert_2>\eta\Big)\leq\gamma
    +\sum_{\zeta\in\{(\varrho,z),(M,z)\}}\P\Big(\Big\Vert\frac1t\int_0^tZ^{\zeta}_sds -\mathfrak{z}\Big\Vert_2>\eta\Big).$$
    Letting $t$ go to infinity and then $\gamma$ to $0$ yields \eqref{th2_super}.
\end{proof}
\section*{Acknowledgements}
The authors thank S. Billiard and C. Leuridan for fruitful discussions. This work was partially funded by the Chair "Modélisation Mathématique et Biodiversité" of VEOLIA-Ecole Polytechnique-MNHN-F.X. and by the French national research agency (ANR) via project ANR ARBOREAL (ANR-23-CE45-0007).

\begin{appendix}

\section{Some results for the deterministic competitive Lotka-Volterra}
\label{secLVdet}

For $a_1,a_2,\beta_{11},\beta_{12},\beta_{21},\beta_{22}>0$ and any initial condition $(\zuz,\zdz)\in\R_+^2$, System \eqref{def_LV_det} admits a unique solution $(\mathbf z_t,t\in\R)$, which stays in $\R_+^2$ (see Section 3 of Chapter 12 of \cite{hirsch1974differential}); we write $\varphi_t(\mathbf z_0):=\mathbf z_t$ for the associated flow.
Recall the definition of the invasion fitnesses in \eqref{fitinv} (the deterministic equivalents of the $\mathfrak{f}_{ij}$'s introduced in \eqref{def_new_fitnesses} for the random system). The following result states that the signs of $f_{12}$ and $f_{21}$ determine the different asymptotic behaviors for the solutions to \eqref{def_LV_det}. 

We list below the different situations that may occur:
\begin{align}
& f_{12}<0 \text{ and } f_{21}<0, \tag{excl}\label{excldet}\\
& f_{ij}\geq 0 \text{ and } f_{ji}\leq 0, \tag{excl$i$}\label{exclidet}\\
%& f_{12}f_{21}\leq 0 \text{ and } f_{12}<f_{21}, \tag{excl2}\label{excl2det}\\
& f_{12}>0 \text{ and } f_{21}>0. \tag{coex}\label{coexdet}
\end{align}

\begin{prop}
    \label{compLVdet}
Let $(\mathbf z_t,t\in\R)=(\zu,\zd,t\in\R)$ be the solution of \eqref{def_LV_det} with $\mathbf z_0\in\mathscr{A}^*$. Assume for the sake of simplicity that $f_{12}$ and $f_{21}$ are nonzero, and set
\begin{equation} \label{xyeq}
x_{eq}=\frac{a_1\beta_{22}-a_2\beta_{12}}{\beta_{11}\beta_{22}-\beta_{12}\beta_{21}},
\qquad
y_{eq}=\frac{a_2\beta_{11}-a_1\beta_{21}}{\beta_{11}\beta_{22}-\beta_{12}\beta_{21}}.
\end{equation}
If a non-trivial equilibrium is a fixed point of $\varphi$ in $(\R_+^*)^2$, exactly one of the following holds.
\begin{enumerate}[leftmargin=*,itemsep=1pt,topsep=2pt]
\item \emph{No non-trivial equilibrium}: if $f_{12}f_{21}<0$, the system $M_1(x,y)=M_2(x,y)=0$ has no solution with $x,y>0$: if $f_{12}>f_{21}$, $\mathbf z_t$ converges to $(a_1/\beta_{11},0)$, else, it converges to $(0,a_2/\beta_{22})$.
\item \emph{One stable equilibrium}: if $f_{12},f_{21}>0$ (\ref{coexdet}), then $x_{eq},y_{eq}>0$ and $\mathbf z_t\to(x_{eq},y_{eq})$.
\item \emph{One unstable equilibrium}: if $f_{12},f_{21}<0$ (\ref{excldet}), then $x_{eq},y_{eq}>0$ and, depending on $\mathbf z_0$, $\mathbf z_t$ converges to $(a_1/\beta_{11},0)$ or to $(0,a_2/\beta_{22})$ (both stable), or to $(x_{eq},y_{eq})$ (unstable).
\end{enumerate}
\end{prop}

We refer the reader to \cite{Murray1993} for more details.

\section{A technical result on Lévy processes} \label{annexe_Levy}

\begin{lemma}
\label{momentinf}
    Let $\mathscr K$ be a spectrally negative Lévy process satisfying $\lim_{t \to \infty}\mathscr K_t=+\infty$ and $\underline{\mathscr{K}}_\infty$ its infimum over $\R_+$. Then $|\underline{\mathscr{K}}_\infty| < \infty$ almost surely.
    If additionally $ \E[ \mathscr K^2_1 ]<\infty$, then $\E[|\underline{\mathscr{K}}_\infty|]<\infty.$
\end{lemma}

\begin{proof}
    Let $\psi(\lambda)=\ln\E[e^{\lambda \mathscr{K}_1}]$ denote the Laplace exponent of $\mathscr K$. As $\mathscr{K}$ drifts to infinity and has a finite second moment, we have $\psi'(0+)=\E[\mathscr{K}_1]>0$ and $\psi''(0+)=\text{Var}[\mathscr{K}_1]<\infty$. Now we recall the equality between (8.15) and (8.16) in \cite{kyprianou2006introductory}, for $\lambda \geq 0$:
    $$ \E[e^{\lambda \underline{\mathscr{K}}_\infty}] = \frac{\psi'(0+)\lambda}{\psi(\lambda)}. $$
    Differentiating this identity and using second order Taylor expansion give
    \begin{align*}
        \E[\underline{\mathscr{K}}_\infty]&= \lim_{\lambda \to 0} \frac{\partial \E[e^{\lambda \underline{\mathscr{K}}_\infty}]}{\partial \lambda}
        = \psi'(0+)\lim_{\lambda \to 0}\frac{\psi(\lambda)-\lambda \psi'(\lambda)}{\psi^2(\lambda)}\\
        &= - \psi'(0+) \frac{\psi''(0+)}{2(\psi'(0+))^2} = -\frac{\text{Var}[\mathscr{K}_1]}{2\E[\mathscr{K}_1]}.
    \end{align*}
    This ends the proof.
\end{proof}

\end{appendix}

\bibliographystyle{abbrv}
\bibliography{bib}

\end{document}